\documentclass[leqno,11pt]{article}
\usepackage[utf8]{inputenc}
\usepackage[T1]{fontenc}
\usepackage{microtype}

\usepackage[letterpaper]{geometry}

\ifdefined\screenview
  \edef\mtht{\the\textheight}
  \edef\mtwd{\the\textwidth}
  \usepackage{xcolor}
  \definecolor{BackgroundColor}{RGB}{253, 246, 227}
  \pagecolor{BackgroundColor}
\fi

\usepackage{amsmath}
\usepackage{amsthm}
\usepackage{tikz-cd}
\usetikzlibrary{arrows} 
\tikzset{
  commutative diagrams/.cd, 
  arrow style=tikz, 
  diagrams={>=stealth}
}
\usetikzlibrary{matrix,decorations.pathreplacing,calc}

\usepackage{textcomp}

\usepackage[sb]{libertine}
\usepackage[varqu,varl]{zi4}
\usepackage[libertine,bigdelims,vvarbb]{newtxmath} 
\usepackage[supstfm=libertinesups,%
  supscaled=1.2,%
  raised=-.13em]{superiors}
\useosf 
\usepackage[scr=boondox,cal=euler]{mathalfa}

\usepackage{slashed}
\usepackage{esint} 
\usepackage[english]{babel}

\usepackage{imakeidx}
\makeindex[intoc]

\usepackage{csquotes}
\usepackage[
  backend=biber,
  hyperref=true,
  backref=true,
  isbn=false,
  doi=true,
  natbib=true,
  eprint=true,
  useprefix=true,
  maxcitenames=99,
  maxbibnames=99,  
  maxalphanames=99, 
  minalphanames=99,
  safeinputenc,
  style=alphabetic,
  citestyle=alphabetic,
  block=space,
  datamodel=preamble/ext-eprint
]{biblatex}
\usepackage[
  hypertexnames = false,
  colorlinks    = true,
  citecolor     = gray,
  linkcolor     = gray,
  urlcolor      = gray,
  breaklinks
]{hyperref}
\makeatletter
\DeclareFieldFormat{arxiv}{%
  arXiv\addcolon\space
  \ifhyperref
    {\href{http://arxiv.org/\abx@arxivpath/#1}{%
       \nolinkurl{#1}%
       \iffieldundef{arxivclass}
         {}
         {\addspace\texttt{\mkbibbrackets{\thefield{arxivclass}}}}}}
    {\nolinkurl{#1}
     \iffieldundef{arxivclass}
       {}
       {\addspace\texttt{\mkbibbrackets{\thefield{arxivclass}}}}}}
\makeatother
\DeclareFieldFormat{mr}{%
  MR\addcolon\space
  \ifhyperref
    {\href{http://www.ams.org/mathscinet-getitem?mr=MR#1}{\nolinkurl{#1}}}
    {\nolinkurl{#1}}}
\DeclareFieldFormat{zbl}{%
  Zbl\addcolon\space
  \ifhyperref
    {\href{http://zbmath.org/?q=an:#1}{\nolinkurl{#1}}}
    {\nolinkurl{#1}}}
\renewbibmacro*{eprint}{%
  \printfield{arxiv}%
  \newunit\newblock
  \printfield{mr}%
  \newunit\newblock
  \printfield{zbl}%
  \newunit\newblock
  \iffieldundef{eprinttype}
    {\printfield{eprint}}
    {\printfield[eprint:\strfield{eprinttype}]{eprint}}
}
\AtEveryBibitem{%
  \clearlist{publisher}%
}
\AtEveryBibitem{%
  \clearlist{address}%
}
\DeclareFieldFormat[article,inproceedings,inbook,incollection,thesis]{title}{\textit{#1}}
\renewbibmacro{in:}{}
\newcommand{\printreferences}{\printbibliography[heading=bibintoc]}

\usepackage[inline,shortlabels]{enumitem}

\usepackage{subcaption}

\usepackage[yyyymmdd]{datetime}

\usepackage{etoolbox}
\ifundef{\abstract}{}{\patchcmd{\abstract}%
    {\quotation}{\quotation\noindent\ignorespaces}{}{}}

\usepackage[super]{nth}

\usepackage{aliascnt}

\numberwithin{equation}{section}

\renewcommand{\eqref}[1]{\hyperref[#1]{\rm(\ref*{#1})}}

\def\makeautorefname#1#2{\AtBeginDocument{\expandafter\def\csname#1autorefname\endcsname{#2}}}

\newcommand{\mynewtheorem}[2]{
  \newaliascnt{#1}{equation}          
  \newtheorem{#1}[#1]{#2}
  \aliascntresetthe{#1}
  \makeautorefname{#1}{#2}
}

\mynewtheorem{axiom}{Axiom}
\newtheorem*{axiom*}{Axiom}
\mynewtheorem{theorem}{Theorem}
\newtheorem*{theorem*}{Theorem}
\mynewtheorem{prop}{Proposition}
\newtheorem*{prop*}{Proposition}

\mynewtheorem{cor}{Corollary}
\mynewtheorem{construction}{Construction}
\mynewtheorem{lemma}{Lemma}
\mynewtheorem{conjecture}{Conjecture}
\newtheorem*{conjecture*}{Conjecture}
\mynewtheorem{hyp}{Hypothesis}

\numberwithin{substep}{step}
\makeautorefname{step}{Step}
\makeautorefname{substep}{Step}

\numberwithin{subcase}{case}
\makeautorefname{case}{Case}
\makeautorefname{subcase}{case}

\usepackage{etoolbox}
\AtBeginEnvironment{proof}{\setcounter{step}{0}}

\theoremstyle{remark}
\mynewtheorem{remark}{Remark}
\newtheorem*{remark*}{Remark}
\mynewtheorem{convention}{Convention}
\newtheorem*{convention*}{Convention}
\newtheorem*{conventions*}{Conventions}

\theoremstyle{remark}
\mynewtheorem{warning}{Warning}

\theoremstyle{definition}
\mynewtheorem{definition}{Definition}
\newtheorem*{definition*}{Definition}
\mynewtheorem{notation}{Notation}
\mynewtheorem{data}{Data}
\mynewtheorem{example}{Example}
\newtheorem*{example*}{Example}
\mynewtheorem{exercise}{Exercise}
\mynewtheorem{solution}{Solution}
\mynewtheorem{question}{Question}
\mynewtheorem{problem}{Problem}
\newtheorem*{question*}{Question}

\mynewtheorem{summary}{Summary}

\makeautorefname{table}{Table}        
\makeautorefname{chapter}{Chapter}
\makeautorefname{section}{Section}
\makeautorefname{subsection}{Section}
\makeautorefname{subsubsection}{Section}
\makeautorefname{footnote}{Footnote}
\AtBeginDocument{\def\itemautorefname~#1\null{(#1)\null}}
\AtBeginDocument{\def\equationautorefname~#1\null{(#1)\null}}

\let\C\undefined
\let\U\undefined

\usepackage{bm}
\usepackage{mathtools} 
\usepackage{stmaryrd} 

\DeclareFontFamily{U}{mathx}{\hyphenchar\font45}
\DeclareFontShape{U}{mathx}{m}{n}{
      <5> <6> <7> <8> <9> <10>
      <10.95> <12> <14.4> <17.28> <20.74> <24.88>
      mathx10
      }{}
\DeclareSymbolFont{mathx}{U}{mathx}{m}{n}
\DeclareFontSubstitution{U}{mathx}{m}{n}
\DeclareMathAccent{\widecheck}{0}{mathx}{"71}
\DeclareMathAccent{\wideparen}{0}{mathx}{"75}

\DeclareMathOperator{\Ad}{Ad}

\DeclareMathOperator{\Ann}{Ann}
\DeclareMathOperator{\Aut}{Aut}

\DeclareMathOperator{\Bij}{Bij}

\DeclareMathOperator{\Diff}{Diff}
\DeclareMathOperator{\End}{End}

\DeclareMathOperator{\GL}{GL}

\DeclareMathOperator{\HF}{\HF}
\DeclareMathOperator{\Hess}{Hess}

\DeclareMathOperator{\Hom}{Hom}
\DeclareMathOperator{\Imm}{Imm}

\DeclareMathOperator{\Lie}{Lie}

\DeclareMathOperator{\Sym}{Sym}

\DeclareMathOperator{\coker}{coker}

\DeclareMathOperator{\im}{im}

\DeclareMathOperator{\rk}{rk}
\DeclareMathOperator{\sign}{sign}
\DeclareMathOperator{\spec}{spec}

\DeclareMathOperator{\tr}{tr}

\DeclarePairedDelimiter\paren{\lparen}{\rparen}
\DeclarePairedDelimiter{\Abs}{\|}{\|}

\DeclarePairedDelimiter{\abs}{\lvert}{\rvert}
\DeclarePairedDelimiter{\braket}{\langle}{\rangle}

\DeclarePairedDelimiter{\set}{\lbrace}{\rbrace}
\def\({\left(}
\def\){\right)}
\def\<{\left\langle}
\def\>{\right\rangle}

\newcommand{\CMA}{{\rm CMA}}

\newcommand{\CMfrom}{\widehat{\rm{CM}}}

\newcommand{\CSD}{\text{CSD}}

\newcommand{\C}{{\mathbf{C}}}
\newcommand{\Gtwo}{G_2}

\newcommand{\HMfrom}{\widehat{\rm{HM}}}

\newcommand{\N}{{\mathbf{N}}}
\newcommand{\Oc}{{\mathbf{O}}}

\newcommand{\PSL}{\P\SL}
\newcommand{\PT}{\mathrm{PT}}

\newcommand{\Q}{\mathbf{Q}}

\newcommand{\R}{\mathbf{R}}
\newcommand{\SF}{\mathrm{SF}}
\newcommand{\SL}{\mathrm{SL}}
\newcommand{\SO}{\mathrm{SO}}
\newcommand{\SU}{\mathrm{SU}}
\newcommand{\SW}{\mathrm{SW}}

\newcommand{\Span}[1]{\braket{#1}}
\newcommand{\Spin}{\mathrm{Spin}}
\newcommand{\Sp}{\mathrm{Sp}}

\newcommand{\U}{\mathrm{U}}

\newcommand{\Z}{\mathbf{Z}}

\newcommand{\andq}{\text{and}\quad}

\newcommand{\biota}{{\bm{\iota}}}
\newcommand{\bsD}{{\bm{\sD}}}

\newcommand{\cone}{{\rm{cone}}}

\newcommand{\co}{\mskip0.5mu\colon\thinspace}
\newcommand{\csum}{\sharp}

\newcommand{\defined}[2][\key]{\def\key{#2}\textbf{#2}\index{#1}}
\newcommand{\delbar}{\bar{\del}}
\newcommand{\delfrom}{\hat{\del}}

\newcommand{\del}{\partial}
\newcommand{\ev}{\mathrm{ev}}

\newcommand{\hkred}{{/\!\! /\!\! /}}

\newcommand{\id}{\mathrm{id}}

\newcommand{\inner}[2]{\braket{#1, #2}}
\newcommand{\into}{\hookrightarrow}
\newcommand{\iso}{\cong}
\newcommand{\itref}{\eqref}

\newcommand{\nsub}{\triangleleft}
\newcommand{\ob}{\mathrm{ob}}

\newcommand{\qandq}{\quad\text{and}\quad}

\newcommand{\qand}{\quad\text{and}}

\newcommand{\qwithq}{\quad\text{with}\quad}
\newcommand{\reg}{\mathrm{reg}}

\newcommand{\si}{{\rm si}}

\newcommand{\spin}{\mathfrak{spin}}

\newcommand{\tImm}{{\widetilde\Imm}}

\newcommand{\uSW}{\underline{\SW}}
\newcommand{\vol}{\mathrm{vol}}

\renewcommand{\H}{\mathbf{H}}
\renewcommand{\Im}{\operatorname{Im}}
\renewcommand{\O}{\mathrm{O}}
\renewcommand{\P}{\mathbf{P}}
\renewcommand{\Re}{\operatorname{Re}}
\renewcommand{\det}{\operatorname{det}}

\renewcommand{\emptyset}{\varnothing}
\renewcommand{\epsilon}{\varepsilon}
\renewcommand{\setminus}{{\backslash}}
\renewcommand{\sp}{\mathfrak{sp}}

\renewcommand{\leq}{\leqslant}
\renewcommand{\geq}{\geqslant}

\makeatletter
\renewcommand*\env@matrix[1][*\c@MaxMatrixCols c]{%
  \hskip -\arraycolsep
  \let\@ifnextchar\new@ifnextchar
  \array{#1}}

\renewcommand\xleftrightarrow[2][]{%
  \ext@arrow 9999{\longleftrightarrowfill@}{#1}{#2}}
\newcommand\longleftrightarrowfill@{%
  \arrowfill@\leftarrow\relbar\rightarrow}
\makeatother

\newcommand{\rd}{{\rm d}}

\newcommand{\bp}{{\mathbf{p}}}

\newcommand{\bF}{{\mathbf{F}}}

\newcommand{\bP}{{\mathbf{P}}}

\newcommand{\bS}{{\mathbf{S}}}

\newcommand{\bX}{{\mathbf{X}}}

\newcommand{\sA}{\mathscr{A}}

\newcommand{\sC}{\mathscr{C}}
\newcommand{\sD}{\mathscr{D}}
\newcommand{\sE}{\mathscr{E}}

\newcommand{\sG}{\mathscr{G}}
\newcommand{\sH}{\mathscr{H}}
\newcommand{\sI}{\mathscr{I}}
\newcommand{\sJ}{\mathscr{J}}

\newcommand{\sL}{\mathscr{L}}

\newcommand{\sO}{\mathscr{O}}
\newcommand{\sP}{\mathscr{P}}

\newcommand{\sV}{\mathscr{V}}

\newcommand{\fg}{{\mathfrak g}}

\newcommand{\fn}{{\mathfrak n}}

\newcommand{\ft}{{\mathfrak t}}
\newcommand{\fu}{{\mathfrak u}}

\newcommand{\fw}{{\mathfrak w}}
\newcommand{\fx}{{\mathfrak x}}

\newcommand{\fA}{{\mathfrak A}}

\newcommand{\fF}{{\mathfrak F}}

\newcommand{\fL}{{\mathfrak L}}
\newcommand{\fM}{{\mathfrak M}}

\newcommand{\slD}{\slashed D}

\newcommand{\dt}{{\rd t}}

\newcommand{\bfeta}{{\bm\eta}}

\newcommand{\bxi}{{\bm\xi}}

\author{
  Aleksander Doan
  \and
  Thomas Walpuski
}
\title{  
  On counting associative submanifolds and Seiberg--Witten monopoles}

\date{2018-09-17}

\begin{document}
\maketitle

\begin{center}
  \textit{Dedicated to Simon Donaldson on the occasion of his \nth{60} birthday}%
  \\[1cm]
\end{center}

\begin{abstract}
  Building on ideas from \cite{Donaldson1998,Donaldson2009,Walpuski2013a,Haydys2017},
  we outline a proposal for constructing Floer homology groups associated with a $\Gtwo$--manifold.
  These groups are generated by associative submanifolds and solutions of the ADHM Seiberg--Witten equations.
  The construction is motivated by the analysis of various transitions which can change the number of associative submanifolds.
  We discuss the relation of our proposal to Pandharipande and Thomas' stable pair invariant of Calabi--Yau $3$--folds.
\end{abstract}

\section{Introduction}
\label{Sec_Introduction}

\citet[Section 3]{Donaldson1998} put forward the idea of constructing enumerative invariants of $\Gtwo$--manifolds by counting $\Gtwo$--instantons.
The principal difficulty in pursuing this program stems from non-compactness issues in higher-dimensional gauge theory \cite{Tian2000,Tao2004}.
In particular, $\Gtwo$--instantons can degenerate by bubbling along associative submanifolds.
\citet{Donaldson2009} realized that this phenomenon can occur along $1$--parameter families of $\Gtwo$--metrics.
Therefore, a naive count of $\Gtwo$--instantons \emph{cannot} lead to a deformation invariant of $\Gtwo$--metrics;
see also \cite{Walpuski2013a}.
\citeauthor{Donaldson2009} proposed to compensate for this phenomenon with a counter-term consisting of a weighted count of associative submanifolds.
However, they did not elaborate on how to construct a suitable coherent system of weights.
Haydys and Walpuski proposed to define such weights by counting solutions to the Seiberg--Witten equations associated with the ADHM construction of instantons on $\R^4$ \cites[paragraphs following Remark 1.7]{Haydys2014}{Haydys2017}[Introduction]{Doan2017a}[Appendix B]{Doan2017c}.

The construction of these weights depends on the choice of the structure group of $\Gtwo$--instantons, an obvious choice being $\SU(r)$.
If one specializes to $r=1$, that is, to trivial line bundles,
then there are no non-trivial $\Gtwo$--instantons and their naive count is, trivially, an invariant.
However, according to the Haydys--Walpuski proposal one should still count associatives weighted by the count of solutions to the Seiberg--Witten equation on them.
It is known that counting associatives by themselves does not lead to an invariant,
because the following situations may arise along a $1$--parameter family of $\Gtwo$--metrics:
\begin{enumerate}
\item
  \label{It_Intersecting}
  An embedded associative submanifold develops a self-intersection.
  Out of this self-intersection a new associative submanifold is created, as shown by \citet{Nordstrom2013}.
  Topologically, this submanifold is a connected sum.
\item
  \label{It_T2Singularities}
  By analogy with special Lagrangians in Calabi--Yau $3$--folds \cite[Section 3]{Joyce2002},
  it has been conjectured that it is possible for three distinct associative submanifolds to degenerate into a singular associative submanifold with an isolated singularity modeled on the cone over $T^2$ \cites[p.154]{Walpuski2013}[Conjecture 5.3]{Joyce2016}.
  Topologically, these three submanifolds form a surgery triad.
\end{enumerate}
We will argue that known vanishing results and surgery formulae for the Seiberg--Witten invariants of $3$--manifolds \cite[Proposition 4.1 and Theorem 5.3]{Meng1996},
show that the count of associatives weighted by solutions to the Seiberg--Witten equation is invariant under transitions \itref{It_Intersecting} and \itref{It_T2Singularities},
assuming that all connected components of the associative submanifolds in question have $b_1 > 1$.
This restriction is needed in order to be able to avoid reducible solutions and obtain a well-defined Seiberg--Witten invariant \emph{as an integer}.%
\footnote{%
  Using spectral counter-terms, \citet{Chen1997,Chen1998,Lim2000} were able to define Seiberg--Witten invariants of $3$--manifolds with $b_1 \leq 1$.
  These, however, are rational and cannot satisfy the necessary vanishing theorem.
}
We know of no natural assumption that would ensure that this restriction holds for all relevant associative submanifolds.
Hence, the Haydys--Walpuski proposal cannot yield an invariant which is just an integer.

One can define a topological invariant using the Seiberg--Witten equation for any compact, oriented $3$--manifold.
This invariant, however, is not a number but rather a homology group, called monopole Floer homology \cite{Marcolli2001,Manolescu2003,Kronheimer2007,Froyshov2010}.
The behavior of monopole Floer homology under connected sum and in surgery triads is well-understood \cites[Theorem 2.4]{Kronheimer2007a}{Bloom20XX}[Theorem 5]{Lin2015}.
We will explain how to construct a chain complex associated with a $\Gtwo$--manifold using the monopole chain complexes of associative submanifolds.
The  homology of this chain complex \emph{might be invariant} under transitions \itref{It_Intersecting} and \itref{It_T2Singularities}.

The discussion so far only involved the classical Seiberg--Witten equation.
There is a further transition that might spoil the invariance of the proposed homology group:
\begin{enumerate}[resume]
\item
  \label{It_MultipleCovers}
  Along generic $1$--parameter families of $\Gtwo$--metrics,
  somewhere injective immersed associative submanifolds can degenerate by converging to a multiple cover.
\end{enumerate}
We will explain why this phenomenon occurs and that it can change the number of associatives,
even when weighted by counts of solutions to the Seiberg--Witten equation.
This is where \emph{ADHM monopoles}, solutions to the Seiberg--Witten equations related to the ADHM construction, enter the picture.
Counting ADHM monopoles does not lead to a topological invariant of $3$--manifolds.
We will provide evidence for the conjecture that the change in the count of ADHM monopoles exactly compensates the change in the number of associatives weighted by the Seiberg--Witten invariant.
Based on this we will give a tentative proposal for how to construct an invariant of $\Gtwo$--manifolds: a homology group generated by associatives and ADHM monopoles.

This paper is organized as follows.
After reviewing in \autoref{Sec_CountingAssociatives} the basics of $\Gtwo$--geometry, we discuss in \autoref{Sec_ProblemsWithCountingAssociatives} and \autoref{Sec_MultipleCoverAssociatives} the three problems with counting associatives described above. 
The core of the paper are: \autoref{Sec_DegenerationsOfADHMMonopoles} where we introduce ADHM monopoles and relate them to multiple covers of associatives, and \autoref{Sec_TentativeProposal} where we outline a construction of a Floer homology group associated with a $\Gtwo$--manifold.
In \autoref{Sec_ADHMBundles} we argue that a dimensional reduction of our proposal should lead to a symplectic analogue of Pandariphande and Thomas' stable pair invariant known in algebraic geometry  \cite{Pandharipande2009}. 
\autoref{Sec_ProofOfSomewhereInjectiveTransversality} contains the proof of a transversality theorem for somewhere injective associative immersions. \autoref{Sec_SeibergWitten} and \autoref{Sec_HaydysCorrespondence} develop a general theory of the Haydys correspondence with stabilizers for Seiberg--Witten equations associated with quaternionic representations. \autoref{Sec_NakajimasProof} summarizes the linear algebra of the ADHM representation.

Finally, we would like to point out that an alternative approach to counting associative submanifolds has been proposed recently by \citet{Joyce2016}.
His proposal does not lead to a number or a homology group, but rather a more complicated object: a super-potential up to quasi-identity morphisms.


\paragraph{Acknowledgements}
We are grateful to Simon Donaldson for his generosity, kindness, and optimism which have inspired and motivated us over the years.
We thank Tomasz Mrowka for pointing out \cite{Bloom20XX} and advocating the idea of incorporating the monopole chain complex, Oscar García--Prada for a helpful discussion on vortex equations, and Richard Thomas for answering our questions about stable pairs.

This material is based upon work supported by  \href{https://www.nsf.gov/awardsearch/showAward?AWD_ID=1754967&HistoricalAwards=false}{the National Science Foundation under Grant No.~1754967}
and
\href{https://sites.duke.edu/scshgap/}{the Simons Collaboration Grant on ``Special Holonomy in Geometry, Analysis and Physics''}.

\section{Counting associative submanifolds}
\label{Sec_CountingAssociatives}

We begin with a review of $\Gtwo$--manifolds and associative submanifolds with a focus towards explaining what we mean by ``counting associative submanifolds''.

\subsection{\texorpdfstring{$\Gtwo$}{G2}--manifolds}

The exceptional Lie group $\Gtwo$ is the automorphism group of the octonions $\Oc$, the unique $8$--dimensional normed division algebra:
\begin{equation*}
  \Gtwo = \Aut(\Oc).
\end{equation*}
Since any automorphism of $\Oc$ preserves the unit $1 \in \Oc$ and its $7$--dimensional orthogonal complement $\Im\Oc \subset \Oc$, we can think of $\Gtwo$ as a subgroup of $\SO(7)$.

\begin{definition}
  A \defined{$\Gtwo$--structure} on a $7$--dimensional manifold $Y$ is a reduction of the structure group of the frame bundle of $Y$ from $\GL(7)$ to $\Gtwo$.
  An \defined{almost $\Gtwo$--manifold} is a  $7$--dimensional manifold $Y$ equipped with a $\Gtwo$--structure.
\end{definition}

The multiplication on $\Oc$ endows $\Im\Oc$ with:
\begin{itemize}
\item 
  an \defined{inner product},
  $g \co S^2\Im\Oc \to \R$
  satisfying
  \begin{equation*}
    g(u,v) = -\Re(uv),
  \end{equation*}
\item
  a \defined{cross-product}
  $\cdot\times\cdot\co \Lambda^2 \Im\Oc \to \Im\Oc$
  defined by
  \begin{equation*}
    (u,v) \mapsto u\times v \coloneq \Im(uv)
  \end{equation*}
  and a corresponding $3$--form $\phi \in \Lambda^3 \Im\Oc^*$ defined by
  \begin{equation*}
    \phi(u,v,w) \coloneq g(u\times v,w),
  \end{equation*}
  as well as
\item
  an \defined{associator}
  $[\cdot,\cdot,\cdot] \co \Lambda^3 \Im\Oc \to \Im\Oc$
  defined by
  \begin{equation}
    \label{Eq_Associator}
    [u,v,w] \coloneq (u\times v)\times w + \inner{v}{w}u - \inner{u}{w}v 
  \end{equation}
  and a corresponding $4$--form $\psi \in \Lambda^4\Im\Oc^*$ defined by
  \begin{equation*}
    \psi(u,v,w,z) \coloneq g([u,v,w],z).
  \end{equation*}
\end{itemize}
These are related by the identities
\begin{equation}
  \label{Eq_PhiPsiG}
  \begin{split}
    i(u)\phi\wedge i(v)\phi \wedge \phi &= 6g(u,v) \vol_g \qand \\
    *_g\phi &= \psi
  \end{split}
\end{equation}
for a unique choice of an orientation on $\Im\Oc$.
We refer the reader to \cites[Chapter IV]{Harvey1982} {Salamon2010} for a more detailed discussion.

A $\Gtwo$--structure on $Y$ endows $TY$ with analogous structures: 
\begin{itemize}
\item
  a Riemannian metric $g$,
\item
  a cross-product $\cdot\times\cdot\co \Lambda^2 TY \to TY$,
\item
  a $3$--form $\phi \in \Omega^3(Y)$, 
\item
  an associator $[\cdot,\cdot,\cdot] \co \Lambda^3 TY \to TY$, and
\item
  a $4$--form $\psi \in \Omega^4(Y)$,
\end{itemize}
satisfying the same relations as above.
From \eqref{Eq_PhiPsiG} it is apparent that from $\phi$ one can reconstruct $g$ and thus also $\psi$, the cross-product, and the associator.
Similarly, one can reconstruct $g$ from $\psi$ together with the orientation.
The condition for a $3$--form $\phi$ or a $4$--form $\psi$ to arise from a $\Gtwo$--structure is that the form be definite;
see \cites[Section 8.3]{Hitchin2001}[Section 2.8]{Bryant2006}.
We say that a $3$--form $\phi$ is \defined{definite} if the bilinear form $G_\phi \in \Gamma(S^2T^*Y\otimes \Lambda^7T^*Y)$ defined by
\begin{equation*}
  G_\phi(u,v) \coloneq i(u)\phi\wedge i(v)\phi\wedge \phi
\end{equation*}
is definite.
We say that a $4$--form $\psi$ is \defined{definite} if the bilinear form $G_\psi^* \in \Gamma\(S^2TY\otimes (\Lambda^7T^*Y)^{\otimes 2}\)$ defined by
\begin{equation*}
  G_\psi^*(\alpha,\beta) \coloneq i(\alpha)\psi\wedge i(\beta)\psi\wedge \psi
\end{equation*}
is definite.
Here we identify $\Lambda^4 T^*Y \iso \Lambda^3 TY\otimes \Lambda^7T^*Y$.
Therefore, a $\Gtwo$--structure can be specified either by a definite $3$--form $\phi$,
or by a definite $4$--form $\psi$ together with an orientation.

A $\Gtwo$--structure on a $7$--manifold induces a spin structure through the inclusion $\Gtwo \subset \Spin(7)$.
In fact, a $7$--manifold admits a $\Gtwo$--structure if and only if it is spin, see \cite[Theorems 3.1 and 3.2]{Gray1969} and \cite[p. 321]{Lawson1989}.
This means that the existence of a $\Gtwo$--structure is a soft, topological condition.
More rigid notions are obtained by imposing conditions on the torsion of the $\Gtwo$--structure, in the sense of $G$--structures, see \cite[Section 2.6]{Joyce2000}.
The most stringent and most interesting condition to impose is that the torsion vanishes.

\begin{definition}
  A \defined{$\Gtwo$--manifold} is a $7$--manifold equipped with a torsion-free $\Gtwo$--structure.
\end{definition}

\begin{theorem}[{\citet[Theorem 5.2]{Fernandez1982}}]
  A $\Gtwo$--structure on a $7$--manifold $Y$ is torsion-free if and only the associated $3$--form $\phi$ as well as the associated $4$--form $\psi$ are closed:
  \begin{equation*}
    \rd\phi = 0 \qandq
    \rd\psi = 0.
  \end{equation*}
\end{theorem}

The Riemannian metric induced by a torsion-free $\Gtwo$--structure has holonomy contained in $\Gtwo$---%
one of two exceptional holonomy groups in Berger's classification \cite[Theorem 3]{Berger1955}.
If $Y$ is compact, then equality holds if and only if $\pi_1(Y)$ is finite \cite[Proposition 10.2.2]{Joyce2000}.
We refer the reader to \cite[Section 10]{Joyce2000} for a thorough discussion of the properties of $\Gtwo$--manifolds.

\begin{example}
  \label{Ex_S1xCY3}
  If $Z$ is a Calabi--Yau $3$--fold with a Kähler form $\omega$ and a holomorphic volume form $\Omega$, and if $t$ denotes the coordinate on $S^1$,
  then $S^1 \times Z$ is a $\Gtwo$--manifold with
  \begin{equation*}
    \phi = \rd t \wedge \omega + \Re\Omega \qandq
    \psi = \frac12\omega\wedge\omega + \rd t\wedge\Im\Omega.
  \end{equation*}
  In this case the holonomy group is contained in $\SU(3) \subset \Gtwo$.
\end{example}

\begin{example}
  The first local, complete, and compact examples of manifolds with holonomy equal to $\Gtwo$ are due to \citet{Bryant1987,Bryant1989,Joyce1996,Joyce1996a,Joyce2000} respectively.
  Joyce's examples arise from a generalized Kummer construction based on smoothing flat $\Gtwo$--orbifolds of the form $T^7\!\!/\Gamma$ where $\Gamma$ is a finite group of isometries of the $7$--torus.
  This method has recently been extended to more general $\Gtwo$--orbifolds by \citet{Joyce2017}.
  The most fruitful construction method for $\Gtwo$--manifolds to this day is the twisted connected sum construction,
  which was pioneered by \citet{Kovalev2003} and improved by    \citet{Kovalev2011,Corti2012,Corti2012a}.
  It is based on gluing, in a twisted fashion, a pair of asymptotically cylindrical $\Gtwo$--manifolds which are products of $S^1$ with asymptotically cylindrical Calabi--Yau $3$--folds.
  Using this construction, 
  \citet{Corti2012a} produced tens of millions of examples of compact $\Gtwo$--manifolds.
\end{example}

\subsection{Associative submanifolds}

\begin{definition}
  Let $Y$ be an almost $\Gtwo$--manifold,
  let $P$ be an oriented $3$--manifold, and
  let $\iota\co P \to Y$ be an immersion.
  We say that $\iota$ is \defined{associative} if
  \begin{equation}
    \label{Eq_Associative}
    \iota^*{[\cdot,\cdot,\cdot]} = 0 \in \Omega^3(P,\iota^*TY)
    \qandq
    \iota^*\phi \text{ is positive}.
  \end{equation}
  
  An \defined{immersed associative submanifold} is an equivalence class $[\iota]$ of an associative immersion $\iota \in \Imm(P,Y)/\Diff_+(P)$ for some oriented $3$--manifold $P$.
  Here $\Imm(P,Y)$ is the space of immersions $P \to Y$ and $\Diff_+(P)$ is the group of orientation-preserving diffeomorphisms of $P$.
\end{definition}

\citet[Chapter IV, Theorem 1.6]{Harvey1982} proved the identity
\begin{equation}
  \label{Eq_AssociatorIdentity}
  \phi(u,v,w)^2 + \abs{[u,v,w]}^2 = \abs{u\wedge v\wedge w}^2.
\end{equation}
This shows that $\phi$ is a semi-calibration and that associative submanifolds are calibrated by $\phi$.
We refer to \cite[Introduction]{Harvey1982} and \cite[Section 3.7]{Joyce2000} for an introduction to calibrated geometry; we recall only the following simple but fundamental  fact.

\begin{prop}
  If $\iota\co P \to Y$ is associative, then
  \begin{equation*}
    \iota^*\phi = \vol_{\iota^*g}.
  \end{equation*}
  In particular, if $\phi$ is closed and $P$ is compact, then the immersed submanifold $\iota(P)$ is volume-minimizing in the homology class $\iota_*[P]$ and
  \begin{equation*}
    \vol(P,\iota^*g) = \inner{[\phi]}{\iota_*[P]}.
  \end{equation*}
\end{prop}

\begin{prop}[{see, e.g., \cite[Lemma 4.7]{Salamon2010}}]
  \label{Prop_AssociativeViaCrossProduct}
  If $\iota\co P \to Y$ is an immersion,
  then the following are equivalent:
  \begin{enumerate}
  \item
    $\iota^*[\cdot,\cdot,\cdot] = 0$,
  \item
    for all $u,v \in \iota_*T_xP$,
    $u \times v \in \iota_*TP$, and
  \item
    for all $u \in \iota_*T_xP$ and $v \in (\iota_*T_xP)^\perp$,
    $u \times v \in (\iota_*T_xP)^\perp$.
  \end{enumerate}
\end{prop}

\begin{example}
  Let $Z$ be a Calabi--Yau $3$--fold.
  Equip $S^1\times Z$ with the $\Gtwo$--structure from \autoref{Ex_S1xCY3}.
  If $\Sigma \subset Z$ is a holomorphic curve,
  then $S^1\times\Sigma$ is associative.
  If $L \subset Z$ is a special Lagrangian submanifold,
  then, for any $t \in S^1$, $\set{t}\times L$ is associative.
\end{example}

\begin{example}
  Examples of associative submanifolds which arise as fixed points of involutions have been given by \citet[Section 4.2]{Joyce1996a}.
  Examples of associative submanifolds arising from holomorphic curves and special Lagrangians in asymptotically cylindrical Calabi--Yau $3$--folds were constructed by \citet[Section 5]{Corti2012a}
\end{example}

\subsection{The \texorpdfstring{$\fL$}{L} functional}
\label{Sec_LFunctional}

Associative submanifolds can be formally thought of as critical points of a functional $\fL$ on the infinite-dimensional space of submanifolds.
In contrast to many other functionals studied in differential geometry (for example, the Dirichlet functional), the Hessian of $\fL$ at a critical point is not positive definite.
As we will see, it is a first order elliptic operator whose spectrum is discrete and unbounded in both positive and negative directions.
Morse theory of functionals with this property, most notably the Chern--Simons functional in gauge theory, was first developed by Floer \cite{Floer1988a, Donaldson2002}.
The existence of such $\fL$ already hints at the possibility of constructing \defined{Floer homology groups} from a chain complex formally generated by associative submanifolds.

\begin{definition}
  Define the $1$--form $\delta\fL = \delta\fL_\psi \in \Omega^1(\Imm(P,Y))$ by%
  \footnote{%
    Although $n$ is not a vector field on $Y$, by slight abuse of notation we denote by $\iota^*i(n)\psi$ the $3$--form on $P$ given by $(u,v,w)\mapsto\psi(\iota_*u,\iota_*v,\iota_*w,n)$.%
  }
  \begin{equation*}
    \delta_\iota\fL(n)
    = \int_P \iota^*i(n)\psi
    = \int_P \inner{\iota^*[\cdot,\cdot,\cdot]}{n}
  \end{equation*}
  for $n \in T_\iota\Imm(P,Y) = \Gamma(P,\iota^*TY)$.
\end{definition}

\begin{prop}
  \label{Prop_LFunctional}
  ~
  \begin{enumerate}
  \item
    \label{Prop_LFunctional_Associative}
    $\iota$ is associative if and only if $\delta_\iota\fL = 0$ and $\iota^*\phi$ is positive.
  \item
    \label{Prop_LFunctional_DiffInvariant}
    $\delta\fL$ is $\Diff_+(P)$--invariant.
  \item
    \label{Prop_LFunctional_Closed}
    If $\rd\psi = 0$, then $\delta\fL$ is a closed $1$--form.
    In fact, there is a $\Diff_+(P)$--equivariant covering space $\pi\co \tImm(P,Y) \to \Imm(P,Y)$ and a $\Diff_+(P)$--equivariant function $\tilde\fL\co \tImm(P,Y) \to \R$ whose differential is  $\pi^*\delta\fL$.%
    \footnote{%
      This justifies the notation $\delta\fL$ since locally it is the differential of a function.
    }
  \end{enumerate}
\end{prop}

\begin{proof}
  Assertions \itref{Prop_LFunctional_Associative} and \itref{Prop_LFunctional_DiffInvariant} are both trivial.
  For $\beta \in H_3(Y,\R)$,
  let $\Imm_\beta(P,Y)$ denote the set of immersions $\iota\co P \to Y$ such that $\iota_*[P] = \beta$.
  Fix $P_0 \in \Imm_\beta(P,Y)$ and denote by $\tImm_\beta(P,Y)$ the space of pairs $(\iota,[Q])$ with $\iota \in \Imm_\beta(P)$ and $[Q]$ an equivalence class of $4$--chains in $Y$ such that $\del Q = P - P_0$ with $[Q]=[Q']$ if and only if $[Q-Q'] = 0 \in H_4(Y,\Z)$.
  Define $\tilde \fL \co \tImm_\beta(P,Y) \to \R$ by
  \begin{equation*}
    \tilde \fL(\iota,[Q]) = \int_Q \psi.
  \end{equation*}
  The function $\tilde \fL$ has the desired properties;
  see also \cite[Section 8]{Donaldson1998}.
\end{proof}

\subsection{The moduli space of associatives}

\begin{definition}
  Let $P$ be a compact, oriented $3$--manifold and let $\beta \in H_3(Y,\Z)$.
  Denote by $\Imm_\beta(P,Y)$ the space of immersions $\iota\co P \to Y$ with $\iota_*[P] = \beta$.
  The group $\Diff_+(P)$ acts on $\Imm_\beta(P,Y)$.
  The \defined{moduli space} of immersed associative submanifolds is
  \begin{equation*}
    \fA(\psi) = \coprod_{\beta \in H_3(Y,\Z)} \fA_\beta(\psi) = \coprod_{\beta \in H_3(Y,\Z)}\coprod_P \fA_{P,\beta}(\psi)
  \end{equation*}
  with
  \begin{equation*}
    \fA_{P,\beta}(\psi)
   \coloneq
   \set*{
      [\iota] \in \Imm_\beta(P,Y)/\Diff_+(Y)
      :
      \eqref{Eq_Associative}
    }.
  \end{equation*}
  Here $P$ ranges over all diffeomorphism types of compact, oriented $3$--manifolds.

  Denote by $\sD^4(Y)$ the space of definite $4$--forms on $Y$.
  If $\sP$ is a subspace of $\sD^4(Y)$,
  then the \defined{$\sP$--universal moduli space} is
  \begin{equation*}
    \fA(\sP) = \bigcup_{\psi \in \sP} \fA(\psi).
  \end{equation*}
\end{definition}

The moduli space $\fA(\sP)$ inherits a topology from the $C^\infty$--topology on $\Imm_{\beta}(P,Y)$.
As we will explain in the following,
the infinitesimal deformation theory of associatives submanifolds is controlled by a first-order elliptic operator
and $\fA(\sP)$ admits corresponding Kuranishi models.

\begin{definition}
  \label{Def_FueterNormal}
  Let $\iota\co P \to Y$ be an associative immersion.
  Denote by
  \begin{equation*}
    N\iota \coloneq \iota^*TY/TP \iso TP^\perp \subset \iota^*TY
  \end{equation*}
  its normal bundle and by $\nabla$ the connection on $N\iota$ induced by the Levi-Civita connection.
  The \defined{Fueter operator} associated with $\iota$ is the first order differential operator $F_\iota = F_{\iota,\psi} \co \Gamma(N\iota) \to \Gamma(N\iota)$ defined by
  \begin{equation*}
    F_\iota(m) \coloneq \sum_{i=1}^3 \iota_*e_i\times\nabla_{e_i}m.
  \end{equation*}
  Here $(e_1,e_2,e_3)$ is an orthonormal frame of $P$.
\end{definition}

This operator arises as follows.
Identify $N\iota$ with $TP^\perp \subset \iota^*TY$ and,
given a normal vector field $m \in \Gamma(N\iota)$,
define $\iota_m \co P \to Y$ by
\begin{equation*}
  \iota_m(x) \coloneq \exp(m(x)).
\end{equation*}
The condition for $\iota_{\epsilon m}$ to be associative to first order in $\epsilon$ is that
\begin{align*}
  0
  &=
    \left.\frac{\rd}{\rd \epsilon}\right|_{\epsilon=0}
    [(\iota_{\epsilon m})_*e_1,(\iota_{\epsilon m})_*e_2,(\iota_{\epsilon m})_*e_3] \\
  &=
    (\iota_*e_1 \times \iota_*e_2) \times \nabla_{e_3}m 
  + \text{cyclic permutations} \\
  &=
    \sum_{i=1}^3 \iota_*e_i \times \nabla_{e_i}m.
\end{align*}
Here we have used the definition of the associator \eqref{Eq_Associator} and the fact that $\iota\co P \to Y$ is associative so we have $\iota_*e_1\times\iota_*e_2 = \iota_*e_3$ (as well as all of its cyclic permutations).

\begin{prop}[{\citet[paragraph after Theorem 2.12]{Joyce2016}}]
  If $\rd\psi = 0$, then
  \begin{equation*}
    \Hess \tilde \fL (n,m) = \int_P \inner{n}{F_\iota m}
  \end{equation*}
  with $\tilde\fL$ as in \autoref{Prop_LFunctional}\itref{Prop_LFunctional_Closed}.
  In particular,
  $F_\iota$ is self-adjoint.
\end{prop}

\begin{theorem}[{\citet{McLean1998} and \citet[Theorem 2.12]{Joyce2016}}]
  \label{Thm_AssociativeKuranishiModel}
  Let $[\iota\co P\to Y] \in \fA_\beta(\psi_0)$.
  Denote by $\Aut(\iota)$ the stabilizer of $\iota$ in $\Diff_+(P)$.

  The group $\Aut(\iota)$ is finite.
  The Fueter operator $F_\iota$ is equivariant with respect to the action of $\Aut(\iota)$ on $\Gamma(N\iota)$.
  If $\sP$ is a submanifold of the space of definite $4$--forms containing $\psi_0$, then there are:
  \begin{itemize}
  \item
    an $\Aut(\iota)$--invariant open subset $U \subset \sP\times \ker F_\iota$,
  \item
    a smooth $\Aut(\iota)$--equivariant map $\ob \co \sP\times U \to \coker F_\iota$ with $\ob(\psi_0,\cdot)$ and its derivative vanishing at $0$,
  \item
    an open neighborhood $V$ of $([\iota],\psi_0)$ in $\fA_\beta(\sP)$, and
  \item
    a homeomorphism $\fx\co \ob^{-1}(0)/\Aut(\iota) \to V$.
  \end{itemize}
  Moreover, if $(\bp,n) \in \ob^{-1}(0)$, then the stabilizer of any immersion representing $\fx(\bp,n)$ is the stabilizer of $n$ in $\Aut(\iota)$.
\end{theorem}

\begin{definition}
  We say that an associative immersion $\iota \co P \to Y$ is \defined{unobstructed} (or \defined{rigid}) if $F_\iota$ is invertible.
\end{definition}

\subsection{Transversality}

It follows from \autoref{Thm_AssociativeKuranishiModel} that if all associative immersions are rigid, then the moduli space $\fA_\beta(\psi)$ is a collection of isolated points---in other words, the functional $\fL$ is a Morse function.
While this is not always true, below we show that it does hold for a large class of immersions and for a generic choice of a closed positive $4$--form $\psi$.

\begin{definition}
  An immersion $\iota \co P \to Y$ is called \defined{somewhere injective} if each connected component of $P$ contains a point $x$ such that $\iota^{-1}(\iota(x)) = \set{x}$.
  Denote by
  \begin{equation*}
    \fA_\beta^\si(\psi)
  \end{equation*}
  the open subset of somewhere injective immersions with respect to $\psi$.
  Given a submanifold $\sP$ of the space of definite $4$--forms,
  set
  \begin{equation*}
    \fA_\beta^\si(\sP) = \bigcup_{\psi \in \sP} \fA_\beta^\si(\psi).
  \end{equation*}
\end{definition}

\begin{prop}
  \label{Prop_SomewhereInjectiveTransversality}
  Denote by $\sD_c^4(Y)$ the space of closed, definite $4$--forms.
  \begin{enumerate}
  \item
    There is a residual subset $\sD^4_{c,\reg} \subset \sD^4_c(Y)$
    such that for every $\psi \in \sD^4_{c,\reg}$
    \begin{enumerate}
    \item
      \label{Prop_SomewhereInjectiveTransversality_11}
      the moduli space $\fA^\si_{\beta}(\psi)$ is a $0$--dimensional manifold and consists only of unobstructed associative submanifolds, and
    \item
      \label{Prop_SomewhereInjectiveTransversality_12}
      $\fA^\si_{\beta}(\psi)$ consists only of embedded associative submanifolds.
    \end{enumerate}

  \item
    If $\psi_0,\psi_1 \in \sD^4_{c,\reg}(Y)$, then there is a residual subset $\bsD^4_{c,\reg}(\psi_0,\psi_1)$ in the space of paths from $\psi_0$ to $\psi_1$ in $\sD^4_c(Y)$
    such that for every $(\psi_t)_{t\in[0,1]} \in \bsD^4_{c,\reg}(\psi_0,\psi_1)$
    \begin{enumerate}
    \item
      \label{Prop_SomewhereInjectiveTransversality_21}
      the universal moduli space $\fA^\si_{\beta}\(\set{\psi_t : t \in [0,1]}\)$ is a $1$--dimensional manifold, and
    \item
      \label{Prop_SomewhereInjectiveTransversality_22}
      for each component $\set{ (\psi_t,[\iota_t]) : t \in J }$ with $J \subset [0,1]$ an interval,
      there is a discrete set $J_\times \subset J$ such that:
      \begin{enumerate}
      \item
        for $t \in J\setminus J_\times$ the map $\iota_t$ is an embedding and
      \item
        for $t_\times \in J_\times$ there is a $T > 0$ and with the property that
                \begin{equation*}
          \bP \coloneq \bigcup_{\abs{t-t_\times} < T} \set{t}\times \iota_t(P) \subset \R\times Y
        \end{equation*}
        has a unique self-intersection and this intersection is transverse.
      \end{enumerate}
    \end{enumerate}
  \end{enumerate}
\end{prop}

The proof of this result is deferred to \autoref{Sec_ProofOfSomewhereInjectiveTransversality}.
It is similar to that of analogous results about pseudo-holomorphic curves in symplectic manifolds, cf.~\citet[Sections 3.2 and 3.4]{McDuff2012}.
In fact, our situation is simpler because we assume from the outset that $\iota$ is an immersion.

\subsection{Compactness and tamed forms}

As we have seen, transversality for associative embeddings can be achieved by perturbing $\psi$. 
However, even if the moduli space $\fA_\beta(\psi)$ consists of isolated points, the number of points can be infinite. 
Indeed, for an arbitrary definite $4$--form $\psi$ there is no reason to expect $\fA_\beta(\psi)$  to be compact.
The situation is better when one considers a special class of tamed $4$--forms.
This is analogous to the notion of a tamed almost complex structure in symplectic topology, which guarantees area bounds for pseudo-holomorphic curves.

\begin{definition}[{\citet[Section 3.2]{Donaldson2009}, \citet[Definition 2.6]{Joyce2016}}]
  Let $Y$ be an almost $\Gtwo$--manifold with $3$--form $\phi$, $4$--form $\psi$, and associator $[\cdot,\cdot,\cdot]$.
  We say that $\tau \in \Omega^3(Y)$ \defined{tames} $\psi$ if $\rd\tau = 0$ and for all $x \in Y$ and $u,v,w \in T_xY$ with $[u,v,w] = 0$ and $\phi(u,v,w) > 0$, we have $\tau(u,v,w) > 0$.
\end{definition}

\begin{example}
  If $\psi$ corresponds to a torsion-free $\Gtwo$--structure,
then $\psi$, as well as any nearby $4$--form, is tamed by $\phi = *\psi$.
\end{example}

One should think of tamed, closed, definite $4$--forms as a softening of the notion of a definite $4$--form giving rise to a torsion-free $\Gtwo$--structure.
The advantage of working with tamed forms is that the volume of any associative submanifold in $\fA_\beta(\psi)$ is bounded and one can, in principle, use geometric measure theory to compactify $\fA_\beta(\psi)$.

\begin{prop}[{\citet[Section 3.2]{Donaldson2009}, \citet[Section 2.5]{Joyce2016}}]
  Let $Y$ be a compact almost $\Gtwo$--manifold with $4$--form $\psi$.
  If $\psi$ is tamed by a closed $3$--form $\tau$, then there is a constant $c > 0$ such that for every associative immersion $\iota\co P \to Y$ with $P$ compact
  \begin{equation*}
    \vol(P,\iota^*g) \leq c\cdot\inner{[\tau]}{\iota_*[P]}.
  \end{equation*}
\end{prop}

\subsection{Enumerative invariants from associatives?}
\label{Sec_EnumerativeInvariants?}

\begin{question}
  Is there a residual subset of tamed, closed, definite $4$--forms for which $\fA_\beta(\psi)$ is a compact $0$--dimensional manifold (or orbifold)?
\end{question}

If the answer to this question is yes, then for every $\psi$ from this residual subset we can define
\begin{equation}
  \label{Eq_NDPhi_1stAttempt}
  n_\beta(\psi) \coloneq \#\fA_\beta(\psi).
\end{equation}

\begin{question}
  \label{Q_NDPsiInvariant}
  Is $n_\beta(\psi)$, or some modification of it, invariant under deforming $\psi$?
\end{question}

If the answer to this question is also yes, then $n_\beta$ would give rise to a deformation invariant of $\Gtwo$--manifolds by defining its value on a torsion-free $\Gtwo$--structure $\psi$ to be that on a nearby tamed, closed, definite $4$--form.

It is easy to see that a naive interpretation of $\#\fA_\beta(\psi)$ as the cardinality of $\fA_\beta(\psi)$ does not lead to an invariant.
Suppose that $\sP = \set{ \psi_t : t \in (-1,1) }$ is $1$--parameter family of tamed, closed, definite $4$--form and $[\iota_0 \co P \to Y] \in \fA_\beta(\psi_0)$ with $\dim \ker F_{\iota_0,\psi_0} = 1$.
By \autoref{Thm_AssociativeKuranishiModel}, a neighborhood of $([\iota_0],\psi_0) \in \fA_\beta(\sP)$ is given by $\ob^{-1}(0)$ with $\ob$ a smooth map satisfying
\begin{equation*}
  \ob(t,\delta) = \lambda t + c \delta^2 + \text{higher order terms}.
\end{equation*}
For a generic $1$--parameter family we will have $\lambda, c \neq 0$.
For simplicity, let us assume that $\lambda = c = 1$.
In this situation for $-1 \ll t < 0$, there are two associative submanifolds $[\iota^\pm_t\co P \to Y]$ with respect to $\psi_t$ near $[\iota_0]$.
As $t$ tends to $0$, $[\iota^\pm_t]$ tends to $[\iota_0]$.
For $t > 0$ there are no associatives near $[\iota_0]$.
This means that $n_\beta(\phi)$ as defined in \eqref{Eq_NDPhi_1stAttempt} changes by $-2$ as $t$ passes through $0$.

The origin of this problem is that $\fA_\beta(\psi)$ should be an oriented manifold and we should count associative immersions $[\iota] \in \fA_\beta(\psi)$ with signs $\epsilon([\iota],\psi) \in \set{\pm 1}$.
These signs should be such that if $\set{ \iota_t \co P \to Y : t \in [0,1] }$ is a $1$--parameter family of associative immersions along a $1$--parameter family of closed, definite $4$--forms, then
\begin{equation}
  \label{Eq_EpsilonSpectralFlow}
  \epsilon([\iota_1],\psi_1) = (-1)^{\SF\(F_{\iota_t,\psi_t} : t \in [0,1]\)}\cdot\epsilon([\iota_0],\psi_0).
\end{equation}
In the above situation we have
\begin{equation*}
  \epsilon([\iota^+_t],\psi_t) = -\epsilon([\iota^-_t],\psi_t).
\end{equation*}
Therefore, $n_\beta(\psi)$ will be be invariant as $t$ passes though $0$ if we interpret $\#$ as as signed count, that is,
\begin{equation}
  \label{Eq_NDPhi_2ndAttempt}
  n_\beta(\psi) \coloneq \sum_{[\iota] \in \fA_\beta(\psi)} \epsilon([\iota],\psi)
\end{equation}
with some choice of $\epsilon$ satisfying \eqref{Eq_EpsilonSpectralFlow}.
An almost canonical method for determining $\epsilon$ was recently discovered by \citet[Section 3]{Joyce2016}.
We refer the reader to Joyce's article for a careful and detailed discussion.

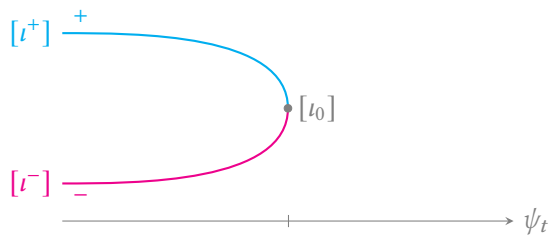
\begin{figure}[h]
  \centering
  \begin{tikzpicture}
    \draw[thick,cyan] (-3,2.5) node [left] {$[\iota^+]$} node [above right] {\footnotesize $+$} .. controls (-2,2.5) and (0,2.5) .. (0,1.5);
    \draw[thick,magenta] (-3,.5) node [left] {$[\iota^-]$} node [below right] {\footnotesize $-$} .. controls (-2,.5) and (0,.5) .. (0,1.5);
    \filldraw[gray] (0,1.5) node [right] {$[\iota_0]$} circle (0.05);
    \draw[gray] (-3,0) -- (0,0);
    \draw[|-stealth,gray] (0,0) -- (3,0) node [right] {$\psi_t$};
  \end{tikzpicture}
  \caption{Two associatives submanifold with opposite signs annihilating in an obstructed associative submanifold.}
  \label{Fig_TwoAssociativesWithOppositeSigns}
\end{figure}


\section{Intersections, \texorpdfstring{$T^2$}{T2}--singularities, and the Seiberg--Witten invariant}
\label{Sec_ProblemsWithCountingAssociatives}

In what follows we describe in more detail transitions \itref{It_Intersecting} and \itref{It_T2Singularities} from \autoref{Sec_Introduction},
and explain
why they spoil the deformation invariance of $n_\beta(\psi)$.
We then argue that the Seiberg--Witten equation on $3$--manifolds might play a role in repairing the deformation invariance.
There is, however, a price to pay: one has to give up on defining a numerical invariant and instead work with more complicated algebraic objects: chain complexes and homology groups.

\subsection{Intersecting associative submanifolds}
\label{Sec_IntersectingAssociatives}

Let $(\psi_t)_{t \in (-T,T)}$ be a $1$--parameter family of closed, tamed, definite $4$--forms on $Y$ and let $(\iota_t \co P \to Y)_{t \in (-T,T) }$ be a $1$--parameter family of somewhere injective unobstructed associative immersions.
By \autoref{Prop_SomewhereInjectiveTransversality},
if $(\psi_t)$ is generic,
then we can assume that $\iota_t$ is an embedding for all $t \neq 0$ and $\iota_0$ has a unique self-intersection as in \autoref{Prop_SomewhereInjectiveTransversality}\itref{Prop_SomewhereInjectiveTransversality_22}.
This intersection is locally modeled on the intersection of two transverse associative subspaces of $\R^7$.
Given any pair of transverse associative subspaces of $\R^7$,
there is a smooth associative submanifold asymptotic to these subspaces at infinity, called the Lawlor neck.
Nordström proved that out of the unique self-intersection of $\iota_0$ a new $1$--parameter family of associative submanifolds is created in $Y$ by gluing in a Lawlor neck.

\begin{theorem}[{\citet{Nordstrom2013}}]
  \label{Thm_DesingularizingIntersectionsOfAssocatives}
  Let $Y$ be a compact $7$--manifold and let $(\psi_t)_{t\in(-T,T)}$ be a family of closed, definite $4$--forms on $Y$.
  Let $P$ be a compact, oriented $3$--manifold.
  Suppose that $(\iota_t \co P \to Y)_{t \in (-T,T) }$ is a $1$--parameter family of unobstructed associative immersions such that
  \begin{equation*}
    \bP \coloneq \bigcup_{t\in(-T,T)} \set{t}\times \iota_t(P) \subset \R\times Y
  \end{equation*}
  has a unique self-intersection which occurs for $t=0$ and is transverse.
  Let $x^\pm$ denote the preimages in $P$ of the intersection in $Y$ and denote by $P^\csum$ the connected sum of $P$ at $x^+$ and $x^-$.

  There is a constant $\epsilon_0 > 0$,
  a continuous function $t \co [0,\epsilon_0] \to (-T,T)$, and
  a $1$--parameter family of immersions $(\iota^\csum_\epsilon \co P^\csum \to Y)_{\epsilon \in (0,\epsilon_0]}$
  such that, for each $\epsilon \in (0,\epsilon_0]$,
  $\iota^\csum_\epsilon$ is an unobstructed associative immersion with respect to $\psi_{t(\epsilon)}$.
  Moreover, as $\epsilon$ tends to zero the images of $\iota_\epsilon^\csum$ converge to the image of $\iota_0$ in the sense of integral currents.
\end{theorem}

\begin{remark}
  \label{Rmk_DesingularizingIntersectionsOfAssocatives}
  The paper \cite{Nordstrom2013} has not yet been made available to a wider audience. 
  A part of what goes into proving \autoref{Thm_DesingularizingIntersectionsOfAssocatives} can be found in \cite[Section 4.2]{Joyce2016}.
  There it is also argued that for a generic choice of $(\psi_t)_{t\in(-T,T)}$ the function $t$ is expected to be of the form $t(\epsilon) = \delta\epsilon + O(\epsilon^2)$ with a non-zero coefficient $\delta$ whose geometric meaning is also explained therein.
\end{remark}

\begin{remark}
  Denote by $P_1,\ldots, P_n$ the connected components of $P$.
  Let $j_\pm$ be such that $x_\pm \in P_{j_\pm}$.
  We have
  \begin{equation*}
    P^\csum
    \iso
    \begin{cases}
      \coprod_{j \neq j_\pm} P_j \sqcup (P_{j_+}\csum P_{j_-})
      & \text{for } j_+ \neq j_- \qand \\
      \coprod_{j \neq j_+} P_j \sqcup (P_{j_+}\csum S^1\times S^2)
      & \text{for } j_+ = j_-.
    \end{cases}
  \end{equation*}
\end{remark}

\begin{figure}[h]
  \centering
  \begin{tikzpicture}
    \draw[thick,cyan] (-3,1.5) -- (3,1.5) node[right] {$[\iota_t]$} node[above left] {\footnotesize  $\pm$};
    \draw[thick,magenta] (0,1.5) .. controls (0,.5) and (2,.5) .. (3,.5) node[right] {$[\iota^\csum_\epsilon]$} node[below left] {\footnotesize  $\pm$};
    \filldraw[gray] (0,1.5) circle (0.05);
    \draw[gray] (-3,0) -- (0,0);
    \draw[|-stealth,gray] (0,0) -- (3,0) node [right] {$\psi_t$};
  \end{tikzpicture}  
  \caption{An associative being born out of an intersection another associative.}
  \label{Fig_IntersectingAssociatives}
\end{figure}
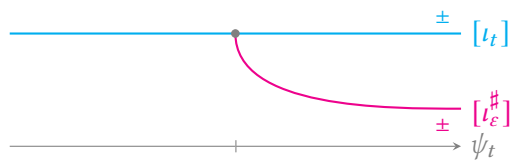

In the situation described in \autoref{Thm_DesingularizingIntersectionsOfAssocatives} and depicted in \autoref{Fig_IntersectingAssociatives},
$n_\beta(\psi_t)$ as defined in \eqref{Eq_NDPhi_2ndAttempt} changes by $\pm 1$ as $t$ crosses $0$.
In particular,
$n_\beta$ is not invariant.

\subsection{Associative submanifolds with \texorpdfstring{$T^2$}{T2}--singularities}
\label{Sec_AssociatviesWithT2Singularities}

Suppose that $\hat P$ is an associative submanifold in $(Y,\psi_0)$ with a point singularity at $x \in \hat P$ modelled on the following cone over $T^2$:
\begin{align*}
  \hat L
  &=
    \set*{
    (0,z_1,z_2,z_3) \in \R\oplus\C^3
    :
    \abs{z_1}^2=\abs{z_2}^2=\abs{z_3}^2,
    z_1z_2z_3 \in [0,\infty) \in \C
    } \\
  &=
    \set*{
    r\cdot (0,e^{i\theta_1},e^{i\theta_2},e^{-i\theta_1-i\theta_2})
    :
    r \in [0,\infty), \theta_1,\theta_2 \in S^1
    }.
\end{align*}
For a more formal discussion we refer the reader to \citet[Section 5.2]{Joyce2016}.
There, in particular, it is argued by analogy with the case of special Lagrangians that such singular associatives should be described by a Fredholm theory of index $-1$.
That is: we should expect them not to exist for a generic choice of $\psi$ but to appear along generic $1$--parameter families $(\psi_t)$.

The singularity model $\hat L$ can be resolved in $3$ ways:
\begin{align*}
  L^1_\lambda
  &=
    \set*{
    (0,z_1,z_2,z_3) \in \R\oplus\C^3
    :
    \abs{z_1}^2-\lambda = \abs{z_2}^2 = \abs{z_3}^2,
    z_1z_2z_3 \in [0,\infty) \in \C
    }, \\
  L^2_\lambda
  &=
    \set*{
    (0,z_1,z_2,z_3) \in \R\oplus\C^3 :
    \abs{z_1}^2 = \abs{z_2}^2-\lambda = \abs{z_3}^2,
    z_1z_2z_3 \in [0,\infty) \in \C
    }, \qand \\
  L^3_\lambda
  &=
    \set*{
    (0,z_1,z_2,z_3) \in \R\oplus\C^3 :
    \abs{z_1}^2 = \abs{z_2}^2 = \abs{z_3}^2-\lambda,
    z_1z_2z_3 \in [0,\infty) \in \C
    }.
\end{align*}

These are asymptotic to $\hat L$ at infinity and smooth, which can be seen by 
identifying $L^i_{\lambda}$ with $S^1 \times \C$ via
\begin{equation}
  \label{Eq_AttachingMaps}
  \begin{split}
    &S^1\times \C \to L^1_{\lambda}, \quad (e^{i\theta},z) \mapsto
    \(0,e^{i\theta}\sqrt{|z|^2+\lambda}, z, e^{-i\theta}\bar z\), \\
    &S^1\times \C \to L^2_{\lambda}, \quad (e^{i\theta},z) \mapsto
    \(0,e^{-i\theta}\bar z,e^{i\theta}\sqrt{|z|^2+\lambda},z\), \qand \\
    &S^1\times \C \to L^3_{\lambda}, \quad
    (e^{i\theta},z) \mapsto
    \(0,z,e^{-i\theta} \bar z,e^{i\theta}\sqrt{|z|^2+\lambda}\).
  \end{split}
\end{equation}
Topologically, $L^i_\lambda$ can be obtained from $\hat L$ via Dehn surgery.

\begin{definition}
  Let $P^\circ$ be a $3$--manifold with $\delbar P^\circ = T^2$.
  Let $\mu$ be a simple closed curve in $T^2$.
  The \defined{Dehn filling} of $P^\circ$ along $\mu$,
  denoted by $P^\circ_\mu$,
  is the $3$--manifold obtained by attaching $S^1\times D$ to $P^\circ$ in such a way that $\set{*} \times S^1$ is identified with $\mu$.
\end{definition}

\begin{remark}
  Up to diffeomorphism, $P^\circ_\mu$ depends only on the homotopy class of $\mu \subset T^2$;
  moreover, it does not depend on the orientation of $\mu$.
\end{remark}

We can identify the boundary of $\hat L^\circ \coloneq \hat L \setminus B_1$ with $T^2$ via
\begin{equation*}
  (e^{i\theta_1},e^{i\theta_2})
  \mapsto
  \frac{1}{\sqrt{3}} \paren*{0,e^{i\theta_1},e^{i\theta_2},e^{-i\theta_1-i\theta_2}}.
\end{equation*}
Comparing the maps introduced in  \eqref{Eq_AttachingMaps} restricted to $\{*\}\times S^1$ with the above identification, we see that $L^i_\lambda$ is obtained by Dehn filling $\hat L^\circ$ along loops representing the homology classes 
\begin{equation}
  \label{Eq_DehnFillingSlopes}
  \mu_1=(0,1), \quad
  \mu_2=(-1,0), \qandq
  \mu_3=(1,-1)
\end{equation}
where $(1,0)$ and $(0,1)$ are the generators of $H_1(T^2,\Z)$ corresponding to the loops $\theta\mapsto(e^{i\theta},0)$ and $\theta\mapsto(0,e^{i\theta})$.

We expect that $\hat P$ can be resolved in three ways as well.

\begin{conjecture}[cf.~{\citet[Conjecture 5.3]{Joyce2016}}]
  \label{Conj_ResolvingT2Singularity}
  Let $(\psi_t)_{t \in (-T,T)}$ be a $1$--parameter family of closed, tamed, definite $4$--forms on $Y$.
  Let $\hat P$ be an unobstructed singular associative submanifold in $(Y,\psi_0)$ with a unique singularity at $x$ which is modeled on $\hat L$.
  Associated to this data there are constants $\delta_1,\delta_2,\gamma \in \R$.
  For a generic $1$--parameter family $(\psi_t)_{t\in(-T,T)}$,
  $\delta_1\neq 0$, $\delta_2\neq 0$, $\delta_1\neq \delta_2$ and $\gamma \neq 0$.
  If this holds, then there is $\epsilon_0>0$ and, for $i=1,2,3$, there are functions $t_i \co [0,\epsilon_0] \to (-T,T)$,
  compact, oriented $3$--manifolds $P^i$, and
  $1$--parameter families of immersions $(\iota^i_\epsilon \co P^i \to Y)_{\epsilon \in (0,\epsilon_0]}$ such that:
  \begin{enumerate}
  \item
    $\iota^i_\epsilon$ is an unobstructed associative immersion with respect to $\psi_{t_i(\epsilon)}$.
  \item
    $\iota^i_\epsilon(P^i)$ is close to $\hat P$ away from $x$ and close to $L^i_\epsilon$ near $x$.
  \item
    $P^i$ is diffeomorphic to the manifold obtained by Dehn filling $\hat P^{\circ} = \hat P \setminus B_\sigma(x)$ along $\mu_i$ where $\mu_i \in H_1(\del \hat P^{\circ}) = H_1(T^2)$ is as in \eqref{Eq_DehnFillingSlopes}.
  \item
    We have
    \begin{gather*}
      t_1(\epsilon) = -\frac{\delta_2}{\gamma} \epsilon + O(\epsilon^2), \quad
      t_2(\epsilon) = \frac{\delta_1}{\gamma} \epsilon + O(\epsilon^2), \\ 
      \andq
      t_3(\epsilon) = \frac{\delta_2-\delta_1}{\gamma} \epsilon + O(\epsilon^2).
    \end{gather*}
  \end{enumerate}
\end{conjecture}

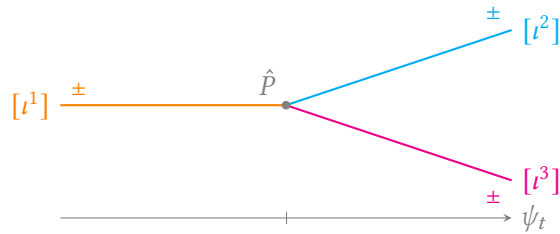
\begin{figure}[h]
  \centering
  \begin{tikzpicture}
    \draw[thick,orange] (-3,1.5) node [left] {$[\iota^1]$} node [above right] {\footnotesize $\pm$} -- (0,1.5);
    \draw[thick,cyan] (0,1.5) -- (3,2.5) node [right] {$[\iota^2]$} node [above left] {\footnotesize $\pm$};
    \draw[thick,magenta] (0,1.5) -- (3,.5) node [right] {$[\iota^3]$} node [below left] {\footnotesize $\pm$};
    \filldraw[gray] (0,1.5) node [above left] {$\hat P$} circle (0.05);
    \draw[gray] (-3,0) -- (0,0);
    \draw[|-stealth,gray] (0,0) -- (3,0) node [right] {$\psi_t$};
  \end{tikzpicture}  
  \caption{Three associatives emerging out of a singular associative for $\delta_2 > \delta_1 > 0$.}
  \label{Fig_T2Singularities}
\end{figure}

In the situation described in \autoref{Conj_ResolvingT2Singularity} and depicted in \autoref{Fig_T2Singularities},
$n_\beta(\psi_t)$ as defined in \eqref{Eq_NDPhi_2ndAttempt} changes as $t$ crosses $0$.
Again, the occurrence of the phenomenon described above would preclude $n_\beta$ from being a deformation invariant.

\subsection{The Seiberg--Witten invariant of \texorpdfstring{$3$}{3}--manifolds}
\label{Sec_RoleOfSeibergWitten}

If there were a topological invariant $w(P) \in \Z$ defined for every compact, oriented $3$--manifold and satisfying
\begin{equation}
  \label{Eq_SurgeryBehaviour}
  \begin{split}
    w(P_1 \csum P_2) &= 0 \qand \\
    \epsilon_1 w(P^{\circ}_{\mu_1})
    + \epsilon_2 w(P^{\circ}_{\mu_2})
    + \epsilon_3 w(P^{\circ}_{\mu_3}) &= 0
  \end{split}
\end{equation}
with $\mu_1,\mu_2,\mu_3$ as in \eqref{Eq_DehnFillingSlopes} and some choice of $\epsilon_1,\epsilon_2,\epsilon_3 \in \set{\pm 1}$,
then
\begin{equation}
  \label{Eq_NDPhi_3rdAttempt}
  n_\beta(\psi) \coloneq \sum_{[\iota] \in \fA_\beta(\psi)} \epsilon([\iota],\psi) w(P)
\end{equation}
would be invariant along the transition discussed in \autoref{Sec_IntersectingAssociatives} and also along the transition discussed in \autoref{Sec_AssociatviesWithT2Singularities} provided the signs work out correctly.

It is easy to see that the only such invariant defined for all $3$--manifolds is trivial since $w(P) = w(P\csum S^3) = 0$ for all oriented $3$--manifolds $P$.
However, for those $3$--manifolds $P$ for which $b_1(P_j) > 1$ for all connected components $P_j$, there are non-trivial invariants satisfying \eqref{Eq_SurgeryBehaviour}.
One example of such an invariant is the \defined{Seiberg--Witten invariant} $\SW(P)$.
We refer the reader to \cite[Section 2]{Meng1996} for a detailed discussion of the construction of $\SW(P)$.
For the moment, it shall suffice to think of $\SW(P)$ as the signed count of all gauge-equivalence classes of solutions to the Seiberg--Witten equation; that is, pairs of $(\Psi,A) \in \Gamma(W)\times \sA(\det(W))$ satisfying
\begin{equation}
  \label{Eq_ClassicalSeibergWitten}
  \begin{split}
    \slD_A \Psi &= 0 \qandq \\
    \frac12 F_A &= \mu(\Psi).
  \end{split}
\end{equation}
Here $W$ is the spinor bundle of a spin$^c$ structure $\fw$ on $P$,
$\slD_A$ is the twisted Dirac operator,
and $\mu(\Psi) = \Psi\Psi^*-\frac12\abs{\Psi}^2\,\id_W$ is identified with an imaginary-valued $2$--form using the Clifford multiplication.

\begin{remark}
  The actual definition of $\SW(P)$ involves perturbing \eqref{Eq_ClassicalSeibergWitten} by a closed $2$--form $\eta$ in order to ensure that the moduli space of solutions is cut-out transversely and contains no reducible solutions.
  The necessity to choose $\eta$ and the fact that $H^2(P,\Z)$ has codimension $b_1(P)$ in $H^2(P,\R)$, where the cohomology class of $\eta$ lies, is responsible for the restriction $b_1(P) > 1$.
\end{remark}

\begin{remark}
  $\SW(P)$ has a refinement $\uSW(P)$ defined for oriented $3$--manifolds $P$ with $b_1(P)>0$; roughly speaking, it is an integer-valued function on the set of the  isomorphism classes of spin$^c$ structures $\fw$ on $P$.
  When $b_1 > 1$, it is zero for all but finitely many $\fw$ and we can take $\SW(P)$ to be the sum of the invariants over all spin$^c$ structures.  
  We come back to this point in \autoref{Sec_MengTaubes}.
\end{remark}

\begin{theorem}[{\citet[Proposition 4.1]{Meng1996}}]
  \label{Thm_SeibergWittenUnstability}
  If $P_1,P_2$ are two compact, connected, oriented $3$--manifolds with $b_1(P_i) \geq 1$, then
  \begin{equation*}
    \SW(P_1\csum P_2) = 0.
  \end{equation*}
\end{theorem}

\begin{theorem}[{\citet[Theorem 5.3]{Meng1996}}]
  \label{Thm_SeibergWittenSurgeryFormula}
  Let $P^{\circ}$ be a compact, connected, oriented $3$--manifold with $\del P^{\circ} = T^2$.
  If $\mu_1,\mu_2,\mu_3 \in H_1(\del P^{\circ})$ are such that
  \begin{equation*}
    \mu_1\cdot \mu_2 = \mu_2\cdot \mu_3 = \mu_3\cdot \mu_1 = -1
  \end{equation*}
  (with $T^2 = \del P^{\circ}$ oriented as the boundary of $P^{\circ}$), then
  \begin{equation*}
    \epsilon_1\cdot\SW(P^{\circ}_{\mu_1})
    + \epsilon_2\cdot\SW(P^{\circ}_{\mu_2})
    + \epsilon_3\cdot\SW(P^{\circ}_{\mu_3})
    = 0
  \end{equation*}
  for suitable choices of $\epsilon_1,\epsilon_2,\epsilon_3 \in \set{\pm 1}$,
  provided $b_1(P^\circ_{\mu_i}) > 1$ for all $i=1,2,3$.
\end{theorem}

\begin{remark}
  The formulation of \cite[Theorem 5.3]{Meng1996} is in terms of $p/q$--surgery on a link $L$ which is rationally trivial in homology.
  The discussion in \cite[Section 42.1]{Kronheimer2007} explains how this is related to Dehn filling, and from this it is clear that the surgery formula given by Meng and Taubes implies the above theorem.
\end{remark}

\begin{remark}
  The Seiberg--Witten invariant is often defined only for compact, connected, oriented $3$--manifolds $P$.
  If $P$ has connected components $P_1,\ldots,P_m$, then $\SW(P) \coloneq \prod_{j=1}^m \SW(P_j)$. 
\end{remark}

Let us temporarily assume that all associative immersions $\iota\co P \to Y$ with $\iota_*[P] = \beta$ happen to be such that all connected components $P_j$ satisfy $b_1(P_j) > 1$.
If we defined $n_\beta$ by \eqref{Eq_NDPhi_3rdAttempt} with the weight $w = \SW$,
then $n_\beta$ would be invariant in the situations described in \autoref{Sec_IntersectingAssociatives} and \autoref{Sec_AssociatviesWithT2Singularities}, at least if the signs work out correctly, or modulo $2$.
Defining $n_\beta$ in this way really amounts to counting a much larger moduli space than $\fA_\beta(\psi)$,
namely:
\begin{equation*}
  \fA^\SW_\beta(\psi)
  =
  \coprod_P\coprod_{\fw} \fA^\SW_{P,\beta,\fw}(\psi)
\end{equation*}
with
\begin{equation*}
  \fA^\SW_{P,\beta,\fw}(\psi)
  \coloneq
  \frac{
    \set*{
      (\iota,\Psi,A) \in \Imm_\beta(P,Y) \times \Gamma(W) \times \sA(\det W)
      :
      \begin{array}{@{}l@{}}
        \iota \text{ satisfies \eqref{Eq_Associative}} \text{ and} \\
        (\Psi,A) \text{ satisfies \eqref{Eq_ClassicalSeibergWitten}} \\
        \text{with respect to } \iota^*g_\psi
      \end{array}
    }
  }{
    \Diff_+(P)\ltimes C^\infty(P,\U(1)).
  }
\end{equation*}
Here $\fw$ ranges over all isomorphism classes of spin$^c$ structures on $P$ and $W$ denotes the spinor bundle.
The non-invariance of $n_\beta$ as defined in \eqref{Eq_NDPhi_2ndAttempt} can be traced back to the completion of $\fA_\beta(\set{\psi_t})$ not being a $1$--manifold.
The moduli space $\fA^\SW_\beta(\set{\psi_t})$ smooths out the singularities in the completion of $\fA_\beta(\set{\psi_t})$ encountered in the situations described in \autoref{Sec_IntersectingAssociatives} and \autoref{Sec_AssociatviesWithT2Singularities};
see \autoref{Fig_SmoothingT2Singularities}.

\begin{figure}[h]
  \centering
  \begin{tikzpicture}
    \draw[thick,cyan] (-3,2) node [left] {\tiny $[\iota^1,\Psi^{1,1},A^{1,1}]$} -- (0,2) -- (3,3) node [right] {\tiny $[\iota^2,\Psi^{2,1},A^{2,1}]$};
    \draw[thick,cyan] (-3,1.75) node [left] {\tiny $[\iota^1,\Psi^{1,2},A^{1,2}]$} -- (0,1.75) -- (3,2.75) node [right] {\tiny $[\iota^2,\Psi^{2,2},A^{2,2}]$};
    \draw[thick,magenta] (-3,1.5) node [left] {\tiny  $[\iota^1,\Psi^{1,3},A^{1,3}]$} -- (0,1.5) -- (3,.5) node [right] {\tiny $[\iota^3,\Psi^3,A^3]$};
    \filldraw[gray] (0,2) circle (0.05);
    \filldraw[gray] (0,1.75) circle (0.05);
    \filldraw[gray] (0,1.5) circle (0.05);
    \draw[gray] (-3,0) -- (0,0);
    \draw[|-stealth,gray] (0,0) -- (3,0) node [right] {$\psi_t$};
  \end{tikzpicture}  
  \caption{An example of how counting with Seiberg--Witten solutions can smooth out the situation depicted in \autoref{Fig_T2Singularities}.}
  \label{Fig_SmoothingT2Singularities}
\end{figure}
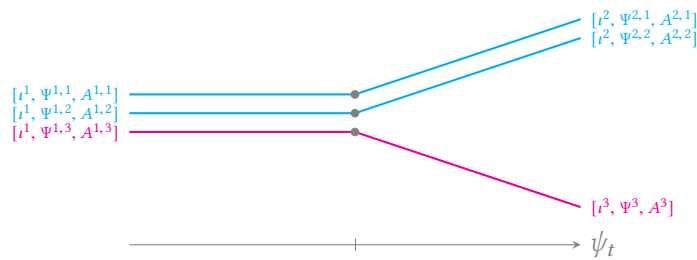

To sum up: the issue with defining a topological invariant $w(P) \in \Z$ with the properties described in \eqref{Eq_SurgeryBehaviour} means that there is indeed no invariant $n_\beta(\psi) \in \Z$ defined by a formula of the form \eqref{Eq_NDPhi_3rdAttempt}.
If it happens that for all associatives with $\iota_*[P] = \beta$ all connected components $P_j$ satisfy $b_1(P_j) > 1$, then the invariance of $n_\beta(\psi)$ can be rescued by setting $w(P) = \SW(P)$.
Unfortunately, there is no reason to believe that this holds for any reasonable class of closed, tamed, definite $4$--forms $\psi$ or choice of $\beta$.
(The situation is somewhat better for associatives arising from holomorphic curves in Calabi--Yau $3$--folds.
We discuss this case in \autoref{Sec_ADHMBundles}.)
However, Seiberg--Witten theory of $3$--manifolds suggests an alternative approach to defining an invariant of $\Gtwo$--manifolds.

\subsection{A putative Floer theory}
\label{Sec_AssociativeMonopoleHomology}

Although there is no topological invariant $w(P) \in \Z$ defined for all closed, oriented $3$--manifolds, satisfying the properties described in \eqref{Eq_SurgeryBehaviour},
there are Seiberg--Witten--Floer homology theories satisfying analogues of \eqref{Eq_SurgeryBehaviour},
see \citet{Marcolli2001,Manolescu2003,Kronheimer2007,Froyshov2010}.
We focus on one of the variants defined by \citeauthor{Kronheimer2007}.
To each closed, oriented $3$--manifold $P$ they assign a homology group
\begin{equation*}
  \HMfrom(P) = H\paren[\big]{\CMfrom(P,\clubsuit),\delfrom}.
\end{equation*}
Very roughly,
the chain complexes $\CMfrom(P,\clubsuit)$ are the $C^\infty(P,\U(1))$--equivariant Morse complexes of the \defined{Chern--Simons--Dirac functional} $\CSD\co \Gamma(W)\times\sA(\det W) \to \R$ defined by
\begin{equation}
  \label{Eq_CSD}
  \CSD(\Psi,A)
  =
  \frac12 \int_P (A-A_0)\wedge F_A
  + \int_P \<\slD_A\Psi,\Psi\> \vol
\end{equation}
on the configuration space
\begin{equation*}
  \sC(P) = \coprod_\fw \sC(P,\fw) \qwithq
  \sC(P,\fw) = \Gamma(W)\times \sA(\det W).
\end{equation*}
(The fact that $C^\infty(P,\U(1))$ does not act freely is a significant problem,
which \citeauthor{Kronheimer2007} resolve by blowing up $\sC(P)$ to a manifold with boundary and defining corresponding Morse complexes adapted to this situation.)
The chain complexes $\CMfrom(P,\clubsuit)$ depend on choices of additional data $\clubsuit$,
in particular, a Riemannian metric on $P$ as well as the choice of a suitable perturbation of the equation).
Different choices of $\clubsuit$, however, lead to quasi-isomorphic chain complexes.
We denote by $\CMfrom(P)$ quasi-isomorphism class of $\CMfrom(P,\clubsuit)$, or rather its isomorphism class in the derived category of chain complexes.
If $Q$ is a $4$--dimensional cobordism with $\del Q = P_1 - P_2$,
then \citeauthor{Kronheimer2007} define an induced chain map
\begin{equation*}
  \CMfrom(Q) \co \CMfrom(P_1) \to \CMfrom(P_2).
\end{equation*}
If $Q = [0,1]\times P$, then $\CMfrom(Q)$ is simply the differential $\delfrom$ on $\CMfrom(P)$.
The construction of $\HMfrom$ involves a choice of coefficients.
For the upcoming results to hold one needs to work with $\Z_2$ coefficients (or suitable local systems).
The monopole homology groups are then $\Z_2\llbracket U\rrbracket$--modules.
Here one should think $U$ as the same $U$ as in $H^\bullet(B\U(1)) = \Z[U]$.

The following results are the analogues of the vanishing result from \autoref{Thm_SeibergWittenUnstability} and the surgery formula from \autoref{Thm_SeibergWittenSurgeryFormula}.

\begin{theorem}[{\citet{Bloom20XX}; \citet[Theorem 5]{Lin2015}}]
  \label{Thm_MonopoleHomology0SurgeryFormula}
  Let $P^+$ and $P^-$ be two compact, connected, oriented $3$--manifolds.
  Denote by $P^+\csum P^-$ their connected sum and by $Q$ the surgery cobordism from $P^+ \sqcup P^-$ to $P^+\csum P^-$.
  Then there is an exact triangle%
  \footnote{%
    We use square brackets to denote the translation $C[p]_n = C_{p+n}$,
    see \cite[Translation 1.2.8]{Weibel1994}.
  }
  \begin{equation*}
    \CMfrom(P^+ \sqcup P^-)
    \xrightarrow{\CMfrom(Q)}
    \CMfrom(P^+\csum P^-)
    \to
    \CMfrom(P^+ \sqcup P^-)
    \to
    \CMfrom(P^+ \sqcup P^-)[-1];
  \end{equation*}
  in particular,
  \begin{equation}
    \label{Eq_MonopoleHomology0Surgeryformula}
    \HMfrom(P^+\sqcup P^-))
    \iso
    H\paren[\big]{\cone\paren[\big]{\CMfrom(P^+\sqcup P^-)
        \xrightarrow{\CMfrom(Q)} \CMfrom(P^+\csum P^-)}}.
  \end{equation}
\end{theorem} 

\begin{remark}
  In \cite[Theorem 5]{Lin2015},
  \autoref{Thm_MonopoleHomology0SurgeryFormula} is stated and proved as an isomorphism
  \begin{equation*}
    \HMfrom(P^+\csum P^-)
    \iso
    H\paren[\big]{\cone\paren[\big]{\CMfrom(P^+)\otimes\CMfrom(P^-)[1] \xrightarrow{\id\otimes U + U\otimes \id} \CMfrom(P^+)\otimes \CMfrom(P^-)}}
  \end{equation*}
  induced by the cobordism $Q$.
  This formulation is much more useful for actual computations of $\HMfrom(P^+\csum P^-)$,
  but we need \eqref{Eq_MonopoleHomology0Surgeryformula} for our purposes.
  The equivalence of these statements follows by observing that once we identify
  \begin{equation*}
    \CMfrom(P^+ \sqcup P^-) = \CMfrom(P^+)\otimes\CMfrom(P^-)
  \end{equation*}
  the map $\CMfrom(P^+ \sqcup P^-) \to \CMfrom(P^+ \sqcup P^-)[-1]$ is given by $\id\otimes U + U\otimes \id$ and rotating the above exact triangle.
\end{remark}

\begin{remark}
  More generally, if $P^\csum$ is obtained by performing a connected sum at two points $x^\pm$ in $P$ and $Q$ denotes the surgery cobordism from $P$ to $P^\csum$,
  then we expect there to be an exact triangle
  \begin{equation*}
    \CMfrom(P)
    \xrightarrow{\CMfrom(Q)}
    \CMfrom(P^\csum)
    \to
    \CMfrom(P)
    \to
    \CMfrom(P)[-1].
  \end{equation*}
  \autoref{Thm_MonopoleHomology0SurgeryFormula} asserts that this is holds if the points $x^\pm$ lie in different connected components of $P$.
\end{remark}

\begin{theorem}[{\citet[Theorem 2.4]{Kronheimer2007a}; see also \cite[Theorem 42.2.1]{Kronheimer2007}}]
  \label{Thm_MonopoleHomology1SurgeryFormula}
  Let $P^{\circ}$ be a compact, connected, oriented $3$--manifold with $\del P^{\circ} = T^2$.
  Let $\mu_1,\mu_2,\mu_3 \in H_1(\del P^{\circ})$ be such that
  \begin{equation*}
    \mu_1\cdot \mu_2 = \mu_2\cdot \mu_3 = \mu_3\cdot \mu_1 = -1
  \end{equation*}
  (with $T^2 = \del P^{\circ}$ oriented as the boundary of $P^{\circ}$.)
  Denote by $Q_{ij}$ the surgery cobordism from $P^{\circ}_{\mu_i}$ to $P^{\circ}_{\mu_j}$.
  There is an exact triangle
  \begin{equation*}
    \CMfrom(P^{\circ}_{\mu_2})
    \xrightarrow{\CMfrom(Q_{23})} \CMfrom(P^{\circ}_{\mu_3})
    \to \CMfrom(P^{\circ}_{\mu_1})
    \to \CMfrom(P^{\circ}_{\mu_2})[-1];
  \end{equation*}
  in particular,
  \begin{equation}
    \label{Eq_MonopoleHomology1Surgeryformula}
    \HMfrom(P^{\circ}_{\mu_1})
    \iso
    H\paren[\big]{\cone\paren[\big]{\CMfrom(P^{\circ}_{\mu_2})
        \xrightarrow{\CMfrom(Q_{23})} \CMfrom(P^{\circ}_{\mu_3})}}.
  \end{equation}
\end{theorem}

\begin{remark}
  While \autoref{Thm_MonopoleHomology1SurgeryFormula} holds for all three version of monopole homology defined by \citeauthor{Kronheimer2007},
  \autoref{Thm_MonopoleHomology0SurgeryFormula} only holds form $\HMfrom$;
  see \cite[paragraph after (13)]{Lin2015}.
  This is why we restricted ourselves to this version from the outset.
\end{remark}

Associative submanifolds are critical points of the functional $\fL$ defined in \autoref{Prop_LFunctional}.
Gradient flow lines of the functional $\fL$ can naturally be identified with immersions $\biota \co \R\times P \to \R\times Y$ such that
\begin{equation*}
  \biota^*(\psi + \rd t\wedge\phi) = \vol_{\biota^*g}
\end{equation*}
and $\pi_\R\circ \biota(t,x) = t$;
see, e.g., \cite[Lemma 12.6]{Salamon2010}.

\begin{definition}
  Let $\iota^\pm \co P^\pm \to Y$ be associative immersions with respect to $\psi$.
  A \defined{Cayley cobordism} in $\R\times Y$ from $\iota_-$ to $\iota_+$ is an oriented $4$--manifold $Q$ together with an immersion $\biota\co Q \to \R\times Y$ such that
  \begin{equation*}
    \biota^*(\psi + \rd t\wedge\phi) = \vol_{\biota^*g}
  \end{equation*}
  and there are two open subsets $U_\pm \subset Q$ such that $Q \setminus (U_+\cup U_-)$ is compact,
  constants $T_\pm$ and $c > 0$,
  and diffeomorphisms $\phi_+\co (T_+,\infty) \times P^+ \to U_+$ and $\phi_-\co (-\infty,T_-) \times P^- \to U_-$ such that
  \begin{equation*}
    \mathrm{dist}(\biota\circ\phi_\pm(t,x),(t,\iota^\pm(x)))
    = O(e^{-c\abs{t}}) \quad\text{as } t \to \pm \infty.
  \end{equation*}
  The \defined{truncation} of a Cayley cobordism is (the diffeomorphism type of)
  \begin{equation*}
    \bar Q \coloneq Q \setminus \(\phi_-(-\infty,T_--1) \cup \phi_+(T_++1,\infty)\).
  \end{equation*}
\end{definition}

The functorial behavior of Seiberg--Witten Floer homology groups under cobordisms leads to the following questions about the existence of Cayley cobordisms.

\begin{question}
  In the situation of \autoref{Thm_DesingularizingIntersectionsOfAssocatives},
  does there exist a Cayley cobordism $\biota \co Q \to \R\times Y$ from $\iota_{t(\epsilon)}$ to $\iota^\csum_{\epsilon}$,
  for all $\epsilon \in (0,\epsilon_0)$,
  whose truncation $\bar Q$ is the surgery cobordism from $P$ to $P^\csum$?
\end{question}

\begin{question}
  In the situation of \autoref{Conj_ResolvingT2Singularity},
  if $\delta_2 > \delta_1 > 0$,
  does there exist a Cayley cobordism $\biota \co Q \to \R\times Y$ from $\iota^2_t$ to $\iota^3_t$ with $\bar Q$ being the surgery cobordism from $P^{\circ}_{\mu_2}$ to $P^{\circ}_{\mu_3}$ for each $t \in (0,T)$?
  (Similarly for the cases $\delta_1 > \delta_2  > 0$, $\delta_2 < \delta_ 1 < 0$, and $\delta_1 < \delta_2 < 0$.)
\end{question}

We \emph{hope} that the answer to these questions is yes.
For the sake of argument, let us assume that this is indeed the case.
Define
\begin{equation}
  \label{Eq_CMA}
  \CMA_\beta(\psi) \coloneq \bigoplus_{P} \bigoplus_{[\iota] \in \fA_{P,\beta}(\psi)} \CMA_{\beta,[\iota]}(\psi)
  \quad\text{with}\quad
  \CMA_{\beta,[\iota]}(\psi) \coloneq \CMfrom(P)
\end{equation}
and define a differential on $\CMA_\beta(\psi)$ by declaring
\begin{equation*}
  \(\del\co \CMA_{\beta,[\iota_-]}(\psi) \to \CMA_{\beta,[\iota_+]}(\psi)\)
  \coloneq \sum_{[\biota]} \CMfrom(\bar Q)
\end{equation*}
where $[\biota\co Q \to \R\times Y]$ ranges over all equivalence classes of Cayley cobordisms from $[\iota_-]$ to $[\iota_+]$.

Since $\CMfrom([0,1]\times P)$ is just the differential $\delfrom$ on $\CMfrom(P)$,
in the situation of \autoref{Thm_DesingularizingIntersectionsOfAssocatives} with $\delta > 0$ as in \autoref{Rmk_DesingularizingIntersectionsOfAssocatives} (and assuming that there no other Cayley cobordism involving $[\iota_t]$ or $[\iota^\csum_t]$),
for $t < 0$, the chain complex $\CMA_\beta(\psi_t)$ contains the contribution
\begin{equation*}
  \CMA_\beta^\times(\psi_t) = \CMfrom(P)
  \quad\text{with}\quad
  \del = \delfrom;
\end{equation*}
for $t > 0$ this changes to
\begin{equation*}
  \CMA_\beta^\times(\psi_t)
  = \CMfrom(P) \oplus \CMfrom(P^\csum)
  \quad\text{with}\quad
  \del =
  \begin{pmatrix}
    \delfrom & 0 \\
    \CMfrom(Q) & \delfrom
  \end{pmatrix}
\end{equation*}
with $Q$ the surgery cobordism from $P$ to $P^\csum$.
The latter is simply the mapping cone
\begin{equation*}
  \cone\paren[\big]{\CMfrom(P) \xrightarrow{\CMfrom(Q)} \CMfrom(P^\csum)}.
\end{equation*}
Therefore, it follows from \autoref{Thm_MonopoleHomology0SurgeryFormula}, that the homology group
\begin{equation*}
  H(\CMA_\beta^\times(\psi_t),\del)
\end{equation*}
does not change as $t$ passes through zero.
Similarly,
in the situation of \autoref{Conj_ResolvingT2Singularity},
by \autoref{Thm_MonopoleHomology1SurgeryFormula},
the relevant contribution to $H(\CMA_\beta(\psi_t),\del)$ does not change as $t$ passes through zero.

To conclude: while there seem to be no way of making the weighted count of associatives $n_\beta(\psi)$ invariant under transitions \autoref{It_Intersecting} and \autoref{It_T2Singularities} described in \autoref{Sec_Introduction}, we conjecture that a more refined object, the homology group  $H(\CMA_\beta^\times(\psi))$ is invariant under both of these transitions. 


\section{Multiple covers of associative submanifolds}
\label{Sec_MultipleCoverAssociatives}

A further problem with counting associatives arises from multiple covers; namely, transition \autoref{It_MultipleCovers} from \autoref{Sec_Introduction}.
This section is concerned with describing the nature of this phenomenon and its consequences for counting associative submanifolds.
In the following we explain how this issue might be rectified using the ADHM Seiberg--Witten equations, in a similar way that the issues described in the previous sections were dealt with using the classical Seiberg--Witten equation.

We have already established that, most likely, one cannot guarantee the number $n_\beta(\psi)$, or some other weighted count of associatives, to be invariant under deformations.
However, the problem with multiple covers is independent of the phenomena discussed earlier.
Thus, for the sake of simplicity we will only discuss how multiple covers affect $n_\beta(\psi)$ rather than the homology group $H(\CMA_\beta^\times(\psi))$; see also \autoref{Rem_Categorification} below. 

\subsection{Collapsing of immersions of multiple covers}
\label{Sec_CollapsingImmersions}

Consider the following situation.
Let $\iota_0 \co P \to Y$ be an associative immersion with respect to $\psi_0 \in \sD^4_c(Y)$ and with $(\iota_0)_*[P] = \beta \in H_3(Y)$.
Let $\pi \co \tilde P \to P$ be an orientation preserving $k$--fold unbranched normal cover with deck transformation group $\Aut(\pi)$.
The composition
\begin{equation*}
  \kappa_0 \coloneq \iota_0\circ\pi \co \tilde P \to Y
\end{equation*}  
is an associative immersion with
\begin{equation*}
  (\kappa_0)_*[\tilde P] = k\cdot \beta \qandq \Aut(\pi) \subset \Aut(\kappa_0).
\end{equation*}
Suppose that $[\iota_0]$ is unobstructed but
\begin{equation*}
  \ker F_{\kappa_0} = \R\Span{n} \subset \Gamma(N\kappa_0).
\end{equation*}
We expect that this situation can arise along generic paths $(\psi_t)_{t \in (-T,T)}$ in $\sD^4_c(Y)$.
A neighborhood of $([\kappa_0],\psi_0)$ in the $1$--parameter family of moduli spaces $\bigcup_t \fM_{k\cdot \beta}(\psi_t)$ can be analyzed using \autoref{Thm_AssociativeKuranishiModel}.

The stabilizer of $\kappa_0$ plays an important role in this analysis.
Since $\Aut(\kappa_0)$ acts on $N\kappa_0$ and $F_{\kappa_0}$ is $\Aut(\kappa_0)$--equivariant,
$\Aut(\kappa_0)$ acts on $\ker F_{\kappa_0}$.
This yields a homomorphism $\sign\co \Aut(\kappa_0) \to \set{\pm 1}$ such that
\begin{equation}
  \label{Eq_SignHomomorphism}
  f\cdot n = \sign(f)n
\end{equation}
for all $f \in \Aut(\kappa_0)$.
The homomorphism $\sign$ must be non-trivial, for otherwise $n$ would be $\Aut(\pi)$--invariant and descend to a non-trivial element of $\ker F_{\iota_0}$.

To summarize,
$\kappa_0 \co P \to Y$ is an associative immersion with respect to $\psi_0 \in \sD^4_c(Y)$ such that:
\begin{enumerate}
\item
  $\Aut(\kappa_0)$ is non-trivial,
\item
  $[\kappa_0]$ is obstructed; more precisely: $\ker F_{\kappa_0} = \R\Span{n}$, and
\item
  the homomorphism $\sign\co \Aut(\kappa_0) \to \set{\pm 1}$ defined by \eqref{Eq_SignHomomorphism} is non-trivial.
\end{enumerate}
In this situation,
if $(\psi_t)_{t \in (-T,T)}$ is generic,
then the obstruction map $\ob$ from \autoref{Thm_AssociativeKuranishiModel}, whose zero set models a neighborhood of $([\kappa_0],\psi_0)$ in $\bigcup_t \fM_{k\cdot \beta}(\psi_t)$, will be of the form
\begin{equation*}
  \ob(t,\delta) = \lambda t\delta + c\delta^3 + \text{higher order terms}.
\end{equation*}
We can assume that $\lambda = c = 1$.
Ignoring the higher order terms, $\ob^{-1}(0)$ consists of the line $\set{\delta=0}$ and the parabola $\set{t+\delta^2=0}$.
Since $[\iota_0]$ is unobstructed,
for each $\abs{t} \ll 1$,
there is an associative immersion $\iota_t \co P \to Y$ with respect to $\psi_t$ near $\iota_0$.
The line $\set{\delta=0}$ corresponds to the unobstructed associative immersions $[\kappa_t] \coloneq [\iota_t\circ\pi]$ for $\abs{t} \ll 1$.
By \autoref{Thm_AssociativeKuranishiModel}, for each $-1 \ll t < 0$ there are also associative immersions $[\kappa^\pm_t : \tilde P \to Y]$ with respect to $\psi_t$ near $[\kappa_0]$.
These correspond to the two branches of the parabola $\set{t+\delta^2=0}$.
As $t$ tends to $0$, $[\kappa^\pm_t]$ tends to $[\kappa_0]$;
and $\Aut(\kappa^\pm_t)$ is the stabilizer of $n$ in $\Aut(\kappa_0)$.
Since $\sign\co \Aut(\kappa_0) \to \set{\pm 1}$ is non-trivial,
there is an $f \in \Aut(\kappa_0)$ such that
\begin{equation*}
  f_*n = -n.
\end{equation*}
Therefore, $\kappa^+_t$ and $\kappa^-_t$ differ by a diffeomorphism of $\tilde{P}$ and give rise to the same element in the moduli space of associatives:
\begin{equation*}
  [\tilde \iota_t] \coloneq [\kappa^+_t] = [\kappa^-_t].
\end{equation*}
Thus, the neighborhood $\ob^{-1}(0)/\Aut(\kappa_0)$ of $([\kappa_0],\psi_0)$ in $\bigcup_t \fM_{k\cdot \beta}(\psi_t)$ is homeomorphic to the figure depicted in \autoref{Fig_CollapsingToMultipleCover}.
Consequently, $n_{k\cdot \beta}(\psi_t)$ as in \eqref{Eq_NDPhi_3rdAttempt} with the weight $w = \SW$ changes by $\pm \SW(\tilde P)$ as $t$ crosses zero.
Similarly, if one were to adopt the approach described in \autoref{Sec_AssociativeMonopoleHomology}, part of the chain complex $\CMA_{k\cdot \beta}(\psi_t)$ would disappear as $t$ crosses zero.

\begin{figure}[h]
  \centering
  \begin{tikzpicture}
    \draw[thick,cyan] (-3,1.5) node [left] {$[\tilde\iota]$} .. controls (-2,1.5) and (0,1.5) .. (0,.5);
    \draw[thick,magenta] (-3,0.5) node [left] {$[\kappa]$} -- (3,.5);
    \draw[gray] (-3,0) -- (0,0);
    \filldraw[gray] (0,.5) circle (.05);
    \draw[|-stealth,gray] (0,0) -- (3,0) node [right] {$\psi_t$};
  \end{tikzpicture}
  \caption{An family of associative immersions collapsing to a multiple cover.}
  \label{Fig_CollapsingToMultipleCover}
\end{figure}
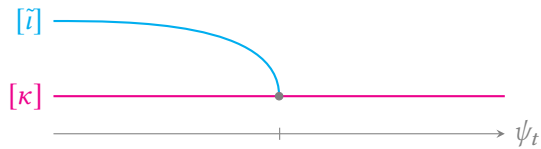

\subsection{Counting orbifolds points}

The standard way to deal with the issue of multiple covers is to count the immersions $[\kappa]$ and $[\tilde\iota]$ described before as orbifold points in the moduli space;
that is, to define
\begin{equation}
  \label{Eq_NDPhi_4thAttempt}
  n_\beta(\psi) \coloneq \sum_{[\iota] \in \fM_\beta(\psi)} \frac{\epsilon([\iota],\psi) w(P)}{\abs{\Aut(\iota)}}.
\end{equation}
Since $[\kappa_0]$ is obstructed,
more precisely, since the Fueter operator associated with $\kappa_0$ has a $1$--dimensional kernel,
\autoref{Eq_EpsilonSpectralFlow} implies that the sign $\epsilon([\kappa_t],\psi_t) \in \set{\pm 1}$ flips as $t$ passes through $0$.
Moreover,
\begin{equation*}
  \Aut(\tilde\iota) = \ker \sign \subset \Aut(\kappa),
\end{equation*}
where $\sign\co \Aut(\kappa_0) \to \set{\pm 1}$ is the homomorphism introduced above, and thus
\begin{equation*}
  \abs{\Aut(\kappa)} = 2\cdot \abs{\Aut(\tilde\iota)}.
\end{equation*}
Consequently, for $0 < t \ll 1$, we have
\begin{equation*}
  \frac{\epsilon([\kappa_{-t}],\psi_{-t}) w(\tilde P)}{\abs{\Aut(\kappa_{-t})}}
  +
  \frac{\epsilon([\tilde\iota_{-t}],\psi_{-t}) w(\tilde P)}{\abs{\Aut(\tilde\iota)}}
  =
  \frac{\epsilon([\kappa_{+t}],\psi_{+t}) w(\tilde P)}{\abs{\Aut(\kappa_{+t})}} \in \Q.
\end{equation*}

This works well for unbranched covers, but we believe that similar situations can occur with branched covers $\pi \co \tilde P \to P$.
If $\pi$ is a branched cover (with non-empty branching locus),
then $\kappa \coloneq \iota\circ\pi$ is not an immersion and thus the theory from \autoref{Sec_CountingAssociatives} does not apply.
What exactly replaces this theory is unclear to us; the work of \citet{Smith2011} might be a starting point.
Nevertheless, one would need to count $[\kappa]$ to be able to compensate the jump.
The crucial point is that, for any given $3$--manifold $P$ and $k \in \N$, infinitely many diffeomorphism types of $3$--manifolds might be realized as $k$--fold branched covers of $P$.
This is illustrated by the following result.

\begin{theorem}[{\citet{Hilden1974,Hilden1976,Montesinos1974}}]
  Every compact, connected, orientable $3$--manifold is a $3$--fold branched cover of $S^3$.
\end{theorem}

Therefore, if $\iota\co S^3 \to Y$ is an associative immersion in $(Y,\psi)$,
then, for every compact, connected, oriented $3$--manifold $\tilde P$, there is a $3$--fold branched cover $\pi \co \tilde P \to P$, and $[\iota\circ \pi]$ would have to contribute to \eqref{Eq_NDPhi_4thAttempt}.
This would lead to an infinite contribution from branched covers.

\subsection{Counting embeddings with multiplicty}
\label{Sec_EmbeddingsWithMultiplicity}

We believe that the origin of the problem is that all the associative submanifolds $[\iota\circ \pi]$ represent the same geometric object, namely, ``$k$ times $\im(\iota)$''.
Instead of trying to count immersions and their compositions with branched covers with weights,
we should count embeddings with multiplicity.
Embeddings with with multiplicity one should be weighted by the Seiberg--Wittten invariant,
as in \autoref{Sec_RoleOfSeibergWitten} or \autoref{Sec_AssociativeMonopoleHomology}.
Below we briefly outline an approach for defining the weights with which to count embeddings with multiplicity $k$ larger than one.
More details are given in \autoref{Sec_DegenerationsOfADHMMonopoles} and \autoref{Sec_TentativeProposal}.

\begin{remark}
  Our approach should be compared with holomorphic curve counting via Donaldson--Thomas/Pandharipande--Thomas theory in algebraic geometry where one counts embedded subschemes, including contributions from thickened subschemes, rather than images of maps.
  We elaborate on the relationship of this approach with Pandharipande--Thomas theory in \autoref{Sec_ADHMBundles}.
\end{remark}

To set the stage,
let us go back to the situation described at the beginning of this section;
that is, we have an unobstructed associative embedding $\iota \colon P \to Y$
and an orientation preserving $k$--fold unbranched cover $\pi\co \tilde P \to P$ such that
\begin{equation*}
  \kappa \coloneq \iota\circ \pi \co \tilde P \to Y
\end{equation*}
is an obstructed associative immersion with $\dim\ker F_\kappa = 1$.
Denote by $\tilde\iota\co \tilde P \to Y$ the associative immersion which is the deformation of $\kappa$ that does not come from deforming $\iota$.
(For simplicity's sake, we dropped the subscripts $t$ from the notation.)
Consider the bundle of stratified spaces
\begin{equation*}
  \Sym^kN\iota \coloneq \SO(N\iota)\times_{\SO(4)}\Sym^k\H = (N\iota)^k/S_k.
\end{equation*}
Here $\H = \R^4$ is the space of quaternions and $S_k$ is the symmetric group on $k$ elements.
To every normal vector field $n \in \Gamma(N\kappa)$ we assign a corresponding section $\tilde n \in \Gamma(\Sym^kN\iota)$ defined by
\begin{equation*}
  \tilde n(x)
  \coloneq
  [n(\tilde x_1),\ldots,n(\tilde x_k)]
\end{equation*}
with $\tilde x_1, \ldots, \tilde x_k$ denoting the preimages of $x$ with multiplicity.
Given such a section $\tilde n \in \Gamma(\Sym^kN\iota)$,
set
\begin{equation*}
  P_{\tilde n} \coloneq \set{ (x,v) \in N\iota : v \in \tilde n(x) }.
\end{equation*}
If $n \in \Gamma(N\kappa)$ is a normal vector field spanning $\ker F_\kappa$,
then $P_{\tilde n}$ is a model for $\im(\tilde\iota)$.
In particular, $\im(\tilde\iota)$ and $P_{\tilde n}$ are diffeomorphic in case they are smooth,
which we conjecture be true generically if $\pi$ is unbranched.

We can decompose $\im(\tilde\iota)$ into components $P^1, \ldots, P^m$ such that $P^j$ is an $\ell_j$--fold cover of $P$ and,
for each $\tilde x \in P^j$ corresponding to $(x,v) \in P_{\tilde n}$, $v$ appears in $\tilde n(x)$ with multiplicity $k_j$.
Geometrically, $[\tilde\iota]$ represents
\begin{equation}
  \label{Eq_SumKjPj}
  k_1\cdot P^1 + \cdots +  k_m\cdot P^m.
\end{equation}
Clearly, we have
\begin{equation}
  \label{Eq_PartitionOfK}
  \sum_{j=1}^m \ell_jk_j = k.
\end{equation}

Henceforth, let us assume that $\im(\tilde\iota)$ is smooth.
In the simplest case, we have $m = 1$ and $k_1=k$.
In this case, $\tilde n$ is a section of
\begin{equation*}
  \Sym^k_\reg N\iota
  \coloneq
  \set*{ (x,[v_1,\ldots,v_k]) \in \Sym^kN\iota : v_1,\ldots,v_k \text{ are pairwise distinct} },
\end{equation*}
the top stratum of $\Sym^kN\iota$.
In general, $\tilde n$ will be a section of a stratum
\begin{equation*}
  \Sym^k_\lambda N\iota \subset  \Sym^kN\iota
\end{equation*}
determined by $\lambda$, the partition of the natural number $k$ given by \eqref{Eq_PartitionOfK}.
Each of the strata $\Sym^k_\lambda N\iota$ is a smooth fibre bundle,
which is naturally equipped with a connection $\nabla$ and and a Clifford multiplication $\gamma$ on its vertical tangent bundle $V \Sym_\lambda^kN\iota$.
These can be used to define a Fueter operator, which assigns to each section $\tilde n \in \Gamma(\Sym^k_\lambda N\iota)$ an element
\begin{equation*}
  \fF \tilde n \in \Gamma(\tilde n^*V\Sym^k_\lambda N\iota).
\end{equation*}
The condition that $n \in \Gamma(N\kappa)$ is in the kernel of $F_\kappa$ means that
\begin{equation*}
  \fF \tilde n \coloneq \gamma(\nabla\tilde n) = 0;
\end{equation*}
that is, $\tilde n$ is a Fueter section of $\Sym^k_\lambda N\iota$.

The above discussion show that what causes $k_1\cdot P^1 + \cdots + k_m\cdot P^m$ to collapse to $k \cdot \im(\tilde\iota)$
is precisely a Fueter section $\tilde n$ of $\Sym^k_\lambda N\iota$.
For simplicity,
let us specialize to the case $m = 1$ and $k_1=k$;
that is:
\begin{itemize}
\item
  for $t < 0$
  there are two embedded associative submanifolds of interest,
  namely,
  $[\tilde\iota_t\co \tilde P \to Y]$ and $[\iota_t\co P \to Y]$;
\item
  as $t$ tends to zero, $\tilde\iota_t$ converges to the associative immersion $\kappa$, the $k$--fold covering of $\iota_0$, and then ceases to exist;
  and
\item
  for $t > 0$ we only have the embedded associative submanifold $[\iota_t\co P \to Y]$.
\end{itemize}
Extending the approach of \autoref{Sec_RoleOfSeibergWitten},
we would like to define weights $w$ such that
\begin{equation}
  \label{Eq_SimpleWeightRelation}
  w(\tilde P,\psi_{-t}) + w(k\cdot P,\psi_{-t})
  =
  w(k\cdot P,\psi_{+t})
\end{equation}
for $0 < t \ll 1$.
From the discussion in \autoref{Sec_RoleOfSeibergWitten} we learn that $w(\tilde P,\psi_t)$ should be $\epsilon(\tilde P,\psi_{-t})\cdot \SW(\tilde P)$ with $\epsilon(P,\psi_{-t}) \in \set{\pm 1}$ as in \autoref{Sec_EnumerativeInvariants?} and $\SW(\tilde P) \in \Z$ being the Seiberg--Witten invariant of $\tilde P$.
Thus \eqref{Eq_SimpleWeightRelation} means that the weight $w(k\cdot P,\psi_t)$ must jump by $\pm \SW(\tilde P)$ as $t$ passes through zero.

We propose that $w(k\cdot P,\psi_t)$ should be defined as the signed count of solutions to the \defined{ADHM$_{1,k}$ Seiberg--Witten equation} on $P$.
This is the Seiberg--Witten equation associated with the ADHM construction of $\Sym^k\H$.
Unlike in the case of the classical Seiberg--Witten equation, compactness fails for the ADHM$_{1,k}$ Seiberg--Witten equation.
As a consequence, the number of solutions can jump as the geometric background varies.
According to the \defined{Haydys correspondence}, those jumps occur precisely when (possibly singular) Fueter sections of $\Sym^k N\iota$ appear.
We will argue that in the above situation the jumps should be precisely by $\pm \SW(\tilde P)$.

The next section is concerned with introducing the ADHM$_{1,k}$ Seiberg--Witten equation,
stating and proving the Haydys correspondence with stabilizers,
and formally analyzing the failure of non-compactness for the ADHM$_{1,k}$ Seiberg--Witten equation.
After this discussion we will also explain what replaces \eqref{Eq_SimpleWeightRelation},
in general, and why defining $w$ via the ADHM$_{1,k}$ Seiberg--Witten equation should be consistent with that.

\begin{remark}
  \label{Rem_Categorification}
  Of course,
  instead of a weighted \emph{count} of embedded associatives with multiplicities,
  one should really try to define a Floer homology generalizing the discussion in \autoref{Sec_AssociativeMonopoleHomology}.
  Such ADHM$_{1,k}$ Seiberg--Witten--Floer homology groups are yet to be defined.
  It will become clear from the discussion in the following sections that these groups could only be expected to yield topological invariants of $3$--manifolds in the case $k=1$ (classical Seiberg--Witten--Floer homology).
  In general, they will depend on various parameters of the equation such as the Riemannian metric. 
\end{remark}

\begin{remark}
  We believe that this approach is also capable of dealing with branched covers.
  These should correspond to singular Fueter sections, that is, sections of $\Sym^k_\lambda N\iota$ defined outside a subset of codimension at most one (which corresponds to the branching locus) and extend a continuous section of the closure of $\Sym^k_\lambda N\iota$ in $\Sym^kN\iota$.
  It is known that singular Fueter sections appear in the compactifications of moduli spaces of solutions to Seiberg--Witten equations, cf.~\cite{Doan2017c}.
\end{remark}


\section{ADHM monopoles and their degenerations}
\label{Sec_DegenerationsOfADHMMonopoles}

The purpose of this section is to introduce ADHM monopoles and to relate their degenerations to the phenomenon of collapsing of associatives to multiple covers.

\subsection{The \texorpdfstring{ADHM}{ADHMrk} Seiberg--Witten equations}
\label{Sec_ADHMrkSeibergWittenEquations}

There is a general construction, summarized in \autoref{Sec_SeibergWitten}, which associates with every quaternionic representation of a Lie group a generalization of the Seiberg--Witten equation on $3$--manifolds.
In a nutshell, the ADHM Seiberg--Witten equations arise from this construction by choosing particular quaternionic representations which appear in the famous ADHM construction of instantons on $\R^4$; see  \autoref{Ex_ADHMAlgebraicData}.
However, below we introduce the ADHM Seiberg--Witten equations directly, without assuming that the reader is familiar with the general construction.

\begin{definition}
  Let $M$ be an oriented Riemannian $3$--manifold.
  Consider the Lie group
  \begin{equation*}
    \Spin^{\U(k)}(n) \coloneq (\Spin(n)\times\U(k))/\Z_2.
  \end{equation*}
  A \defined{spin$^{\U(k)}$ structure} on $M$ is a principal $\Spin^{\U(k)}(3)$--bundle together with an isomorphism
  \begin{equation}
    \label{Eq_SpinUkStructure}
    \fw \times_{\Spin^{\U(k)}(3)} \SO(3) \iso \SO(TM).
  \end{equation}
  The \defined{spinor bundle} and the \defined{adjoint bundle} associated with a spin$^{\U(k)}$ structure $\fw$ are
  \begin{equation*}
    W \coloneq \fw \times_{\Spin^{\U(k)}(3)} \H\otimes_\C \C^k
    \qandq
    \fg_\sH \coloneq \fw \times_{\Spin^{\U(k)}(3)} \fu(k)
  \end{equation*}
  respectively.
  The left multiplication by $\Im\H$ on $\H\otimes\C^k$ induces a \defined{Clifford multiplication} $\gamma\co TM \to \End(W)$.
  
  A \defined{spin connection} on $\fw$ is a connection $A$ inducing the Levi-Civita connection on $TM$.
  Associated with each spin connection $A$ there is a \defined{Dirac operator} $\slD_A\co \Gamma(W) \to \Gamma(W)$.
  
  Denote by $\sA^s(\fw)$ the space of spin connections on $\fw$,
  and by $\sG^s(\fw)$ the \defined{restricted gauge group}, consisting of those gauge transformations which act trivially on $TM$.
  Let $\varpi\co \Ad(\fw) \to \fg_\sH$ be the map induced by the projection $\spin^{\U(k)}(3) \to \fu(k)$.
\end{definition}

\begin{definition}
  \label{Def_ADHMGeometricData}
  Let $M$ be an oriented $3$--manifold.
  The \defined{geometric data} needed to formulate the ADHM$_{r,k}$ Seiberg--Witten equation are:
  \begin{itemize}
  \item
    a Riemannian metric $g$,
  \item
    a spin$^{\U(k)}$ structure $\fw$,
  \item
    a Hermitian vector bundle $E$ of rank $r$
    with a fixed trivialization $\Lambda^rE = \C$ and an $\SU(r)$--connection $B$,
  \item
    an oriented Euclidean vector bundle $V$ of rank $4$ together with an isomorphism
    \begin{equation}
      \label{Eq_SOLambda+V=SOTM}
      \SO(\Lambda^+V) \iso \SO(TM)
    \end{equation}
    and an $\SO(4)$--connection $C$ on $V$ with respect to which this isomorphism is parallel.
  \end{itemize}
\end{definition}

\begin{remark}
  If $\iota \co P \to Y$ is an associative immersion,
  then the normal bundle $V = N\iota$ admits a natural isomorphism \eqref{Eq_SOLambda+V=SOTM} by \autoref{Prop_AssociativeViaCrossProduct} and we can take $C$ to be the connection induced by the Levi-Civita connection.
  In this context, the bundle $E$ should be the restriction to $P$ of a bundle on the ambient $\Gtwo$--manifold and $B$ should be the restriction of a $\Gtwo$--instanton.
  Soon we will specialize to the case $r=1$, in which $E$ is trivial and $B$ is the trivial connection.
\end{remark}

The above data makes both $\Hom(E,W)$ and $V\otimes \fg_\sH$ into Clifford bundles over $M$;
hence, there are Dirac operators $\slD_{A,B} \co \Gamma(\Hom(E,W)) \to \Gamma(\Hom(E,W))$ and $\slD_{A,C} \co \Gamma(V\otimes \fg) \to \Gamma(V\otimes \fg)$.
The ADHM$_{r,k}$ Seiberg--Witten equation involves also two quadratic moment maps defined as follows.
If $\Psi \in \Hom(E,W)$, then $\Psi\Psi^* \in \End(W)$.
Since $\Lambda^2T^*M\otimes \fg_\sH$ acts on $W$, there is an adjoint map $(\cdot)_0\co \End(W) \to \Lambda^2T^*M\otimes \fg_\sH$.
Define $\mu \co \Hom(E,W) \to \Lambda^2T^*M\otimes \fg_\sH$ by
\begin{equation*}
  \mu(\Psi) \coloneq (\Psi\Psi^*)_0.
\end{equation*}
If $\bxi \in V\otimes\fg$, then $[\bxi\wedge \bxi] \in \Lambda^2 V\otimes \fg_\sH$.
Denote its projection to $\Lambda^+V\otimes \fg_\sH$ by $[\bxi\wedge \bxi]^+$.
Identifying $\Lambda^+V \iso \Lambda^2T^*M$ via the isomorphism \eqref{Eq_SOLambda+V=SOTM}, we define $\mu \co V\otimes\fg \to  \Lambda^2T^*M \otimes \fg_\sH$ by
\begin{equation*}
  \mu(\bxi) \coloneq [\bxi\wedge\bxi]^+
\end{equation*}

\begin{definition}
  Given a choice of geometric data as in \autoref{Def_ADHMGeometricData},
  the \defined{ADHM$_{r,k}$ Seiberg--Witten equation} is the following partial differential equation for $(\Psi,\bxi,A) \in \Gamma(\Hom(E,W))\times\Gamma(V\otimes\fg_\sH)\times\sA^s(\fw)$:
  \begin{equation}
    \label{Eq_ADHMrkSW}
    \begin{split}
      \slD_{A,B} \Psi &= 0, \\
      \slD_{A,C} \bxi &= 0, \qand \\
      \varpi F_A &= \mu(\Psi) + \mu(\bxi).
    \end{split}
  \end{equation}
  A solution of this equation is called an \defined{ADHM$_{r,k}$ monopole}.
\end{definition}

The moduli space of ADHM$_{r,k}$ monopoles might be non-compact.
The reason is that the $L^2$ norm of the pair $(\Psi,\bxi)$ is not a priori bounded and can diverge to infinity for a sequence of solutions.
To understand this phenomenon, one blows-up the equation by multiplying $(\Psi,\bxi)$ by $\epsilon^{-1}$ and studies the equation obtained by taking the formal limit $\epsilon \to 0$. 
This is explained in greater detail in \autoref{Sec_SeibergWitten}.

\begin{definition}
  The \defined{limiting ADHM$_{r,k}$ Seiberg--Witten equation} the following partial differential equation for $(\Psi,\bxi,A) \in \Gamma(\Hom(E,W))\times\Gamma(V\otimes\fg_\sH)\times\sA^s(\fw)$
  \begin{equation}
    \label{Eq_LimitingADHMrkSW}
    \begin{split}
      \slD_{A,B} \Psi &= 0, \\
      \slD_{A,C} \bxi &= 0, \qand \\
      \mu(\Psi) + \mu(\bxi) &=0.
    \end{split}
  \end{equation}
  together with the normalization $\Abs{(\Psi,\bxi)}_{L^2} = 1$.
\end{definition}

The ADHM$_{r,k}$ Seiberg--Witten equation \eqref{Eq_ADHMrkSW} and the corresponding limiting equation are preserved by the action of the restricted gauge group $\sG^s(\fw)$.

\begin{remark}
  \label{Rmk_ADHM11SW}
  Suppose that $r=k=1$.
  A spin$^{\U(1)}$ structure is simply a spin$^c$ structure and
  \begin{equation*}
    \varpi F_A = \frac12 F_{\det A}.
  \end{equation*}
  Also, $\fg_\sH = i\underline{\R}$;
  hence, $\slD_{A,C}$ is independent of $A$ and  $\mu(\bxi) = 0$.
  The ADHM$_{1,1}$ Seiberg--Witten equation is thus simply
  \begin{equation*}
    \slD_A \Psi = 0 \qandq
    \frac12 F_{\det A} = \mu(\Psi),
  \end{equation*}
  the classical Seiberg--Witten equation \eqref{Eq_ClassicalSeibergWitten} for $(\Psi,A)$, together with the Dirac equation
  \begin{equation*}
    \slD_C \bxi = 0.
  \end{equation*}
  
  If $\iota\co P \to Y$ is an associative immersion and $M = P$ and $V = N\iota$,
  then $\slD_C$ is essentially the Fueter operator $F_\iota$ from \autoref{Def_FueterNormal}.
  In particular, $\bxi$ must vanish if $\iota$ is unobstructed.
  (There is a variant of \eqref{Eq_ADHMrkSW} in which $\bxi$ is taken to be a section of $V\otimes \fg^\circ_{\sH}$ with $\fg^\circ_{\sH}$ denoting the trace-free component of $\fg_{\sH}$.
  For $r=k=1$, this equation is identical to the classical Seiberg--Witten equation.
  However, working with this equation somewhat complicates the upcoming discussion of the following sections.)
\end{remark}

\subsection{The Haydys correspondence for the \texorpdfstring{ADHM$_{1,k}$}{ADHM1k} Seiberg--Witten equation}

In what follows, we specialize to the case $r=1$ and analyze solutions of the limiting ADHM$_{1,k}$ Seiberg--Witten equation \eqref{Eq_LimitingADHMrkSW}.
This will lead to a conjectural compactification of the moduli space of ADHM$_{1,k}$ monopoles.
Our analysis is based on the general framework of the Haydys correspondence with stabilizers developed in \autoref{Sec_HaydysCorrespondence}.
We will also make use of several algebraic facts proved in \autoref{Sec_NakajimasProof}.
It is helpful but not necessary have read the appendices to understand the results stated in this section.

Assume the situation of \autoref{Sec_ADHMrkSeibergWittenEquations};
that is:
$\fw$ is a spin$^{\U(k)}$ structure on $M$ with spinor bundle $W$ and adjoint bundle $\fg_\sH$, and
$V$ is a Dirac bundle of rank $4$ over $M$ with connection $C$.
The \defined{limiting ADHM$_{1,k}$ Seiberg--Witten equation} for a triple $(\Psi,\bxi,A) \in \Gamma(W)\times\Gamma(V\otimes\fg_\sH)\times\sA^s(\fw)$ is
\begin{equation}
  \label{Eq_LimitingADHM1kSW}
  \begin{split}
    \slD_A\Psi &= 0, \\
    \slD_{A,C}\bxi &= 0, \qand \\
    \mu(\Psi) + \mu(\bxi) &= 0
  \end{split}
\end{equation}
as well as $\Abs{(\Psi,\bxi)}_{L^2} = 1$.

It follows from the third equation that if $(\Psi,\bxi,A)$ is a solution of \eqref{Eq_LimitingADHM1kSW}, then
\begin{enumerate}
\item $\Psi = 0$, and
\item $\bxi$ induces a section $\tilde n$ of the bundle $\Sym^k V$ over $M$ whose fiber is $\Sym^k \H$.
\end{enumerate}
The first statement is the content of \autoref{Prop_1kADHMMu=0=>Phi=0} and the second statement follows from a special case of the Haydys correspondence discussed in \autoref{Sec_HaydysCorrespondence}, combined with the observation that $\Sym^k \H$ is the hyperk\"ahler quotient of the ADHM$_{1,k}$ representation; see \autoref{Thm_1kADHM}.
Furthermore, the section $\tilde n$ satisfies the Fueter equation, as explained in \autoref{Sec_ProjectingDiracEquation}.

A more difficult part of the Haydys correspondence deals with the converse problem: given a section $\tilde n$ of $\Sym^k V$ which satisfies the Fueter equation, can we lift it to a solution $(\Psi,\bxi,A)$ of \eqref{Eq_LimitingADHM1kSW}? If yes, what is the space of all such lifts up to the action of the gauge group?

A technical difficulty that one has to overcome is that $\tilde n$ takes values in the symmetric product $\Sym^k \H$ which is not a manifold.
Rather, it is a stratified space whose strata correspond to the partitions of $k$.

\begin{definition}
  A \defined{partition} of $k \in \N$ is a non-increasing sequence of non-negative integers $\lambda = (\lambda_1, \lambda_2, \ldots)$ which sums to $k$.
  The \defined{length} of a partition is
  \begin{equation*}
    \abs{\lambda}
    \coloneq
    \min \set{ n \in \N : \lambda_n = 0 } - 1.
  \end{equation*}
  
  For every $n \in \N$, denote by $S_n$ the permutation group on $n$ elements. 
  With each partition $\lambda$ we associate the groups
  \begin{equation*}
    G_\lambda
    \coloneq
    \set*{
      \sigma \in S_{\abs{\lambda}}
      :
      \lambda_{\sigma(n)} = \lambda_n \text{ for all } n \in \set {1 \ldots, \abs{\lambda}}
    }
  \end{equation*}
  and the generalized diagonal
  \begin{equation*}
    \Delta_{\abs{\lambda}}
    =
    \set{
      v_1,\ldots,v_{\abs{\lambda}} \in \H^{\abs{\lambda}} : v_i = v_j \text{ for some } i \neq j
    }.
  \end{equation*}
  There is an embedding $(\H^{\abs{\lambda}}\setminus \Delta_{\abs{\lambda}})/G_\lambda \into \Sym^k \H$ defined by
  \begin{equation*}
    [v_1,\ldots,v_{\abs{\lambda}}] \mapsto [\underbrace{v_1,\ldots,v_1}_{\lambda_1 \text{ times}},\cdots,\underbrace{v_{\abs{\lambda}},\ldots,v_{\abs{\lambda}}}_{\lambda_{\abs{\lambda}} \text{ times}}].
  \end{equation*}
  The image of this inclusion is denoted by $\Sym^k_\lambda\H$.
\end{definition}
Each stratum $\Sym^k_\lambda \H$ is a smooth manifold. 
Let us assume that $\tilde n$ takes values in such a stratum:
\begin{equation*}
  \tilde n \in \Gamma(\Sym^k_\lambda V),
\end{equation*}
for some partition $\lambda$ of $k$.
This is familiar from \autoref{Sec_EmbeddingsWithMultiplicity}.

The next result summarizes the Haydys correspondence for solutions of \eqref{Eq_LimitingADHM1kSW}.
On first reading, the reader might assume that $\lambda = (1,\ldots,1)$, the partition yielding the top stratum of $\Sym^k\H$, since this simplifies the situation considerably.
For $j = 1,\ldots,m$,
denote by $k_j$ the $j$--th largest positive number appearing in the partition $\lambda$ and by $\ell_j$ the multiplicity with which it appears.

\begin{prop}
  \label{Prop_ADHM1kHaydys}
  Given $\tilde n \in \Gamma(\Sym^k_\lambda V)$,
  set
  \begin{equation*}
    \tilde M \coloneq \set{ (x,v) \in V : x \in M \qandq v \in \tilde n(x) }
  \end{equation*}
  and denote by $\pi\co \tilde M \to M$ by the projection map.

  \begin{enumerate}
  \item \label{Prop_ADHM1kHaydys_Covering}
    The map $\pi$ is a $\abs{\lambda}$--fold unbranched cover of $M$.
    Moreover, we can decompose $\tilde M$ into components $\tilde M_1, \ldots, \tilde M_m$ such that $\pi_j \coloneq \pi|_{\tilde M_j}$ restricts to a $\ell_j$--fold cover on $\tilde M_j$.
  \item \label{Prop_ADHM1kHaydys_Correspondence} There is a natural bijective correspondence between
    \begin{enumerate}[(a)]
    \item
      gauge equivalence classes of solutions $(\Psi,\bxi,A)$ of \eqref{Eq_LimitingADHM1kSW} for which the corresponding section of $\Sym^k V$ takes values in the stratum $\Sym^k_\lambda V$, and
    \item
      Fueter sections $\tilde n \in \Gamma(\Sym^k_\lambda V)$ together with a spin$^{\U(k_j)}$ structure $\fw_j$ on $\tilde M_j$ and a spin connection $A_j$ on $\fw_j$ for each $j = 1,\ldots,m$.
    \end{enumerate}
  \end{enumerate}
\end{prop}

\begin{remark}
  If $\lambda = (1,\ldots,1)$, then $m = 1$ and $\fw_1$ is simply a spin$^c$ structure on $\tilde M$.
\end{remark}

\begin{proof}
  Part \autoref{Prop_ADHM1kHaydys_Covering} follows from the definitions of $\Sym^k_\lambda V$ and $\tilde M$. 
  It is part \autoref{Prop_ADHM1kHaydys_Correspondence} which requires a proof.  
  This statement is a special case of the Haydys correspondence with stabilizer proved in \autoref{Sec_HaydysCorrespondence}; in particular, we will use the notation introduced in there. 

  We require the following pieces of notation.
  For every $n \in \N$, denote by $[n]$ the set $\set{1,\ldots,n}$,
  and let $S_n$ be the permutation groups on $n$ elements.
  Denote by $Q^\diamond$ the principal $\prod_{j=1}^m S_{\ell_j}$--bundle over $M$, denoted whose fibre over $x$ is
  \begin{equation}
    \label{Eq_QDiamond}
    Q^\diamond_x
    =
    \prod_{j=1}^{m} \Bij\paren[\big]{[l_j], \pi_j^{-1}(x)}.
  \end{equation}
  Tautologically,
  $\tilde M$ is the fiber bundle with fiber $[l_1]\times \cdots\times [l_m]$ associated with $Q^\diamond$ using the action of $\prod_{j=1}^m S_{\ell_j}$ on $[l_1]\times \cdots\times [l_m]$.
  Define
  \begin{align*}
    T_\lambda
    &\coloneq
      \prod_{n = 1}^{\abs{\lambda}} \U(\lambda_n)
      \subset
      \U(k), \\    
    W_{\hat H}(T_\lambda)
    &\coloneq
      \(\prod_{j=1}^m S_{\ell_j}\) \times \SO(4), \qandq \\
    N_{\hat H}(T_\lambda)
    &\coloneq
      \Spin(4)\times_{\Z_2} \(\prod_{j=1}^m S_{\ell_j} \ltimes \U(k_j)^{\ell_j}\) \\
    & = \(\Spin(3)\times_{\Z_2}\(\prod_{j=1}^m S_{\ell_j} \ltimes \U(k_j)^{\ell_j}\) \)\times_{\SO(3)} \SO(4).
  \end{align*}
  With this notation the following summarizes the discussion in \autoref{Sec_HaydysCorrespondence}.

  \begin{prop}
    \label{Prop_ADHM1kHaydysTechnical}
    Let $Q^\diamond$ be the principal $\prod_{j=1}^mS_{\ell_j}$--bundle defined by \autoref{Eq_QDiamond}. 
    Define a principal $W_{\hat H}(T_\lambda)$--bundle $\hat Q^\diamond$ associated with $\tilde n$ by
    \begin{equation*}
      \hat Q^\diamond = Q^\diamond\times \SO(V).
    \end{equation*}
    \begin{enumerate}
    \item
      \label{Prop_ADHM1kHaydys_SpinUl}
      The choice of a $N_{\hat H}(T_\lambda)$--bundle  $\hat Q^\circ$ lifting $\hat Q^\diamond$ is equivalent to the choice of a spin$^{\U(k_j)}$ structure $\fw_j$ on $\tilde M_j$ for each $j = 1,\ldots,m$.
    \item
      \label{Prop_ADHM1kHaydys_SpinConnection}
      Given a spin$^{\U(k_j)}$ structure $\fw_j$ on $\tilde M_j$ for each $j = 1,\ldots,m$,
      there exists a lift $(\Psi,\bxi)$ of $\tilde n$.
      The space of connections $\sA^{\Psi,\bxi}_C(\hat Q)$, defined in \eqref{Eq_APhiB}, is identified with the space
      \begin{equation*}
        \prod_{j=1}^m \sA^s(\fw_j)
      \end{equation*}
      and $\ft_P$, defined in \eqref{Eq_ftP}, is identified with the sum of the push-forward bundles
      \begin{equation*}
        \bigoplus_{j=1}^m (\pi_j)_*\fg_{\sH_j}.
      \end{equation*}
    \end{enumerate}
  \end{prop}

  \begin{proof}[Proof of \autoref{Prop_ADHM1kHaydysTechnical}]
    We prove part \autoref{Prop_ADHM1kHaydys_SpinUl}.
    Given a spin$^{\U(k_j)}(3)$ structure $\fw_j$ on $\tilde M_j$ for each $j = 1,\ldots,m$,
    denote by $\tilde \fw_j$ the corresponding spin$^{\U(k_j)}(4)$ structure on $\pi_j^*V$.
    The principal $N_{\hat H}(T_\lambda)$--bundle $\hat Q^\circ$ with fibre over $x$ given by
    \begin{equation*}
      \hat Q^\circ_x
      =
      \prod_{j=1}^{m}
      \set[\Big]{
        (f,g_1,\ldots,g_{\ell_j})
        \in
        \Bij\paren[\big]{\set{ 1,\ldots,\ell_j }, \pi_j^{-1}(x)}
        \times \tilde\fw_j^{\ell_j}
        :
        g_i \in (\tilde\fw_j)_{f(i)}
      }
    \end{equation*}
    lifts $\hat Q^\diamond$.
    Conversely, given principal $N_{\hat H}(T_\lambda)$--bundle $\hat Q^\circ$ lifting $\hat Q^\diamond$ its pullback to $\tilde M_j$ contains a principal $\Spin^{\U(k_j)}(4)$--bundle $\tilde \fw_j$ which yields a spin$^{\U(k_j)}$ structure on $\pi_j^*V$ and thus on $\tilde M_j$.
    With this discussion in mind and the discussion in \autoref{Sec_HaydysCorrespondence}, part \autoref{Prop_ADHM1kHaydys_SpinConnection} of this proposition becomes apparent.
  \end{proof}

  Once \autoref{Prop_ADHM1kHaydysTechnical} is established, part 
  \autoref{Prop_ADHM1kHaydys_Correspondence} of \autoref{Prop_ADHM1kHaydys} follows from the discussion in 
  \autoref{Sec_LiftingSections} and \autoref{Sec_ProjectingDiracEquation} together with \autoref{Thm_1kADHM}.
\end{proof}

\subsection{Formal expansion around limiting solutions}
\label{Sec_FormalExpansion}

\autoref{Prop_ADHM1kHaydys} imposes very weak conditions on a connection $A \in \sA^s(\fw)$ which is part of a solution of the limiting equation \eqref{Eq_LimitingADHM1kSW}.
Indeed, given $(\bxi, \Psi)$ and one such connection, all other choices of $A$ are parametrized by choices of spin connections $A_j$ on $\fw_j$, for every $j$, and the spaces of these spin connections are infinite-dimensional.
However, we are only interested in those solutions of \eqref{Eq_LimitingADHM1kSW} which are obtained as limits of rescaled ADHM$_{1,k}$ monopoles.
To determine further constraints for such limits, let $(\Psi_0=0,\bxi_0,A_0)$ be a solution of \eqref{Eq_LimitingADHM1kSW} with $\tilde n \in \Gamma(\Sym^k_\lambda V)$ for some partition $\lambda$ of $k$,
and suppose that
\begin{align*}
  \Psi_\epsilon 
  =
  \sum_{i = 1}^\infty \epsilon^i\Psi_i, \quad
  \bxi_\epsilon
  =
  \sum_{i = 0}^\infty \epsilon^i\bxi_i, \qandq
  A_\epsilon
  =
  A_0 + \sum_{i = i}^\infty \epsilon^i a_i
\end{align*}
is a formal power series solution of the rescaled ADHM$_{1,k}$ Seiberg--Witten equation:
\begin{equation}
  \label{Eq_FormalExpansionADHM1kSW}
  \begin{split}
    \slD_{A_\epsilon} \Psi_\epsilon
    &=
    0, \\
    \slD_{A_\epsilon,C} \bxi_\epsilon
    &=
    0, \qand \\ 
    \epsilon^2\varpi F_{A_\epsilon}
    &=
    \mu(\Psi_\epsilon) + \mu(\bxi_\epsilon).
  \end{split}
\end{equation}
Moreover, we can assume the gauge fixing condition $\bxi_1 \perp \rho(\fg_P)\bxi_0$, that is,
\begin{equation*}
  R_{\bxi_0}^*\bxi_1 = 0
\end{equation*}
in the notation of \autoref{Prop_LinearizedMomentMapAtZero}.
The next proposition imposes constraints on the terms of order $\epsilon$ in the power series expansions.

Let $W_j$ and $\fg_{\sH_j}$ be, respectively, the spinor bundle and  adjoint bundle associated with the spin$^{\U(k_j)}$ structure $\fw_j$ on the total space of the covering map $\pi_j \co \tilde M_j \to M$.

\begin{prop}
  In the above situation, there exist $\tilde\Psi_{1,j} \in \Gamma(W_j)$ and $\tilde \bxi_j \in \Gamma(V\otimes \fg_{\sH_j})$ such that
  \begin{equation}
    \label{Eq_SpinorDecomposition}
    \Psi_1 = \bigoplus_{j=1}^m (\pi_j)_*\tilde\Psi_{1,j}
    \qandq
    \bxi_1 = \bigoplus_{j=1}^m (\pi_j)_*\tilde\bxi_{1,j}.
  \end{equation}
  Furthermore, $A_0$ arises from a collection of spin connections $A_{0,j} \in \sA^s(\fw_j)$, and each triple $(A_{0,j}, \bxi_{1,j}, \tilde\Psi_{1,j})$ satisfies the ADHM$_{1,k_j}$ equation
  \begin{equation}
    \label{Eq_SWForComponents}
    \begin{split}
      \slD_{A_{0,j}} \tilde\Psi_{1,j} &= 0, \\
      \slD_{A_{0,j},C} \bxi_{1,j} &= 0, \qand \\
      \varpi F_{A_{0,j}} &= \mu(\tilde\Psi_{1,j}) + \mu(\tilde\bxi_{1,j})
    \end{split}
  \end{equation}
  on $\tilde M_j $ for $j = 1, \ldots, m$.
\end{prop}

\begin{proof}
  From \autoref{Prop_ADHM1kHaydysTechnical}, we know that
  \begin{equation*}
    \bxi_0 = (\bxi_{0,1},\cdots,\bxi_{0,m}) \in \Gamma(V\otimes\ft_P)
    \qwithq
    \ft_P = \bigoplus_{j=1}^m (\pi_j)_*\fg_{\sH_j}
  \end{equation*}
  and $A_0$ arises from spin connections $A_{0,j} \in \sA^s(\fw_j)$.

  The coefficient in front of $\epsilon$ on the right-hand side of the third equation of \eqref{Eq_FormalExpansionADHM1kSW} must vanish;
  hence,
  \begin{equation*}
    (\rd_{\bxi_0}\mu)\bxi_1 = 0.
  \end{equation*}
  By \autoref{Prop_LinearizedMomentMapAtZero} it follows that $[\bxi_0\wedge\bxi_1] = 0$.
  Therefore,
  \begin{equation*}
    \mu(\bxi_1) \in \Omega^2(M,[\ft_P,\ft_P])
  \end{equation*}
  by the following self-evident observation combined with \autoref{Thm_1kADHM}.

  \begin{prop}
    If $\bxi_0,\bxi_1 \in \H\otimes\fg$, $[\bxi_0\wedge \bxi_1] = 0$, and the stabilizer of $\bxi_0 \in \U(k)$ is precisely $T_\lambda = \prod_{n=1}^{\abs{\lambda}} \U(\lambda_n)$,
    then $\bxi_1 \in \H\otimes\ft_\lambda$ with $\ft_\lambda = \bigoplus_{n=1}^{\abs{\lambda}} \fu(\lambda_n)$.
    In particular,
    \begin{equation*}
      [\bxi_1\wedge\bxi_1] \in \H\otimes[\ft_\lambda,\ft_\lambda] \subset \H\otimes\ft_\lambda.
    \end{equation*}
  \end{prop}

  \begin{remark}
    If $\lambda = (1,\ldots,1)$, then $[\ft_P,\ft_P] = 0$;
    cf.~\autoref{Rmk_ADHM11SW}.
  \end{remark}

  The third equation in \eqref{Eq_FormalExpansionADHM1kSW} to order $\epsilon^2$ is thus equivalent to
  \begin{equation}
    \label{Eq_Epsilon23FormalExpansionADHM1kSW}
    \varpi F_{A_0} = \mu(\bxi_1) + (\rd_{\bxi_0}\mu)\bxi_2 + \mu(\Psi_1).
  \end{equation}

  In terms of the spin connections $A_{0,j} \in \sA^s(\fw_j)$,
  we have
  \begin{equation*}
    \varpi F_{A_0} = \bigoplus_{j=1}^m (\pi_j)_*\varpi F_{A_{0,j}} \in \Omega^2(M,\ft_P).
  \end{equation*}
  By \eqref{Eq_DMUPhiOnST},
  we have
  \begin{equation*}
    (\rd_{\bxi_0}\mu)\bxi_2 \in \Omega^2(M,\ft_P^\perp).
  \end{equation*}
  Thus, if we denote by $\mu_\shortparallel(\Psi_1)$ the component of $\mu(\Psi_1)$ in $\ft_P$ and by $\mu_\perp(\Psi_1)$ the component of $\mu(\Psi_1)$ in $\ft_P^\perp \subset \fg_P$,
  then \eqref{Eq_Epsilon23FormalExpansionADHM1kSW} is equivalent to
  \begin{equation}
    \label{Eq_LeadingFormalExpansionADHM1kSW_1}
    \begin{split}
      \varpi F_{A_0} &= \mu(\bxi_1) + \mu_\shortparallel(\Psi_1) \qand \\
      (\rd_{\bxi_0}\mu)\bxi_2 &= -\mu_\perp(\Psi_1).
    \end{split}
  \end{equation}
  Since $\ft_P$ is parallel with respect to $A_0$ and $V\otimes\ft_P$ is perpendicular to $\bar\gamma(T^*M\otimes\fg_P)\bxi_0$,
  the first and the second equation of \eqref{Eq_FormalExpansionADHM1kSW} to order $\epsilon$ are equivalent to
  \begin{equation}
    \label{Eq_LeadingFormalExpansionADHM1kSW_2}
    \begin{split}
      \slD_{A_0} \Psi_1 &= 0, \\
      \slD_{A_0,C} \bxi_1 &= 0, \qand \\
      \gamma(a_1)\bxi_0 &= 0.
    \end{split}
  \end{equation}

  Let $\tilde\Psi_{1,j} \in \Gamma(W_j)$ and $\tilde \bxi_j \in \Gamma(V\otimes \fg_{\sH_j})$ be such that \eqref{Eq_SpinorDecomposition} holds.
  The first equation of \eqref{Eq_LeadingFormalExpansionADHM1kSW_1} and the first two equations of \eqref{Eq_LeadingFormalExpansionADHM1kSW_2} are precisely equivalent to the ADHM$_{1,k_j}$ Seiberg--Witten equation \eqref{Eq_SWForComponents} for the triple $(A_{0,j}, \bxi_{1,j}, \tilde\Psi_{1,j})$.
\end{proof}

\subsection{A compactness conjecture for \texorpdfstring{ADHM$_{1,k}$}{ADHM1k} monopoles}

The discussion in the preceding sections together with known compactness results for Seiberg--Witten equations
\cite{Taubes2012,Taubes2013,Haydys2014,Taubes2016,Taubes2017} lead to the following conjecture.

\begin{conjecture}
  \label{Conj_ADHM1kCompactness}
  Let $(\epsilon_i,\Psi_i,\bxi_i,A_i)$ be a sequence of solutions of the blown-up ADHM$_{1,k}$ Seiberg--Witten equation
  \begin{align*}
    \slD_{A_i} \Psi_i &= 0, \\
    \slD_{A_i,C} \bxi_i &= 0, \\
    \epsilon_i^2 \varpi F_{A_i} &= \mu(\Psi_i) + \mu(\bxi_i), \qand \\
    \Abs{(\Psi_i,\bxi_i)}_{L^2} &= 1
  \end{align*}
  with $\epsilon_i \to 0$.
  After passing to a subsequence the following hold:
  \begin{enumerate}
  \item
    There is a closed subset $Z \subset M$ of Hausdorff dimension at most one, such that outside of $Z$ and up to gauge transformations $(\Psi_i,\bxi_i,A_i)$ converges to a limit $(0,\bxi^\infty_0,A^\infty_0)$ and $\epsilon_i^{-1}(\Psi_i,\bxi_i-\bxi^\infty_0)$ converges to a limit $(\Psi^\infty_1,\bxi^\infty_1)$.
  \item
    The triple $(0,\bxi^\infty,A^\infty)$ is a solution of the limiting ADHM$_{1,k}$ Seiberg--Witten equation \eqref{Eq_LimitingADHM1kSW}.
  \item
    There is a section $\tilde n \in \Gamma(M\setminus Z,\Sym^k_\lambda V)$ for some partition $\lambda$ of $k$ induced by $\bxi^\infty_0$.
    The section $\tilde n$ extends to to a continuous section of $\Sym^k V$ on all of $M$.
  \item
    Denote by $\tilde M\setminus\tilde Z$ the unbranched cover of $M\setminus Z$ induced by $\tilde n$.
    If $k_j$, $\tilde M_j\setminus\tilde Z_j$, $\fw_j$ are as in \autoref{Prop_ADHM1kHaydys} and $A_{0,j} \in \sA^s(\fw_j)$ denote the spin connections giving rise to $A^\infty_0$, and $\tilde \Psi_{1,j}$ and $\tilde \bxi_{1,j}$ are such that
    \begin{equation*}
      \Psi^\infty_1 = \bigoplus_{j=1}^m (\pi_j)_*\tilde\Psi_{1,j}
      \qandq
      \bxi^\infty_1 = \bigoplus_{j=1}^m (\pi_j)_*\tilde\bxi_{1,j},
    \end{equation*}
    then, for each $j = 1,\ldots,m$,
    $(\tilde \Psi_{1,j},\tilde \bxi_{1,j}, A_{0,j})$ is a solution of the ADHM$_{1,k_j}$ Seiberg--Witten equation on $\tilde M_j\setminus\tilde Z_j$.
  \end{enumerate}
\end{conjecture}

\begin{remark}
  The reader should observe that while $\tilde M_j$ in $\tilde M_j\setminus\tilde Z_j$ does exist, it need not be a smooth manifold.
\end{remark}

\begin{remark}
  If $\Psi = 0$, $V = TM \oplus \R$ and $(a,\xi) \in \Omega^1(M,\fg_\sH) \oplus \Omega^0(M,\fg_\sH) = \Gamma(V\otimes\fg_\sH)$,
  then the ADHM$_{1,k}$ Seiberg--Witten equation becomes the equation
  \begin{equation}
    \label{Eq_ExtendedGLkCFlat}
    \begin{split}
      F_{A+ia} - *[\xi,a] + *i\rd_A\xi &= 0 \qand \\
      \rd_A^* a &= 0
    \end{split}
  \end{equation}
  with
  \begin{equation*}
    F_{A+ia} = F_A - \frac12[a\wedge a] + i\rd_Aa.
  \end{equation*}
  If $(a,\xi,A)$ is a solution of \eqref{Eq_ExtendedGLkCFlat} and $M$ is closed,
  then a simple integration by parts argument shows that $\rd_A\xi = 0$;
  hence, $F_{A+ia} = 0$.
  That is, \eqref{Eq_ExtendedGLkCFlat} is effectively the condition that condition that $A + i a$ is a flat $\GL_k(\C)$--connection together with the moment map equation $\rd_A^*a = 0$.
  
  \autoref{Conj_ADHM1kCompactness} thus predicts that as limits of flat $\GL_k(\C)$--connections we should see data consisting of a closed subset $Z \subset M$ of Hausdorff dimension at most one, $m \in \N$ and, for each $j = 1,\ldots,m$, a $\ell_j$--fold cover $\tilde M\setminus \tilde Z_j$ of $M\setminus Z$, and solutions of \eqref{Eq_ExtendedGLkCFlat} on $\tilde M_j\setminus\tilde Z_j$ such that $\sum_{j=1}^m \ell_jk_j = k$.
\end{remark}


\section{A tentative proposal}
\label{Sec_TentativeProposal}

We are ready to outline how ADHM monopoles can be used to deal with the problem of multiple covers described in \autoref{Sec_MultipleCoverAssociatives}.

Let $\psi$ be a tamed, closed, definite $4$--form,
let $P$ be  a compact, connected, oriented $3$--manifold,
let $P \subset Y$ be an unobstructed associative embedding.
Set
\begin{equation*}
  \fM^{1,k}(P,\psi)
  \coloneq
  \coprod_{\fw}
  \fM^{1,k}_\fw(P,\psi)
\end{equation*}
with the disjoint union taken over all spin$^{\U(k)}$ structures $\fw$ on $P$ and
\begin{equation*}
  \fM^{1,k}_\fw(P,\psi)
  \coloneq
  \frac{
    \set*{
      (\Psi,\bxi,A) \in \Gamma(W)\times\Gamma(NP\otimes\fg_\sH)\times\sA^s(\fw)
      :
      \begin{array}{@{}l@{}}
        (\Psi,\bxi,A) \text{ satisfies }\eqref{Eq_ADHMrkSW} \\
        \text{ with respect to } g_\psi|_P
      \end{array}
    }
  }{\sG^s(\fw)}.
\end{equation*}
Ignoring issues to do with reducible solutions,
one should be able to extract a number
\begin{equation*}
  w(kP,\psi) \in \Z
\end{equation*}
by counting $\fM^{1,k}(P,\psi)$, at least, for generic $\psi$ and possibly after slightly perturbing the ADHM Seiberg--Witten equation \eqref{Eq_ADHMrkSW}.
More generally, if $P$ has connected components $P^1,\ldots,P^m$ and $k_1,\ldots,k_m \in \N$, we set
\begin{equation*}
  w(k_1\cdot P^1 + \cdots + k_m\cdot P^m,\psi)
  \coloneq \prod_{j=1}^k w(k_j\cdot P^j,\psi).
\end{equation*}

For $k = 1$, this number is the Seiberg--Witten invariant $\SW(P) \in \Z$ mentioned in \autoref{Sec_RoleOfSeibergWitten}
For $k > 0$, this number should be independent of the choice of perturbation but it will depend on $\psi$.
Assume the situation of \autoref{Sec_CollapsingImmersions};
that is, we have:
\begin{itemize}
\item
  a generic $1$--parameter family of tamed, closed, definite $4$--forms $(\psi_t)_{t\in(-T,T)}$,
\item
  a $1$--parameter family of compact, connected, unobstructed embedded associative submanifolds $(P_t)_{t \in (-T,T)}$ with respect to $(\psi_t)_{t \in (-T,T)}$, and
\item
  for every $j = 1, \ldots, m$ a $1$--parameter family of compact, connected, unobstructed embedded associative submanifolds $(P^j_t)_{t \in (-T,0)}$ with respect to $(\psi_t)_{t \in (-T,0)}$ such that
  \begin{equation*}
    P^j_t \to \ell_j\cdot P_0
  \end{equation*}
  as integral currents as $t$ tends to zero for some $\ell_j \in \set{ 2, 3, \ldots }$.
\end{itemize}
Given $k_1,\ldots,k_m$, set
\begin{equation*}
  k \coloneq \sum_{j=1}^m \ell_jk_j.
\end{equation*}
From the discussion in preceding three sections we expect that,
for $0 < t \ll 1$,
\begin{equation}
  \label{Eq_MultipleCoverTransitionFormula}
  w(k\cdot P_{-t},\psi_{-t})
  + w(k_1\cdot P^1_{-t} + \cdots + k_m\cdot P^m_{-t},\psi_{-t})
  =
  w(k\cdot P_{+t},\psi_{+t})
\end{equation}
because \autoref{Conj_ADHM1kCompactness} suggests that as $t$ passes through zero $w(k_1\cdot P^1_0 + \cdots + k_m\cdot P^1_0,\psi_0)$ ADHM$_{1,k}$ monopoles on $P_t$ degenerate and disappear (if counted with the correct sign).

Suppose that one can indeed define a weight $w$ as above satisfying \eqref{Eq_MultipleCoverTransitionFormula} as well as analogues of \eqref{Eq_SurgeryBehaviour}.
Define
\begin{equation}
  \label{Eq_NDPhi_5thAttempt}
  n_\beta(\psi) = \sum w(k_1\cdot P^1 + \cdots + k_m\cdot P^m,\psi)
\end{equation}
with the summation ranging over all $m \in \N$, $k_1,\ldots,k_m \in \N$ and all compact, connected, unobstructed embedded associative submanifolds $P^1,\ldots,P^m \subset Y$ such that
\begin{equation*}
  \sum_{j=1}^m k_j [P^j] = \beta.
\end{equation*}
This number would be invariant under the transitions described in \autoref{Sec_IntersectingAssociatives}, \autoref{Sec_AssociatviesWithT2Singularities}, and \autoref{Sec_CollapsingImmersions}.

From \autoref{Sec_RoleOfSeibergWitten} we know that reducible solutions will prevent us from defining $w$ in general.
However, the above can serve as a first approximation.
To deal with reducibles one likely has to develop ADHM$_{1,k}$ analogues of \citeauthor{Kronheimer2007}'s monopole homology and construct a chain complex extending \eqref{Eq_CMA} which does depend on $\psi$ but whose homology does not.

\begin{remark}
  By analogy with monopole Floer homology, one can envision also a corresponding $8$--dimensional version of the invariant proposed in this article. 
  Such an invariant would be obtained by counting Cayley submanifolds inside a closed $\Spin(7)$--manifold, weighted by solutions of the $4$--dimensional ADHM Seiberg--Witten equations.
  A relative version of this theory would associate with every cylindrical $\Spin(7)$--manifold $X$ whose end is asymptotic to a compact $\Gtwo$--manifold $Y$ a distinguished element of the Floer homology group associated with $Y$.  
  In order to develop such a $7+1$ dimensional theory, one has to deal with higher-dimensional moduli spaces of Cayley submanifolds and ADHM monopoles, which poses additional technical complications.
  Note that in order to define $\Gtwo$ Floer homology, one has to consider only $\Spin(7)$--manifolds of the form $Y \times (-\infty, \infty)$, and only zero-dimensional moduli spaces.
\end{remark}


\section{Counting holomorphic curves in Calabi--Yau \texorpdfstring{$3$}{3}--folds}
\label{Sec_ADHMBundles}

Let $Z$ be a Calabi--Yau $3$--fold with Kähler form $\omega$ and holomorphic volume form $\Omega$. 
The product $S^1\times Z$ is naturally a $\Gtwo$--manifold with the $\Gtwo$--structure given by
\begin{equation*}
  \phi = dt \wedge\omega + \Re\Omega.
\end{equation*}
Every holomorphic curve $\Sigma\subset Z$ gives rise to an associative submanifold $S^1\times \Sigma \subset S^1 \times Z$.

\begin{prop}
  Let $\beta\in H_2(Z)$ be a homology class.
  Every associative submanifold in $S^1 \times Z$ representing the class  $[S^1]\times\beta$ is necessarily of the form $S^1\times \Sigma$ with $\Sigma\subset Z$ a holomorphic curve.
\end{prop}

\begin{proof}
  The argument is similar to the one used to prove an analogous statement for instantons \cite[Section 3.2]{Lewis1998}.
  Let $P \subset S^1\times Z$ be an associative submanifold representing $[S^1]\times\beta$.
  Since
  \begin{equation*}
    \phi|_P = (\dt\wedge\omega + \Re\Omega)|_P = \vol_P,
  \end{equation*}
  there is a smooth function $f$ on $P$ such that
  \begin{equation*}
    \dt\wedge\omega|_P = f \vol_P \qandq
    \Re\Omega|_P = (1-f)\vol_P
  \end{equation*}
  By Wirtinger's inequality \cite{Wirtinger1936}, $f \leq 1$.
  We need to prove that $f = 1$,
  since this implies that $\partial_t$ is tangent to $P$ and, therefore, $P$ is of the form $S^1 \times \Sigma$, with $\Sigma \subset Z$ calibrated by $\omega$.
  
  One the one hand we have
  \begin{equation*}
    \int_P \vol_P
    = \inner{[\phi]}{[P]}
    = \inner{[\rd t\wedge\omega] + [\Re\omega]}{[S^1]\times\beta}
    = \inner{[\rd t\wedge\omega]}{[S^1]\times\beta},
  \end{equation*}
  while on the other hand
  \begin{equation*}
    \int_P f\vol_P
    = \int_P \rd t\wedge \omega
    = \inner{[\rd t\wedge\omega]}{[P]}
    = \inner{[\rd t\wedge\omega]}{[S^1]\times\beta}.
  \end{equation*}
  It follows that $f$ has mean-value $1$ and thus $f = 1$ because $f \leq 1$.
\end{proof}

The deformation theory of the associative submanifold $S^1\times \Sigma$ in $S^1\times Z$ coincides with that of the holomorphic curve $\Sigma$ in $Y$ \cite[Lemma 5.11]{Corti2012a}.
In particular, the putative enumerative theory for associative submanifolds discussed in this paper should give rise to an enumerative theory for holomorphic curves in Calabi--Yau $3$--folds.
Algebraic geometry abounds in such theories and various interplays between them; see \cite{Pandharipande2014} for an introduction to this rich subject. 
Our approach is closer in spirit to the original proposal by Donaldson and Thomas \cite{Donaldson1998}. 
We will argue that it should lead to a symplectic analogue of a theory already known to algebraic geometers.

\subsection{The Seiberg--Witten invariants of Riemann surfaces}
\label{Sec_DimensionalReduction}

In the naive approach of \autoref{Sec_RoleOfSeibergWitten} each associative submanifold is counted with its total Seiberg--Witten invariant.
The Seiberg--Witten equation \eqref{Eq_ClassicalSeibergWitten} over the $3$--manifold $M=S^1\times \Sigma$ was studied extensively  \cite{Morgan1996a,Mrowka1997,Munoz2005}.
The equation admits irreducible solutions only for the spin$^c$--structures pulled-back from $\Sigma$.
Such a spin$^c$ structure corresponds to a Hermitian line bundle $L\to \Sigma$;
the induced spinor bundle is $W = L\oplus T^*\Sigma^{0,1} \otimes L$.
Up to gauge transformations,
all irreducible solutions of the Seiberg--Witten equation are pulled-back from triples $(A,\psi_1,\psi_2)$ on $\Sigma$, where $(\psi_1,\bar\psi_2)\in\Gamma(L)\oplus\Omega^{0,1}(\Sigma,L)$, $A \in \sA(\det(W))$ and
\begin{equation}
  \label{Eq_Vortex}
  \begin{split}
    \delbar_A \psi_1 =0, \quad
    \delbar_A^*\bar\psi_2 &=0, \\
    \inner{\psi_1}{\bar\psi_2} &=0, \qand \\
    \frac{i}2 *F_A + \abs{\psi_1}^2- \abs{\bar\psi_2}^2 &=0. 
  \end{split}
\end{equation}
Here $\inner{\psi_1}{\bar\psi_2}$ is the $(0,1)$--form obtained from pairing $\psi_1$ and $\bar\psi_2$ using the Hermitian inner product.

The second equation implies that either $\psi_1$ or $\bar\psi_2$ must vanish identically---which one, depends on the sign of the degree
\begin{equation*}
  2d \coloneq \inner{c_1(W)}{\Sigma}.
\end{equation*}
Since $\det(W) = L^2\otimes K_\Sigma^{-1}$,
we have
\begin{equation*}
  \deg(L) = g-1+d.
\end{equation*}

Suppose that $d < 0$.
It follows from integrating the third equation that $\psi_1\neq 0$ and so $\bar\psi_2=0$.
The pair $(A,\psi_1)$ corresponds to an effective divisor of degree $g-1+d$ on $\Sigma$:
the zero set of $\psi_1$ counted with multiplicities.
This corresponds to an element of the symmetric product $\Sym^{g-1+d}\Sigma$.
If $d > 0$,
then a similar argument and Serre duality associates with every solution of \eqref{Eq_Vortex} an element of $\Sym^{g-1-d}\Sigma$.
The above correspondence, in fact, goes both ways:

\begin{theorem}[\citet{Noguchi1987,Bradlow1990,GarciaPrada1993}]
  \label{Thm_VortexCorrespondence}
  Let $\lambda \in \R\setminus\set{d}$.
  The moduli space of solutions to the perturbed vortex equation
  \begin{equation}
    \label{Eq_PerturbedVortex}
    \begin{split}
      \delbar_A \psi_1 =0, \quad
      \delbar_A^*\bar\psi_2 &=0, \\
      \inner{\psi_1}{\bar\psi_2} &=0, \qand \\
      \frac{i}2 *F_A + \abs{\psi_1}^2- \abs{\bar\psi_2}^2 &= \frac{2\pi}{\vol(\Sigma)}\cdot\lambda
    \end{split}   
  \end{equation}
  is homeomorphic to
  \begin{equation*}
    \begin{cases}
      \Sym^{g-1+d}(\Sigma) & \quad\text{if } d - \lambda < 0 \qand \\
      \Sym^{g-1-d}(\Sigma) & \quad\text{if } d - \lambda > 0.      
    \end{cases}
  \end{equation*}
\end{theorem}

The Seiberg--Witten invariant can be obtained by integrating the Euler class of the obstruction bundle, in this case the cotangent bundle, over the moduli space.
As a consequence, if $\Sigma \neq S^2$,
then the total Seiberg--Witten invariant is
\begin{equation*}
  \SW(S^1\times \Sigma)
  = \sum_{d \in \Z}  (-1)^{g-1+d}\chi(\Sym^{g-1+d}\Sigma). 
\end{equation*}
Here we can sum over all $d \in \Z$ since for $\abs{d} > g-1$ we have $\chi(\Sym^{g-1+d}\Sigma) = 0$.

\subsection{Rational curves and the Meng--Taubes invariant}
\label{Sec_MengTaubes}

For $\Sigma = S^2$,
the above series is not summable.
This is consistent with the general theory alluded to in \autoref{Sec_RoleOfSeibergWitten}: we have $b_1(S^1\times S^2) = 1$ and, due to the appearance of reducible solutions, the total Seiberg--Witten invariant is defined only for $3$--manifolds with $b_1 \geq 2$.
In full generality, this problem can be solved within the framework of Floer homology.
However, if one considers only closed, oriented $3$--manifolds with $b_1\geq 1$ there is also a middle ground approach due to Meng and Taubes \cite{Meng1996}.
For every such a $3$--manifold $M$ they define an invariant
\begin{equation*}
  \uSW(M) \in \Z\llbracket H\rrbracket/H.
\end{equation*}
Here $H$ is the torsion-free part of $H^2(M,\Z)$, $\Z\llbracket H \rrbracket$ is the set of $\Z$--valued functions on $H$, and $H$ acts on $\Z\llbracket H \rrbracket$ by pull-back.

The Meng--Taubes invariant takes a particularly simple form for $M = S^1\times \Sigma$. 
In this case, there is a distinguished spin$^c$ structure, corresponding to the line bundle $L$ being trivial, and the invariant can be naturally lifted to an element $\uSW(M) \in\Z\llbracket H \rrbracket$. 
Moreover, the support of $\uSW(M) $ is $\Z = H^2(\Sigma,\Z) \subset H$, reflecting the fact that the Seiberg--Witten equation has solutions only for the spin$^c$ structures pulled-back from $\Sigma$.
Thus, $\uSW(M)$ can be interpreted as an element of the ring of formal Laurent series in a single variable, $q$ say,
\begin{equation*}
  \uSW(M) \in \Z(\!(q)\!).
\end{equation*}
For $g\geq 1$, this is the Laurent polynomial whose coefficients are the Seiberg--Witten invariants:
\begin{equation}
  \label{Eq_SWInvariantOfSigma}
  \uSW(S^1\times \Sigma) = \sum_{d \in \Z} (-1)^{g-1+d}\chi(\Sym^{g-1+d}\Sigma)q^d
\end{equation}
and we see that $\SW(S^1\times \Sigma)$ is obtained by evaluating $\uSW(S^1\times \Sigma)$ at $q=1$.
It is easy to see from the definition of the Meng--Taubes invariant that the same formula is true for $\Sigma = S^2$, although now the series has infinitely many non-zero terms.
One cannot evaluate $\uSW(S^1\times S^2)$ at $q=1$ and is forced to work with the refined invariant.

\subsection{Stable pair invariants of Calabi--Yau \texorpdfstring{$3$}{3}--folds}
\label{Sec_StablePairs}

Pandharipande and Thomas introduced a numerical invariant counting holomorphic curves in Calabi--Yau $3$--folds together with points on them; see \cite[Section $4 \frac{1}{2}$]{Pandharipande2014} for a brief introduction and \cite{Pandharipande2009,Pandharipande2010} for more technical accounts. 
Since the space of curves and points on them is not necessarily compact, one considers the larger moduli space of \defined{stable pairs}, consisting of a coherent sheaf $F$ on $Z$ together with a section $s\in H^0(Z,F)$ which, thought of as a sheaf morphism $s\colon \sO_Z \to F$, is surjective outside a zero-dimensional subset of $Z$.
The sheaf is required to be supported on a (possibly singular and thickened) holomorphic curve $\Sigma \subset Z$.%
\footnote{%
  More precisely, $F$ is pure of dimension one and $s$ has zero-dimensional cokernel.
}
\begin{example}
  \label{Ex_StablePair}
  The simplest examples arise when $\Sigma$ is smooth and $(F,s)$ is the pushforward of a pair $(\sL,\psi)$ on $\Sigma$ consisting of a holomorphic line bundle and a non-zero section.
  Conversely, all stable pairs whose support is a smooth, unobstructed curve are of this form \cite[Section 4.2]{Pandharipande2009}.
\end{example}
The topological invariants of a stable pair are the homology class $[\Sigma] \in H_2(Z)$ and the Euler characteristic $\chi(X,F) \in \Z$. 
For instance, in \autoref{Ex_StablePair}, with $\Sigma$ of genus $g$, we have
\begin{equation}
  \label{Eqn_EulerChar}
  \chi(X,F) = 1-g+\deg(F).
\end{equation}
For every $\beta \in H_2(Z)$ and $d\in \Z$, Pandharipande and Thomas use virtual fundamental class techniques to define an integer $\PT_{d,\beta}$ which counts stable pairs with homology class $\beta$ and Euler characteristic $d$. 
These numbers for different values of $d$ can be conveniently packaged into the generating function
\begin{equation*}
  \PT_{\beta} = \sum_d \PT_{\beta,d} q^d.
\end{equation*}
For a holomorphic curve $\Sigma\subset Z$ with $[\Sigma]=\beta$, denote by $\PT_\Sigma(q)$ the contribution to $\PT_{\beta}(q)$ coming from stable pairs whose support is $\Sigma$. (It makes sense to talk about such a contribution even for non-isolated curves \cite[Section 3.1]{Pandharipande2010}.)

In the situation of \autoref{Ex_StablePair}, the moduli space of stable pairs with support on $\Sigma$ and Euler characteristic $d$ is simply the space of effective divisors whose degree, computed using \eqref{Eqn_EulerChar}, is $g-1+d$. 
From the deformation theory of such stable pairs one concludes that in this case, 
\begin{equation}
  \label{Eq_CurveContribution}
  \PT_\Sigma(q) = \sum_{d} (-1)^{g-1+d} \chi(\Sym^{g-1+d}\Sigma)q^d;
\end{equation}
see \cite[Equation (4.4)]{Pandharipande2009} for details.
As a result, we obtain the following.

\begin{prop}
  If $\Sigma\subset Z$ is a smooth, unobstructed holomorphic curve, then 
  \begin{equation*}
    \PT_\Sigma = \uSW(S^1\times \Sigma).
  \end{equation*} 
\end{prop}

\begin{remark}
  From the $3$--dimensional perspective,
  the symmetry between $d$ and $-d$ is a special case of the involution in Seiberg--Witten theory induced from the involution on the space of spin$^c$ structures \cite[Section 6.8]{Morgan1996};
  from the $2$--dimensional viewpoint, it is a manifestation of the Serre duality between $H^1(\sL)$ and $H^0(K_\Sigma\otimes \sL^*)$.
\end{remark}

\begin{remark}
  The fact that the stable pair invariant is partitioned into an integers worth of invariants corresponding to the degrees of the spin$^c$ structures on curves suggests that something similar could be true for associative submanifolds.
  However, unlike in the dimensionally reduced setting, where a spin$^c$ corresponds in a natural way to an integer, for two distinct associatives $P_1$ and $P_2$ we are not aware of any way to relate the spin$^c$ structures on them.
\end{remark}

In general, the stable pair invariant includes also more complicated contributions from singular and obstructed curves representing the given homology class. 
For irreducible classes, Pandharipande and Thomas proved that such a contribution is a finite sum of Laurent series of the form \eqref{Eq_CurveContribution} \cite[Theorem 3 and Section 3]{Pandharipande2010}.

\subsection{ADHM bundles over Riemann surfaces}

The stable pair invariant includes also contributions from thickened curves.
If a homology class $\beta \in H_2(Z,\Z)$ is divisible by $k$ and $\beta/k$ is represented by a holomorphic curve $\Sigma\subset Z$, then there exist stable pairs  having $k\Sigma$ as their support. 
Thinking of $S^1\times k\Sigma$ as a multiple cover of the associative $S^1\times \Sigma$ in $S^1 \times Z$, we are led by the  discussion of \autoref{Sec_EmbeddingsWithMultiplicity} to the conclusion that the contribution of such a thickened curve should be in some way related to the solutions of the ADHM$_{1,k}$ Seiberg--Witten equation on the $3$--manifold $S^1\times \Sigma$.
We will argue that this is indeed the case.

Consider the more general ADHM$_{r,k}$ Seiberg--Witten equation introduced in \autoref{Sec_ADHMrkSeibergWittenEquations}
under the following assumptions:
\begin{hyp} 
  \label{Hyp_DimensionalReduction}
  Let $\Sigma$ be a closed Riemann surface and $M=S^1\times \Sigma$ with the geometric data as in \autoref{Def_ADHMGeometricData} such that:
  \begin{enumerate}
  \item $g$ is a product Riemannian metric,
  \item $E$ and the connection $B$ are pulled-back from $\Sigma$, and
  \item $V$ and the connection $C$ are pulled-back from a $\U(2)$--bundle with a connection on $\Sigma$ such that $\Lambda^2_{\C} V \cong K_\Sigma$ as bundles with connections.
  \end{enumerate}
\end{hyp}

\begin{prop}
  \label{Prop_PulledBack}
  If \autoref{Hyp_DimensionalReduction} holds and $(\Psi,\xi,A)$ is an irreducible solution of the ADHM$_{r,k}$ Seiberg--Witten equation \autoref{Eq_ADHMrkSW}, then the spin$^{\U(k)}$ structure $\fw$ is pulled-back from a spin$^{\U(k)}$ structure on $\Sigma$ and $(\Psi,\xi,A)$ is gauge-equivalent to a configuration pulled-back from $\Sigma$, unique up to gauge equivalence on $\Sigma$.
\end{prop}

This is a special case of \cite[Theorem 3.8]{Doan2017}. 
In the situation of \autoref{Prop_PulledBack},
\autoref{Eq_ADHMrkSW} reduces to a non-abelian vortex equation on $\Sigma$.
Recall that a choice of a spin$^{\U(k)}$ structure on $\Sigma$ is equivalent to a choice of a $\U(k)$--bundle $H \to \Sigma$. 
Consequently, $A$ can be seen as a connection on $H$.
The corresponding spinor bundles are  
\begin{equation*}
  \fg_{H} = \fu(H)
  \qandq
  W = H\oplus T^*\Sigma^{0,1}\otimes H.
\end{equation*}

\begin{prop}
  \label{Prop_DimensionalReduction}
  Let $(A,\Psi,\xi)$ be a configuration pulled-back from $\Sigma$.
  Under the splitting $W = H\oplus T^*\Sigma^{0,1}\otimes H$ we have $\Psi = (\psi_1, \psi_2^*)$ where
  \begin{align*}
    \psi_1 &\in \Gamma(\Sigma,\Hom(E,H) ), \\
    \psi_2 &\in \Omega^{1,0}(\Sigma,\Hom(H,E)), \qand \\
    \xi &\in \Gamma(\Sigma,V\otimes \End(H)).
  \end{align*} 
  Equation \eqref{Eq_ADHMrkSW} for $(A,\Psi,\xi)$ is equivalent to
  \begin{equation}
    \label{Eq_ADHMSWReduced}
    \begin{split}
      \delbar_{A,B}\psi_1 = 0, \quad \delbar_{A,B}\psi_2 = 0, \quad \delbar_{A,C}\xi &= 0, \\
      [\xi\wedge\xi] + \psi_1\psi_2 &= 0, \qand \\
      i * F_A + [\xi\wedge\xi^*] + \psi_1\psi_1^* - *\psi_2^*\psi_2 &=0.
    \end{split}
  \end{equation}
  In the second equation we use the isomorphism $\Lambda_{\C}^2 V \cong K_\Sigma$ so that the left-hand side is a section of $\Omega^{1,0}(\Sigma,\End(H))$.
  In the third equation we contract $V$ with $V^*$ so that the left-hand side is a section of $i\fu(H)$. 
\end{prop}

This follows from \cite[Proposition 3.6, Remark 3.7]{Doan2017} and the complex description \eqref{Eq_ComplexMomentMap} of the hyperkähler moment map appearing in the ADHM construction.

We can also perturb \eqref{Eq_ADHMSWReduced}  by  $\tau\in\R$ and $\theta\in H^0(\Sigma,K_\Sigma)$:
\begin{equation}
  \label{Eq_ADHMSWReducedPerturbed}
  \begin{split}
    \delbar_{A,B}\psi_1 = 0, \quad \delbar_{A,B}\psi_2 = 0, \quad \delbar_{A,C}\xi &= 0, \\
    [\xi\wedge\xi] + \psi_1\psi_2 &= \theta\otimes\id, \qand \\
    i * F_A + [\xi\wedge\xi^*] - \psi_1\psi_1^* + *\psi_2^*\psi_2 &=\tau \ \id.
  \end{split}
\end{equation}

There is a Hitchin--Kobayashi correspondence between gauge-equivalence classes of solutions of \eqref{Eq_ADHMSWReducedPerturbed} and isomorphism classes of certain holomorphic data on $\Sigma$. 
Let $\sE=(E,\delbar_B)$ and $\sV=(V,\delbar_C)$ be the holomorphic bundles induced from the unitary connections on $E$ and $V$.

\begin{definition}
  An \defined{ADHM bundle} with respect to $(\sE,\sV,\theta)$ is a quadruple
  \begin{equation*}
    (\sH,\psi_1,\psi_2,\xi)
  \end{equation*}
  consisting of:
  \begin{itemize}
  \item
    a rank $k$ holomorphic vector bundle $\sH \to \Sigma$,
  \item
    $\psi_1\in H^0(\Sigma,\Hom(\sE,\sH))$,
  \item
    $\psi_2\in H^0(\Sigma,K_\Sigma\otimes\Hom(\sH,\sE))$, and
  \item
    $\xi\in H^0(\Sigma,\sV\otimes\End(\sH))$
  \end{itemize}
  such that
  \begin{equation*}
    [\xi\wedge\xi]+\psi_1\psi_2 = \theta\otimes\id \in H^0(\Sigma,K_\Sigma \otimes\End(\sH)).
  \end{equation*}
\end{definition}

\begin{definition}
  For $\delta \in \R$, the $\delta$--\defined{slope} of an ADHM bundle $(\sH,\psi_1,\psi_2,\xi)$ is
  \begin{equation*}
    \mu_{\delta}(\sH) \coloneq \frac{2\pi}{\vol(\Sigma)}\frac{\deg \sH}{\rk \sH} + \frac{\delta}{\rk\sH}.
  \end{equation*}
  The \defined{slope} of $\sH$ is $\mu(\sH) \coloneq \mu_0(\sH)$.
\end{definition}

\begin{definition}
  \label{Def_ADHMStable}
  Let $\delta\in\R$.
  An ADHM bundle $(\sH,\psi_1,\psi_2,\xi)$ is \defined{$\delta$ stable} if it satisfies the following conditions:
  \begin{enumerate}
  \item
    If $\delta > 0$, then $\psi_1 \neq 0$ and if $\delta < 0$, then $\psi_2 \neq 0$.
  \item
    If $\sG \subset \sH$ is a proper $\xi$--invariant holomorphic subbundle such that $\im \psi_1 \subset \sG$, then $\mu_{\delta}(\sG) < \mu_{\delta}(\sH)$.
  \item
    If $\sG \subset \sH$ is a proper $\xi$--invariant holomorphic subbundle such that $\sG\subset \ker \psi_2 $, then $\mu(\sG) < \mu_{\delta}(\sH)$.
  \end{enumerate}
  We say that $(\sH,\psi_1,\psi_2,\xi)$ is \defined{$\delta$--polystable} if there exists a $\xi$--invariant decomposition $\sH = \bigoplus_i \sG_i \bigoplus_j \sI_j$ such that:
  \begin{enumerate}
  \item
    $\mu_{\delta}(\sG_i) = \mu_{\delta}(\sH)$ for every $i$ and the restrictions of $(\psi_1,\psi_2,\xi)$ to each $\sG_i$ define a $\delta$ stable ADHM bundle, and
  \item
    $\mu(\sI_j) = \mu_{\delta}(\sH)$ for every $j$, the restrictions of $\psi_1$, $\psi_2$ to each $\sI_j$ are zero, and there exist no $\xi$--invariant proper subbundle $\sJ \subset \sI_i$ with $\mu(\sJ)<\mu(\sI_j)$. 
  \end{enumerate} 
\end{definition}

In the proposition below we fix $\delta$ and the topological type of $\sH$, and set $\tau = \mu_\delta(\sH)$. 

\begin{prop}
  \label{Prop_HKCorrespondence}
  Let $(A,\psi_1,\psi_2,\xi)$ be a solution of \eqref{Eq_ADHMSWReducedPerturbed}.
  Denote by $\sH$ the holomorphic vector bundle $(H,\delbar_A)$.
  Then $(\sH,\psi_1,\psi_2,\xi)$ is a $\delta$--polystable ADHM bundle. 
  Conversely, every $\delta$--polystable ADHM bundle arises in this way from a solution to \eqref{Eq_ADHMSWReducedPerturbed}  which is unique up to gauge equivalence.
\end{prop}

\begin{proof}
A standard calculation going back to \cite{Donaldson1983} shows that \eqref{Eq_ADHMSWReducedPerturbed}  implies $\delta$--polystability.
 The difficult part is showing that every $\delta$--polystable ADHM bundle admits a compatible unitary connection solving the third equation of \eqref{Eq_ADHMSWReducedPerturbed}, unique up to gauge equivalence.
  This is a special case of the main result of \cite[Theorem 31]{Consul2003}, with the minor difference that the connections on the bundles $E$ and $V$ are fixed and not part of a solution.
  The necessary adjustment in the proof is discussed in a similar setting in \cite{Bradlow2003}. 
\end{proof}

Stable ADHM bundles on Riemann surfaces were studied extensively by Diaconescu \cite{Diaconescu2012,Diaconescu2012a} in the case when $n=1$, $\sE$ is a trivial line bundle, and $\sV$ is the direct sum of two line bundles.
Thus, we have a splitting $\xi = (\xi_1,\xi_2)$ and $[\xi\wedge\xi] = [\xi_1,\xi_2]$, so the holomorphic equation
\begin{equation*}
  [\xi\wedge\xi] + \psi_1\psi_2 = \theta\otimes\id
\end{equation*}
is preserved by the $\C^*$--action $t(\psi_1,\psi_2,\xi_1,\xi_2) = (t\psi_1,t^{-1}\psi_2,t\xi_1,t^{-1}\xi_2)$. 
Moreover, if the perturbing form $\theta$ is chosen to be zero, there is an additional $\C^*$ symmetry given by rescaling every the sections. 
Assuming that the stability parameter $\delta$ is sufficiently large, Diaconescu shows the fixed-point locus of the resulting $\C^*\times\C^*$--action on the moduli space of $\delta$ stable ADHM bundles is compact.
Furthermore, the moduli space is equipped with a $\C^*\times\C^*$--equivariant perfect obstruction theory.
This can be used to define a numerical invariant via equivariant virtual integration.
This number is then shown to be equal to the \defined{local stable pair invariant} of the non-compact Calabi--Yau $3$--fold $\sV$.
This invariant counts, in the equivariant and virtual sense, stable pairs whose support is a $k$--fold thickening of the zero section $\Sigma\subset \sV$.
Here $k$ is the rank of $\sH$ so that the stable ADHM bundles in question correspond, by \autoref{Prop_HKCorrespondence}, to solutions to the ADHM$_{1,k}$ Seiberg--Witten equation on $S^1\times\Sigma$.
This suggests that the relation between Seiberg--Witten monopoles and stable pairs discussed in the previous section could extend to the case of multiple covers.

\subsection{Towards a numerical invariant}

Due to the appearance of reducible solutions, one does not expect to be able to count solutions to the ADHM$_{1,k}$ Seiberg--Witten equation on a general $3$--manifold.
Instead, the enumerative theory for associatives in tamed almost $\Gtwo$--manifolds should incorporate a version of equivariant Floer homology, as explained in \autoref{Sec_AssociativeMonopoleHomology} and \autoref{Sec_TentativeProposal}.
However, the existence of the stable pair invariant and the discussion of the previous sections indicate that we can hope for a differential-geometric invariant counting pseudo-holomorphic curves in a symplectic Calabi--Yau $6$--manifold $Z$ which in the projective case would recover the stable pair invariant.
It is expected that such an invariant would encode the same symplectic information as the Gromov--Witten invariants by the conjectural GW/PT correspondence, known also as the MNOP conjecture \cite[Sections $3\frac{1}{2}$ and $4\frac{1}{2}$]{Pandharipande2014}.
The algebro-geometric version of this conjecture is at present widely open.
Like the Gromov--Witten invariant, the putative symplectic stable pair invariant is given by a weighted count of simple $J$--holomorphic maps. 
Thus, we expect that a symplectic definition of the stable pair invariant will shed new light on the MNOP conjecture.

For a homology class $\beta\in H_2(Z,\Z)$ the invariant would take values in the ring of Laurent series $\Z(\!(q)\!)$ and be defined by
\begin{equation*}
  n_{\beta}(Z) =  \sum_{\Sigma^1,\ldots,\Sigma^m} \prod_{j=1}^m  \uSW_{1,k_j}(S^1\times \Sigma^j) \sign(\Sigma^j).
\end{equation*}
Some explanation is in order:
\begin{enumerate}
\item
  The sum is taken over all collections of embedded, connected pseudo-holomorphic curves $\Sigma^1,\ldots,\Sigma^m$ such that
   \begin{equation*}
     \sum_{j=1}^m k_j [\Sigma^j] = \beta.   
   \end{equation*}
   We assume here that we can choose a generic tamed almost-complex structure such that there are finitely many such curves and all of them are unobstructed.
\item
  $\sign(\Sigma) = \pm 1$ comes from an orientation on the moduli space of pseudo-holomorphic curves.
\item
  $\uSW_{1,k}(S^1\times \Sigma)$ is a generalization of the Meng--Taubes invariant defined using the moduli spaces of solutions to the ADHM$_{1,k}$ Seiberg--Witten equation on $S^1\times \Sigma^j$.
  This is yet to be defined, but if it exists, it should be naturally an element of $\Z(\!(q)\!)$  because of the identification of the set of the spin$^{\U(k)}$ structures on $\Sigma$ with the integers, as in \autoref{Sec_MengTaubes}.
\item
  We use here crucially that $b_1(S^1\times\Sigma) \geq 1$; otherwise even the classical Meng--Taubes invariant $\uSW_{1,1}$ is ill-defined.
  For $k > 1$, the ADHM$_{1,k}$ Seiberg--Witten equation, admits in general, reducible solutions: for example, flat connections or solutions to the ADHM$_{1,k-1}$ Seiberg--Witten equation.
  A good feature of the dimensionally-reduced setting is that if the perturbing holomorphic $1$--form $\theta$ in \eqref{Eq_ADHMSWReducedPerturbed} is non-zero, then we automatically avoid reducible solutions.
  Indeed, a simple algebraic argument shows that in this case the triple $(\xi,\psi_1,\psi_2)$ has trivial stabilizer in $\U(k)$ at every point where $\theta$ is non-zero.
\end{enumerate}


\appendix
\section{Transversality for associative embeddings}
\label{Sec_ProofOfSomewhereInjectiveTransversality}

The goal of this section is to prove \autoref{Prop_SomewhereInjectiveTransversality}.
The proof relies on the following observations.
The tangent space $T_\psi\sD_c^4(Y) \subset \Omega^4(Y)$ is the space of closed $4$--forms.
Define
\begin{equation*}
  X_{\iota,\psi} \co T_\psi\sD_c^4(Y) \to \Gamma(N\iota)
\end{equation*}
by
\begin{equation}
  \label{Eq_VariationOperator}
  \inner{X_{\iota,\psi}\eta}{n}_{L^2}
  \coloneq
  \left.\frac{\rd}{\rd\epsilon}\right|_{\epsilon=0}
  \delta\fL_{\psi+\epsilon\eta}(n)
  =
  \int_P \iota^* i(n)\eta
\end{equation}
for every closed $4$--form $\eta$ on $Y$.

\begin{prop}
  \label{Prop_VariationOfPsi}
  If $\iota\co P \to Y$ is a somewhere injective associative immersion,
  then for every non-zero $n \in \ker F_\iota \subset \Gamma(N\iota)$, there exists $\alpha \in \Omega^3(Y)$ such that
  \begin{equation*}
    \inner{X_{\iota,\psi}\rd\alpha}{n} \neq 0.
  \end{equation*}
\end{prop}

\begin{proof}
  We can assume that $P$ is connected.
  Pick a point $x$ such that $\iota^{-1}(\iota(x)) = \set{x}$.
  Since $P$ is compact, there is a neighborhood $U$ of $x\in P$ which is embedded via $\iota$ and satisfies $\iota(U) \cap \iota(P\setminus U) = \emptyset$.
  Choose a tubular neighborhood $V$ of $\iota(U)$ and $\rho > 0$ such that $B_\rho(N\iota(U)) \xrightarrow{\exp} V$ is a diffeomorphism.
  By unique continuation, $n$ cannot vanish identically on $U$.
  Thus we can find a function $f$ supported in $V$ such that $\rd f(n) \geq 0$ and $\rd f(n) > 0$ somewhere.
  Let $\nu$ be a $3$--form on $Y$ with $\nu|_U = (\vol_P)|_U$ and $i(n)\rd\nu|_V = 0$.
  With
  \begin{equation*}
    \alpha = f\nu
  \end{equation*}
  we have
  \begin{equation*}
    \int_P \iota^*(i(n)\rd\alpha) 
    =
    \int_P \rd f(n) \vol_P > 0.
    \qedhere
  \end{equation*}
\end{proof}

For a somewhere injective immersed associative $[\iota\co P \to Y]$, $\Aut(\iota)$ must be trivial.
Denote by $\pi_{\Imm}\co \Imm_\beta(P,Y) \times \sD_c^4(Y) \to \Imm_\beta(P,Y)$ the canonical projection.
By \autoref{Prop_VariationOfPsi},
the linearization of the section
\begin{equation*}
  \delta\fL \in \Gamma(\pi_{\Imm}^*T^*\Imm_\beta(P,Y))
\end{equation*}
is surjective.
Hence, it follows from the Regular Value Theorem,
and the fact that there are only countably many diffeomorphism types of $3$--manifolds \cite{Cheeger1970},
that the universal moduli space of immersed associatives
\begin{equation*}
  \fA^\si_\beta = \fA^\si_\beta(\sD^4_c(Y))
\end{equation*}
is a smooth manifold.
This directly implies \itref{Prop_SomewhereInjectiveTransversality_11} and \itref{Prop_SomewhereInjectiveTransversality_21} by the Sard--Smale Theorem.

Consider the moduli space of immersed associative submanifolds with $n$ marked points
\begin{equation*}
  \fA^\si_{\beta,n}(\psi)
  \coloneq
  \coprod_P
  \left.\set*{ (\iota,x_1,\ldots,x_n) \in \Imm_\beta(P,Y)\times P^n : [\iota] \in \fA^\si_\beta(\psi) }\right/\Diff_+(P)
\end{equation*}
as well as the corresponding universal moduli space
\begin{equation*}
  \fA^\si_{\beta,n}
  \coloneq
  \bigcup_{\psi \in \sD^4_c(Y)} \fA^\si_{\beta,n}(\psi).
\end{equation*}
Define the map $\ev \co \fA^\si_{\beta,n} \to Y^n$ by
\begin{equation*}
  \ev([\iota,x_1,\ldots,x_n],\psi) \coloneq (\iota(x_1),\ldots,\iota(x_n)).
\end{equation*}

\begin{prop}
  \label{Prop_EvaluationMapRegular}
  For each $([\iota,x_1,\ldots,x_n],\psi) \in \fA^\si_{\beta,n}$,
  the derivative of $\ev$,
  \begin{equation*}
    \rd_{([\iota,x_1,\ldots,x_n],\psi)}\ev\co
    T_{([\iota,x_1,\ldots,x_n],\psi)}\fA^\si_{\beta,n} \to \bigoplus_{i=1}^n T_{\iota(x_i)}Y,
  \end{equation*}
  is surjective.
\end{prop}

\begin{proof}
  We will show that if $(v_1,\ldots,v_n) \in \bigoplus_{i=1}^n N_{x_i}\iota$,
  then there exist $n \in \Gamma(N\iota)$ and $\eta \in T_{\psi}\sD_c^4(Y)$ such that
  \begin{equation*}
    n(x_i) = v_i \qandq
    (n,\eta) \in T_{[\iota],\psi}\fA^\si_\beta.
  \end{equation*}
  This immediately implies the assertion.

  Denote by $\ev_{x_1,\ldots,x_n}\co \Gamma(N\iota) \to \bigoplus_{i=1}^n N_{x_i}\iota$ the evaluation map and define
  \begin{equation*}
    \bF^{k} \coloneq \(F_\iota \oplus X_{\iota,\psi}
    \co
    W^{k,2}\ker\ev_{x_1,\ldots,x_n} \oplus T_\psi\sD_c^4(Y)
    \to W^{k-1,2}\Gamma(N\iota)\),
  \end{equation*}
  where $F_{\iota}$ is the Fueter operator and $X_{\iota,\psi}$ is defined in \eqref{Eq_VariationOperator}.
  We prove that the operator $\bF^1$ is surjective,
  cf. \citet[Proof of Lemma 3.4.3]{McDuff2012}.
  To see this note that its image is closed and thus we need to show only that if $\nu \perp\im \bF^1$, then $\nu = 0$.
  Since $\nu \perp F_\iota(W^{1,2}\ker \ev_{x_1,\ldots,x_n})$,
  on $P \setminus \set{ x_1,\ldots, x_n }$,
  $\nu$ is smooth and satisfies $F_\iota\nu = 0$.
  We also know that $\nu \perp \im X_{\iota,\psi}$.
  The argument from the Proof of \autoref{Prop_VariationOfPsi} shows that $\nu = 0$,
  because the set of points $x\in P$ satisfying $\iota^{-1}(\iota(x)) = \set{x}$ is open in $P$ so we can choose such a point $x$ belonging to $P\setminus \set{ x_1, \ldots, x_n }$).
  That $\bF^k$ is surjective follows from the fact that $\bF^1$ is surjective by elliptic regularity.
  
  Pick $n_0 \in \Gamma(N\iota)$ with
  \begin{equation*}
    n_0(x_i) = v_i
  \end{equation*}
  and pick $(n_1,\eta) \in \ker\ev_{x_1,\ldots,x_n} \oplus T_\psi\sD_c^4(Y)$ such that
  \begin{equation*}
    F_\iota n_1 + X_{\iota,\psi}(\eta) = - F_\iota n_0.
  \end{equation*}
  The pair $(n_0+n_1,\eta) \in T_{[\iota],\psi}\fA^\si_\beta$ has the desired properties.
\end{proof}

Finally, we are in a position to prove \itref{Prop_SomewhereInjectiveTransversality_12} and \itref{Prop_SomewhereInjectiveTransversality_22} of \autoref{Prop_SomewhereInjectiveTransversality}.
Denote by $\pi \co \fA^\si_{\beta,2} \to \fA^\si_\beta$ the forgetful map and
denote by $\Delta = \set{ (x,x) \in Y \times Y : x \in Y }$ the diagonal in $Y$.
\autoref{Prop_EvaluationMapRegular}
The universal moduli space of non-injective but somewhere injective immersed associatives is precisely
\begin{equation*}
  \pi(\ev^{-1}(\Delta)).
\end{equation*}
By \autoref{Prop_EvaluationMapRegular},
$\ev^{-1}(\Delta) \subset \fA^\si_\beta$ is a codimension $7$ submanifold.
Since $\pi$ is a Fredholm map of index $6$ and
$\rho\co \fA^\si \to \sD^4_c(Y)$ is a Fredholm map of index $0$,
it follows that $\rho(\pi(\ev^{-1}(\Delta))) \subset \sD^4_c(Y)$ is residual.
This proves \itref{Prop_SomewhereInjectiveTransversality_12} because an injective immersion of a compact manifold is an embedding.
The proof of \itref{Prop_SomewhereInjectiveTransversality_22} is similar.
This completes the proof of \autoref{Prop_SomewhereInjectiveTransversality}.
\qed


\section{Seiberg--Witten equations in dimension three}
\label{Sec_SeibergWitten}

We very briefly review how to associate a Seiberg--Witten equation to a quaternionic representation.
More detailed discussions can be found in \cites{Taubes1999b}{Pidstrigach2004}{Haydys2008}[Section 6]{Salamon2013}[Section 6(i)]{Nakajima2015};
we follow  \cite[Section 1]{Doan2017a} closely.
The first ingredient is a choice of algebraic data.

\begin{definition}
  A \defined{quaternionic Hermitian vector space} is a real vector space $S$ together with a linear map $\gamma\co \Im \H \to \End(S)$ and an inner product $\inner{\cdot}{\cdot}$, such that $\gamma$ makes $S$ into a left module over the quaternions $\H = \R\Span{1,i,j,k}$, and $i,j,k$ act by isometries.
  The \defined{unitary symplectic group} $\Sp(S)$ is the subgroup of $\GL(S)$ preserving $\gamma$ and $\inner{\cdot}{\cdot}$.
  A \defined{quaternionic representation} of a Lie group $G$ on $S$ is a homomorphism $\rho \co G \to \Sp(S)$.
\end{definition}

Let $\rho\co G \to \Sp(S)$ be a quaternionic representation.
Denote by $\fg$ the Lie algebra of $G$.
There is a canonical hyperkähler moment map $\mu \co S \to (\fg \otimes \Im \H)^*$ defined as follows.
By slight abuse of notation denote by $\rho\co \fg \to \sp(S)$ the Lie algebra homomorphism induced by $\rho$.
Combine $\rho$ and $\gamma$ into the map $\bar\gamma\co \fg\otimes\Im\H \to \End(S)$ given by
\begin{equation*}
  \bar\gamma(\xi\otimes v)\Phi \coloneq \rho(\xi)\gamma(v)\Phi.
\end{equation*}
The map $\bar\gamma$ takes values in the space of symmetric endomorphisms of $S$.
Denote by $\bar\gamma^*\co \End(S) \to (\fg\otimes \Im\H)^*$ the adjoint of $\bar\gamma$.
Define
\begin{equation*}
  \mu(\Phi) \coloneq \frac12\bar\gamma^*(\Phi\Phi^*).
\end{equation*}

\begin{definition}
  \label{Def_CanonicalPermutingAction}
  The \defined{canonical permuting action} $\theta\co \Sp(1) \to \O(S)$ is defined by left-multiplication by unit quaternions.
  It satisfies
  \begin{equation*}
    \theta(q) \gamma(v)\Phi = \gamma(\Ad(q)v)\theta(q)\Phi
  \end{equation*}
  for all $q \in \Sp(1) = \set{ q \in \H : \abs{q} = 1 }$, $v \in \Im\H$, and $\Phi \in S$.
\end{definition}

\begin{definition}
  \label{Def_AlgebraicData}
  A set of \defined{algebraic data} consists of:
  \begin{itemize}
  \item
    a quaternionic Hermitian vector space $(S,\gamma,\inner{\cdot}{\cdot})$,
  \item
    a compact, connected Lie group $H$, 
    an injective homomorphism $\Z_2 \to Z(H)$, 
    an $\Ad$--invariant inner product on $\Lie(H)$,
  \item
    a closed, connected, normal subgroup $G\nsub H$, and
  \item
    a quaternionic representation $\rho\co H\to \Sp(S)$ such that $-1 \in \Z_2 \subset Z(H)$ acts as $-\id_S$.
  \end{itemize}
\end{definition}

\begin{definition}
  Given a set of algebraic data, set
  \begin{equation*}
    \hat H \coloneq (\Sp(1)\times H)/\Z_2, \quad
    K \coloneq H/G, \qandq
    \hat K \coloneq (\Sp(1)\times K)/\Z_2.
  \end{equation*}
  The group $K$ is called the \defined{flavor group}.
\end{definition}

\begin{example}
  \label{Ex_ADHMAlgebraicData}
  The ADHM$_{r,k}$ Seiberg--Witten equation arise by choosing
  \begin{equation*}
    S = S_{r,k} \coloneq \Hom_\C(\C^r,\H\otimes_\C \C^k) \oplus \H\otimes_\R\fu(k)
  \end{equation*}
  with
  \begin{equation*}
    G = \U(k) \nsub H = \SU(r) \times \Sp(1) \times \U(k)
  \end{equation*}
  where $\SU(r)$ acts on $\C^r$ in the obvious way, $\U(k)$ acts on $\C^k$ in the obvious way and on $\fu(k)$ by the adjoint representation, and $\Sp(1)$ acts on the first copy of $\H$ trivially and on the second copy by right-multiplication with the conjugate.
  The homomorphism $\Z_2\to Z(H)$ is defined by $-1 \mapsto (\id_{\C^r},-\id_\H,-\id_{\C^k})$.
  In particular,
  \begin{equation*}
    \hat H = \SU(r) \times \Spin^{\U(k)}(4)
  \end{equation*}
  with
  \begin{equation*}
    \Spin^{\U(k)}(n) \coloneq (\Spin(n)\times\U(k))/\Z_2.
  \end{equation*}
  Although notationally cumbersome,
  we usually prefer to think of $\hat H$ as
  \begin{equation*}
    \hat H = \SU(r) \times \Spin^{\U(k)}(3)\times_{\SO(3)} \SO(4).
  \end{equation*}
  Here the second factor is the fiber product of $\Spin^{\U(k)}(3)$ with $\SO(4)$ with respect to the obvious homomorphism $\Spin^{\U(k)}(3) \to \SO(3)$ and the homomorphism $\SO(4) \to \SO(3)$ is given by the action on $\Lambda^+\R^4$.
\end{example}

In addition to a set of algebraic data has been chosen one also needs to fix the geometric data for which the Seiberg--Witten equation will be defined.

\begin{definition}
  \label{Def_GeometricData}
  Let $M$ be a closed, connected, oriented $3$--manifold.
  A set of \defined{geometric data} on $M$ compatible with a set of algebraic data as in \autoref{Def_AlgebraicData} consists of:
  \begin{itemize}
  \item
    a Riemannian metric $g$ on $M$,
  \item
    a principal $\hat H$--bundle $\hat Q \to M$ together with an isomorphism
    \begin{equation}
      \label{Eq_HatQInducesSOTM}
      \hat Q\times_{\hat H}\SO(3) \iso \SO(TM),
    \end{equation}
    and
  \item
    a connection $B$ on the principal $K$--bundle
    \begin{equation*}
      R \coloneq \hat Q\times_{\hat H} K.
    \end{equation*}
  \end{itemize}
\end{definition}

\begin{definition}
  Given a choice of geometric data,
  the \defined{spinor bundle} and the \defined{adjoint bundle} are the vector bundles%
  \footnote{%
    If $H = G\times K$, then there is a principal $G$--bundle $P \to M$ associated with $\hat Q$ and $\fg_P$ is the adjoint bundle of $P$.
    In general, $P$ might not exist but traces of it remain, e.g., its adjoint bundle $\fg_P$ and its gauge group $\sG(P)$.
  }
  \begin{equation*}
    \bS \coloneq \hat Q\times_{\theta\times\rho} S
    \qandq
    \fg_P \coloneq \hat Q\times_{\Ad} \fg.
  \end{equation*}
  Because of \eqref{Eq_HatQInducesSOTM} the maps $\gamma$ and $\mu$ induce maps
  \begin{equation*}
    \gamma\co T^*M \to \End(\bS)
    \qandq
    \mu\co \bS \to \Lambda^2 T^*M \otimes \fg_P.
  \end{equation*}
  Here we take $\mu$ to be the moment map corresponding to the action of $G \nsub H$.
\end{definition}

\begin{definition}
  Set
  \begin{equation*}
    \sA_B(\hat Q)
    \coloneq
    \set*{
      A \in \sA(\hat Q)
      :
      \begin{array}{@{}l@{}}
        A \text{ induces } B \text { on } R \text{ and the} \\
        \text{Levi-Civita connection on } TM
      \end{array}
    }.
  \end{equation*}
  Any $A \in \sA_B(\hat Q)$ defines a covariant derivative $\nabla_A\co \Gamma(\bS) \to \Omega^1(M,\bS)$.
  The \defined{Dirac operator} associated with $A$ is the linear map $\slD_A\co \Gamma(\bS) \to \Gamma(\bS)$ defined by
  \begin{equation*}
    \slD_A\Phi \coloneq \gamma(\nabla_A\Phi).
  \end{equation*}
  $\sA_B(\hat Q)$ is an affine space modeled on $\Omega^1(M,\fg_P)$.
  Denote by $\varpi \co \Ad(\hat Q) \to \fg_P$ the projection induced by $\Lie(\hat H) \to \Lie(G)$.
\end{definition}

Finally, we are in a position to define the Seiberg--Witten equation.

\begin{definition}
  The \defined{Seiberg--Witten equation} associated with the chosen algebraic and geometric data is the following system of partial differential equations for $(\Phi,A) \in \Gamma(\bS) \times \sA_B(\hat Q)$:
   \begin{equation}
    \label{Eq_SeibergWitten}
    \begin{split}
      \slD_A\Phi 
      &= 0 \qand \\
      \varpi F_A 
      &=  \mu(\Phi).
    \end{split}
  \end{equation}
\end{definition}

The Seiberg--Witten equation is invariant with respect to gauge transformations which preserve the flavor bundle $R$ and $\SO(T^*M)$.

\begin{definition}
  The \defined{group of restricted gauge transformations}
  is
  \begin{equation*}
    \sG(P)
    \coloneq
    \set*{ u \in \sG(\hat Q) : u \text{ acts trivially on } R \text{ and } \SO(TM) }.
  \end{equation*}
  $\sG(P)$ can be identified with the space of sections of $\hat Q\times_{\hat H} G$ with $\hat H$ acting on $G$ via $[(q,h)] \cdot g = hgh^{-1}$.
\end{definition}

If $\mu^{-1}(0) = \set{0}$,
then one proves in the same way as for the classical Seiberg--Witten equation that solutions of \eqref{Eq_SeibergWitten} obey a priori bounds on $\Phi$.
In many cases of interest $\mu^{-1}(0) \neq \set{0}$ and in these cases a priori bounds fail to hold.
Anticipating this, we blow-up the Seiberg--Witten equation.

\begin{definition}
  The \defined{blown-up Seiberg--Witten equation} is the following partial differential equation for $(\epsilon,\Phi,A) \in [0,\infty)\times\Gamma(\bS)\times\sA_B(\hat Q)$:
  \begin{equation}
    \label{Eq_BlownUpSeibergWitten}
    \begin{split}
      \slD_A\Phi 
      &=
      0, \\
      \epsilon^2 \varpi F_A 
      &=
      \mu(\Phi), \qand \\
      \Abs{\Phi}_{L^2}
      &=
      1.
    \end{split}
  \end{equation}
  
  The \defined{limiting Seiberg--Witten equation} is the following partial differential equation for $(\Phi,A) \in [0,\infty) \in \Gamma(\bS)\times \sA_B(\hat Q)$:
  \begin{equation}
    \label{Eq_LimitingSeibergWitten}
    \begin{split}
      \slD_A\Phi &= 0 \qand \\
      \mu(\Phi) &= 0
    \end{split}
  \end{equation}
  as well as $\Abs{\Phi}_{L^2} = 1$.
\end{definition}

The phenomenon of $\Phi$ tending to infinity for \eqref{Eq_SeibergWitten} corresponds to $\epsilon$ tending to zero for \eqref{Eq_BlownUpSeibergWitten}
Formally,
the compactifiction of the moduli space of solutions of \eqref{Eq_SeibergWitten} should thus be given by adding solution of the limiting equation.
\citet{Taubes2012,Haydys2014} proved that---up to allowing for codimension on singularities in the limiting solutions---this is true for the flat $\PSL(2,\C)$--connections and the Seiberg--Witten equation with multiple spinors, which are particular instances of equation \eqref{Eq_SeibergWitten}.
Although one might initially hope that it is unnecessary to allow for singularities in solutions of the limiting equation,
it has been shown in \cite{Doan2017c} that this phenomenon cannot be avoided.


\section{The Haydys correspondence with stabilizers}
\label{Sec_HaydysCorrespondence}

Throughout this appendix we assume that algebraic data and geometric data as in \autoref{Def_AlgebraicData} and \autoref{Def_GeometricData} have been chosen.
Denote by 
\begin{equation*}
  X \coloneq S\hkred G = \mu^{-1}(0)/G
\end{equation*}
the \defined{hyperkähler quotient} of $X$ by $G$,
and denote by $p\co \mu^{-1}(0) \to X$ the canonical projection.
The action of $\hat H$ on $S$ induces an action of $\hat K = \hat H/G$ on $X$.
Set
\begin{equation*}
  \bX \coloneq \hat R\times_{\hat K}X.
\end{equation*}
If $\Phi \in \Gamma(\bS)$ satisfies $\mu(\Phi) = 0$, then
\begin{equation}
  \label{Eq_S=PPhi}
  s \coloneq p\circ \Phi \in \Gamma(\bX).
\end{equation}
The Haydys correspondence \cite[Section 4.1]{Haydys2011} relates solutions of the limiting Seiberg--Witten equation \eqref{Eq_LimitingSeibergWitten} with certain sections of $\bX$.
The discussions of the Haydys correspondence available in the literature so far \cites[Section 4.1]{Haydys2011}[Section 3]{Doan2017a} assume that the action of $G$ on $\mu^{-1}(0)$ is generically free.
This hypothesis does not hold in \autoref{Ex_ADHMAlgebraicData} with $r = 1$,
which leads to the ADHM$_{1,k}$ Seiberg--Witten equation.
This appendix is concerned with extending the Haydys correspondence to the case when $G$ acts on $\mu^{-1}(0)$ with a non-trivial generic stabilizer.

\subsection{Decomposition of hyperkähler quotients}
\label{Sec_DecompositionOfHyperkahlerQuotients}

Denote by $S_{\set{e}}$ the subset of $S$ on which $G$ acts freely.
By \cite[Section 3(D)]{Hitchin1987}, the quotient
\begin{equation*}
  (S_{\set{e}}\cap\mu^{-1}(0))/G
\end{equation*}
can be given the structure of hyperkähler manifold of dimension $4(\dim_\H S - \dim G)$ such that, for $\Phi \in S_{\set{e}}\cap\mu^{-1}(0)$,
\begin{equation}
  p_*\co (\rho(\fg)\Phi)^\perp\cap T_\Phi\mu^{-1}(0) \to T_{[\Phi]}X
\end{equation}
is a quaternionic isometry.
If $G$ acts on $\mu^{-1}(0)$ with trivial generic stabilizer (that is: $S_{\set{e}}$ is dense and open),
then this makes an dense open subset of $X$ into a hyperkähler manifold.
In general, $X$ can be decomposed as a union of hyperkähler manifolds according to orbit type as follows.

\begin{definition}
  \label{Def_ST}
  For $\Phi \in S$, denote by $G_\Phi$ the stabilizer of $\Phi$ in $G$.
  Let $T < G$ be a subgroup.
  Set
  \begin{equation*}
    S_T
    \coloneq
    \set{
      \Phi \in S : G_\Phi = T
    }
    \qandq
    S_{(T)}
    \coloneq
    \set{
      \Phi \in S : gG_\Phi g^{-1} = T \text{ for some } g \in G
    }.
  \end{equation*}
\end{definition}

\begin{definition}
  Given a subgroup $T < G$, set
  \begin{equation*}
    W_G(T) \coloneq N_G(T)/T.
  \end{equation*}
  Here $N_G(T)$ denotes the normalizer of $T$ in $G$. 
\end{definition}

\begin{remark}
  This notation is motivated by the example $S = \H \otimes \fg$, with $G$ acting via the adjoint representation. 
  In this case, the stabilizer $T$ of a generic point in $\mu^{-1}(0)$ is a maximal torus and $W_G(T)$ is the Weyl group of $G$;
  cf.~\autoref{Sec_NakajimasProof} for the case $G=\U(k)$.
\end{remark}

\begin{theorem}[{\citet[Theorem 2.1]{Dancer1997}; \citet{Sjamaar1991}, \citet[Section 6]{Nakajima1994a}}]
  \label{Thm_HyperkahlerQuotientWithStabilizer}
  For each $T < G$, the quotient
  \begin{equation*}
    X_{(T)} \coloneq (\mu^{-1}(0)\cap S_{(T)})/G
  \end{equation*}
  is a hyperkähler manifold, and
  \begin{equation}
    \label{Eq_DecompositionOfX}
    X = \bigcup_{(T)} X_{(T)} 
  \end{equation}
  where $(T)$ runs through all conjugacy classes of subgroups of $G$.%
  \footnote{%
    There can be subgroups $T < G$ with $S_{(T)} \neq 0$, but $\mu^{-1}(0)\cap S_{(T)} = \emptyset$.
  }
  More precisely, for each $T < G$:
  \begin{enumerate}
  \item
    \label{Thm_HyperkahlerQuotientWithStabilizer_STSubmanifold}
    $S_T$ is a hyperkähler submanifold of $S$ and
    $S_{(T)}$ is a submanifold of $S$.
  \item
    \label{Thm_HyperkahlerQuotientWithStabilizer_S(T)/G=ST/W}
    We have
    \begin{equation*}
      (\mu^{-1}(0)\cap S_{(T)})/G = (\mu^{-1}(0)\cap S_{T})/W_G(T).
    \end{equation*}
  \item
    \label{Thm_HyperkahlerQuotientWithStabilizer_MuOnST0}
    Denote by $S_T^0$ denotes the union of the components of $S_T$ intersecting $\mu^{-1}(0)$.
    Then $W_G(T)$ acts freely on $S_T^0$ and
    \begin{equation*}
      \mu(S_T^0)
      \subset
        (\fw\otimes\Im\H)^*
    \end{equation*}
    with $\fw \coloneq \Lie(W_G(T))$.
    In particular, the restriction of $\mu$ to $S_T^0$ induces a hyperkähler moment map on $S_T^0$ for the action of $W_G(T)$.
  \item
    \label{Thm_HyperkahlerQuotientWithStabilizer_X(T)}
    $X_{(T)}$  can be given the structure of a hyperkähler manifold such that, for each $\Phi \in \mu^{-1}(0)\cap  S_{(T)}$,
    \begin{equation*}
      p_*\co (\rho(\fg)\Phi)^\perp \cap \ker \rd_\Phi\mu \cap T_\Phi S_{(T)} \to T_{[\Phi]}X_{(T)}
    \end{equation*}
    is a quaternionic isometry.
  \end{enumerate}
\end{theorem}

\begin{proof}
  We recall \citeauthor{Dancer1997}'s argument, since some aspects of it will play a role later on.

  To prove \itref{Thm_HyperkahlerQuotientWithStabilizer_STSubmanifold},
  denote by
  \begin{equation*}
    S^T \coloneq \set{ \Phi \in S : G_\Phi \supset T }
  \end{equation*}
  the fixed-point set of the action of $T$.
  $S^T$ is an $\H$--linear subspace of $S$ and $S_T$ is an open subset of $S^T$ (by the Slice Theorem).
  Therefore, $S_T$ is a hyperkähler submanifold of $S$.
  The group action induces a bijection
  \begin{equation*}
    S_T \times_{W_G(T)} G/T \iso S_{(T)}, \quad
    [\Phi,gT] \mapsto \rho(g)\Phi.
  \end{equation*}
  This shows that $S_{(T)}$ is a submanifold of $S$.
  For future reference, we also observe that
  \begin{equation}
    \label{Eq_TS(T)}
    T_\Phi S_T
    =
      S^T \qandq
    T_\Phi S_{(T)}
    =
      S^T + \rho(\fg)\Phi
    \iso
      \frac{S^T \oplus \rho(\fg)\Phi}{\rho(\fw)\Phi}.
  \end{equation}

  The assertion made in \itref{Thm_HyperkahlerQuotientWithStabilizer_S(T)/G=ST/W} follows directly from the definitions.
  
  To prove \itref{Thm_HyperkahlerQuotientWithStabilizer_MuOnST0},
  observe that by the definition of $S_T$,
  the group $W_G(T) = N_G(T)/T$ acts freely on $S_T$.
  Since $\mu$ is $G$--equivariant, $\mu(S_T) \subset (\fg^*)^T \subset \fn^*$ with $\fn \coloneq \Lie(N_G(T))$.
  Let $\ft = \Lie(T)$.
  If $\Phi \in S_T$, then
  \begin{equation}
    \label{Eq_DMUPhiOnST}
    \rd_\Phi\mu \in \Ann_{\fg^*} \ft \otimes (\Im\H)^*,
  \end{equation}
  because, for $\xi \in \ft$, $v \in \Im\H$, and $\phi \in S$, we have
  \begin{equation*}
    \inner{(\rd_\Phi\mu)\phi}{\xi\otimes v}
    = \inner{\gamma(v)\rho(\xi)\Phi}{\phi}
    = 0.
  \end{equation*}
  Since $\fw^* = \fn^*\cap \Ann_{\fg_*}\ft$,
  we have $\mu(S_T^0) \subset (\fw\otimes\Im\H)^*$.
  This proves \itref{Thm_HyperkahlerQuotientWithStabilizer_MuOnST0}.

  Finally, we prove \itref{Thm_HyperkahlerQuotientWithStabilizer_X(T)}.
  Since
  \begin{equation*}
    X_{(T)}
    =
      (\mu^{-1}(0)\cap S_{(T)})/G
    =
      (\mu^{-1}(0)\cap S_{(T)})/W_G(T)
    =
      S_T^0 \hkred W_G(T),
  \end{equation*}
  $X_{(T)}$ can be given a hyperkähler structure by the construction in \cite[Section 3(D)]{Hitchin1987}.
  If $\Phi \in S_T$, then
  \begin{equation*}
    (\rho(\fg)\Phi)^\perp \cap T_\Phi S_{(T)} = (\rho(\fw)\Phi)^\perp\cap T_\Phi S_T
  \end{equation*}
  by \eqref{Eq_TS(T)};
  hence, by the discussion before \autoref{Def_ST},
  \begin{equation*}
    p_*\co
    (\rho(\fg)\Phi)^\perp \cap \ker \rd_\Phi\mu \cap T_\Phi S_{(T)}
    =
      (\rho(\fw)\Phi)^\perp\cap \ker \rd_\Phi\mu \cap T_\Phi S_T
    \to
      T_{[\Phi]}X_{(T)}
  \end{equation*}
  is a quaternionic isometry.
  This finishes the proof of \itref{Thm_HyperkahlerQuotientWithStabilizer_X(T)}.
\end{proof}

In general,
the action of $\hat K = \hat H/G$ need not preserve the strata $X_{(T)}$.
The following hypothesis,
which holds for all the examples considered in this article,
guarantees that the action of $\hat H$ on $S$ preserves $S_{(T)}$ and that the action of $\hat K$ on $X$ preserves $X_{(T)} \subset X$.

\begin{hyp}
  \label{Hyp_HConjugateOrbit=GConjugateOrbit}
  Given $T < G$, assume that, for all $h \in H$, there is a $g \in G$ such that
  \begin{equation*}
    hTh^{-1} = gTg^{-1}.
  \end{equation*}
\end{hyp}

\begin{prop}
  \label{Prop_HatHPreservesS(T)X(T)}
  If \autoref{Hyp_HConjugateOrbit=GConjugateOrbit} holds for $T < G$,
  then the action of $\hat H$ on $S$ preserves the submanifold $S_{(T)}$ and the action of $\hat K$ on $X$ preserves $X_{(T)}$.
  \qed
\end{prop}

\begin{proof}
  For $h \in H$ and $\Phi \in S_{(T)}$,
  we have $G_{\rho(h)\Phi} = hG_{\Phi}h^{-1} = hTh^{-1} = gTg^{-1}$ for some $g \in G$.
  Thus, $\rho(h)\Phi \in S_{(T)}$ and the action of $H$ preserves $S_{(T)}$.
  The action of $\Sp(1)$ commutes with that of $H$ and so it also preserves $S_{(T)}$.
  We conclude that $S_{(T)}$ is preserved by the action of $\hat H $.
  Since $X_{(T)} = (\mu^{-1}(0)\cap S_{(T)})/G$, the action of $\hat{K}$ preserves $X_{(T)}$.
\end{proof}

\begin{prop}
  \label{Prop_NHT/NGT=K}
  For any $T < G$, $N_G(T)$ is a normal subgroup of $N_H(T)$,
  and the identity $K = H/G$ induces an injective homomorphism $N_H(T)/N_G(T) \into K$.
  If \autoref{Hyp_HConjugateOrbit=GConjugateOrbit} holds for $T < G$,
  then this map is an isomorphism
  \begin{equation*}
    N_H(T)/N_G(T) \iso K.
  \end{equation*}
\end{prop}

\begin{proof}
  If $g \in N_G(T)$ and $h \in N_H(T)$,
  then $\tilde g \coloneq hgh^{-1} \in G$ since $G\nsub H$;
  hence, $\tilde g \in N_G(T)$.
  Since $N_H(T) \cap G = N_G(T)$, we have an injective homomorphism $N_H(T)/N_G(T) \into K$.
  
  Assuming \autoref{Hyp_HConjugateOrbit=GConjugateOrbit} and given $k = hG \in K$,
  there is a $g \in G$ such that
  \begin{equation*}
    hTh^{-1} = gTg^{-1}.
  \end{equation*}
  It follows that $\tilde h \coloneq g^{-1}h \in N_H(T)$ and $\tilde h G = k$;
  hence, $N_H(T)/N_G(T) \into K$ is an isomorphism.
\end{proof}

Assuming \autoref{Hyp_HConjugateOrbit=GConjugateOrbit} for $T < G$, we can define fiber bundles over $M$ whose fibers are the strata $S_{(T)}$ and $X_{(T)}$:
\begin{equation*}
  \bS_{(T)}
  \coloneq
    \hat Q\times_{\hat H} S_{(T)}
  \qandq
  \bX_{(T)}
  \coloneq
    \hat R \times_{\hat K} X_{(T)}.
\end{equation*}
If it holds for all $T < G$ with non-empty $S_T$, we decompose $\bS$ and $\bX$ as
\begin{equation*}
  \bS
  =
    \bigcup_{(T)} \bS_{(T)}
  \qandq
  \bX_{(T)}
  =
  \bigcup_{(T)} \bX_{(T)}.
\end{equation*}

\subsection{Lifting sections of \texorpdfstring{$\bX_{(T)}$}{XT}}
\label{Sec_LiftingSections}

For the remainder of this section we will assume \autoref{Hyp_HConjugateOrbit=GConjugateOrbit} for $T < G$.
The first part of the Haydys correspondence is concerned with the questions:
\begin{quote}
  \centering
  When can a section $s \in \Gamma(\bX_{(T)})$ be lifted a section of $\Phi \in \Gamma(\bS_{(T)})$ with $\mu(\Phi) = 0$ for some choice of $\hat Q$?
\end{quote}
and
\begin{quote}
  \centering
  To what extend is the principal $\hat H$--bundle $\hat Q$ determined by $s$?
\end{quote}

\begin{prop}
  \label{Prop_HatQCirc}
  If $\Phi \in \Gamma(\bS_{(T)})$, then
  \begin{equation*}
    \hat Q^\circ = \hat Q^\circ_\Phi
    \coloneq
    \set*{
      q \in \hat Q : \Phi(q) \in S_T
    }%
    \footnote{%
      Here we think of $\Phi$ as a $\hat H$--equivariant map $\Phi\co \hat Q \to S$.
    }
  \end{equation*}
  is a principal $N_{\hat H}(T)$--bundle over $M$ whose associated principal $\hat H$--bundle is isomorphic to $\hat Q$. 
  Moreover,
  the stabilizer of $\Phi$ in $\sG(P) = \Gamma(\hat Q\times_{\hat H} G)$ is 
  \begin{equation*}
    \Gamma(\hat Q^\circ\times_{N_{\hat H}(T)}T) \subset \sG(P),
  \end{equation*}
  and the kernel of $\rho(\cdot)\Phi \co \fg_P \to \bS$ is
  \begin{equation}
    \label{Eq_ftP}
    \ft_P \coloneq \hat Q^\circ\times_{N_{\hat H}(T)} \Lie(T) \subset \fg_P.
  \end{equation}
\end{prop}

\begin{proof}
  If $\Phi \in S_T$, $\hat h = [(q,h)] \in \hat H = (\Sp(1)\times H)/\Z_2$ and $\Psi \coloneq \theta(q)\rho(h)\Phi$, then
  \begin{equation*}
    G_\Psi = hG_\Phi h^{-1} = hTh^{-1}.%
    \footnote{%
      \autoref{Hyp_HConjugateOrbit=GConjugateOrbit} ensures that $hTh^{-1} \subset G$.
    }
  \end{equation*}
  Therefore, $\Psi \in S_T$ if and only if $\hat h \in N_{\hat H}(T) = (\Sp(1)\times N_H(T))/\Z_2$.
  Moreover, for each $\Phi \in S_{(T)}$ there is a $g \in G \subset \hat H$ such that $\rho(g)\Phi(q) \in S_T$.
  This implies that $\hat Q^\circ$ is a principal $N_{\hat H}(T)$--bundle.

  The isomorphism $\hat Q^\circ \times_{N_{\hat H}(T)} \hat H \iso \hat Q$ is given by $[(\hat q,\hat h)] \mapsto \hat q\cdot\hat h$.
  In particular,
  \begin{equation*}
    \sG(P) \iso \Gamma(\hat Q^\circ \times_{N_{\hat H}(T)} G)
  \end{equation*}
  where $N_{\hat H}(T)$ acts on $G$ by conjugation.
  The last two assertions follow from the fact that, for every $q \in \hat Q^\circ$, the $G$--stabilizer of $\Phi(q)$ is $T$.
\end{proof}

\begin{definition}
  Given any $\Phi \in \Gamma(\bS_{(T)})$, the \defined{Weyl group bundle} associated with $\Phi$ is
  \begin{equation*}
    \hat Q^\diamond = \hat Q^\diamond_\Phi \coloneq \hat Q^\circ_\Phi/T.
  \end{equation*}
\end{definition}

\begin{prop}
  \label{Prop_ProjectionDeterminesWeylGroupBundle}
  Suppose that two choices of geometric data have been made such that $\hat R_1 = \hat R_2$.
  Suppose that $\Phi_i \in \Gamma(\bS_{\hat Q_i,(T)})$ satisfy $\mu(\Phi_i) = 0$.
  Denote by $\hat Q^\diamond_i$ the associated Weyl group bundles.

  If $p\circ\Phi_1 = p\circ \Phi_2 \in \Gamma(\bX_{(T)})$,
  then there is an isomorphism $\hat Q^\diamond_1 \iso \hat Q^\diamond_2$ compatible with the isomorphism 
  \begin{equation*}
    \hat Q^\diamond_1/W_G(T) \iso \hat R_1 = \hat R_2 \iso \hat Q^\diamond_2/W_G(T).
  \end{equation*}
\end{prop}

\begin{remark}
  \label{Rmk_HatQCircNeedNotBeIsomorphic}
  The principal $N_G(T)$--bundles $\hat Q^\circ_1$ and $\hat Q^\circ_2$ need not be isomorphic.
\end{remark}

\begin{proof}[Proof of \autoref{Prop_ProjectionDeterminesWeylGroupBundle}]
  Since $\hat Q_i/G \iso \hat R_i$, we have $\hat Q^\circ_i/N_G(T) \iso \hat R_i$.
  The sections $\Phi_i$ restrict to $N_G(T)$--equivariant maps $\Phi^\circ_i \co \hat Q^\circ_i \to \mu^{-1}(0)\cap S_T$, which in turn induce $W_G(T)$--equivariant maps $\Phi^\diamond_i\co \hat Q^\diamond_i = \hat Q^\circ_i/T \to \mu^{-1}(0)\cap S_T$.
  The resulting commutative diagrams
  \begin{equation*}
    \begin{tikzcd}
      \hat Q^\diamond_i \ar[d,"q^\diamond_i"] \ar[r,"\Phi^\diamond_i"] & \mu^{-1}(0)\cap S_T \ar[d,"p"] \\
      \hat R_i \ar[r,"s"] & X_{(T)}
    \end{tikzcd}
  \end{equation*}
  are pullback diagrams;
  hence, the assertion follows from the universal property of pullbacks.
\end{proof}

\begin{prop}
  \label{Prop_ExistenceOfWeylGroupLifts}
  Let $\hat R$ be a principal $\hat K$--bundle.
  Given $s \in \Gamma(\bX_{(T)})$,
  there exists a principal $W_{\hat H}(T)$--bundle $\hat Q^\diamond$ together with an isomorphism
  \begin{equation*}
    \hat Q^\diamond/W_G(T) \iso \hat R
  \end{equation*}
  and a section
  \begin{equation*}
    \Phi^\diamond \in \Gamma(\hat Q^\diamond\times_{W_{\hat H}(T)} S_T)
  \end{equation*}
  satisfying
  \begin{equation*}
    \mu(\Phi^\diamond) = 0
    \qandq
    p \circ \Phi^\diamond = s.
  \end{equation*}
  The section $\Phi^\diamond$ is unique up to the action of the restricted gauge group $\Gamma(\hat Q^\diamond\times_{W_{\hat H}(T)} W_G(T))$.
\end{prop}

\begin{proof}
  We can think of the section $s$ as a $\hat K$--equivariant map $s \co \hat R \to X_{(T)}$.
  The quotient map $p \colon \mu^{-1}(0) \cap S_T \to X_{(T)}$ defines a principal $W_G(T)$--bundle.
  Set
  \begin{align*}
    \hat Q^\diamond
    \coloneq{}&
      s^*(\mu^{-1}(0)\cap S_T) \\
    ={}&
      \set{ (r,\Phi) \in R \times (\mu^{-1}(0)\cap S_T) : s(r) = W_G(T)\cdot\Phi }
  \end{align*}
  and denote by $\Phi^\diamond\co \hat Q^\diamond \to \mu^{-1}(0)\cap S_T$ the projection to the second factor.
  The projection to the first factor $q^\diamond \co \hat Q^\diamond \to R$ makes $\hat Q^\diamond$ into a principal $W_G(T)$--bundle over $\hat R$.
  We have the following diagram with the square being a pullback:
  \begin{equation*}
    \begin{tikzcd}
      \hat Q^\diamond \ar[d,"q^\diamond"] \ar[r,"\Phi^\diamond"] & \mu^{-1}(0)\cap S_T \ar[d,"p"] \\
      \hat R \ar[d] \ar[r,"s"] & X_{(T)} & \\
      M.
    \end{tikzcd}
  \end{equation*}

  $\hat Q^\diamond$ can be given the structure of a principal $W_{\hat H}(T)$--bundle over $M$ as follows.
  By \autoref{Prop_NHT/NGT=K} we have a short exact sequence
  \begin{equation*}
    \begin{tikzcd}
    0 \ar[r] & W_G(T) \ar[r] & W_{\hat H}(T) \ar[r,"\pi"] & \hat K \ar[r] & 0.
    \end{tikzcd}
  \end{equation*}
  Define an right-action of $W_{\hat H}(T)$ on $\hat Q^\diamond$ by 
  \begin{equation*}
    (r, \Phi)\cdot[\hat h]
    \coloneq
      (r\cdot \pi([\hat h]), (\theta\times\rho)(\hat h^{-1})\Phi)
  \end{equation*}
  for $[\hat h] \in W_{\hat H}(T)$ and $(r,\Phi) \in \hat Q^\diamond$ and with $\theta$ as in \autoref{Def_CanonicalPermutingAction}.
  A moment's thought shows that this action is free and
  \begin{equation*}
    \hat Q^\diamond/W_{\hat H}(T) = (\hat Q^\diamond/W_G(T))/\hat K = \hat R/\hat K = M.
  \end{equation*}

  Since $s$ is $\hat K$--equivariant, $\Phi^\diamond$ is $W_{\hat H}(T)$--equivariant and thus defines the desired section.
  The assertion about the uniqueness of $\Phi^\diamond$ is clear.
\end{proof}

\begin{prop}
  \label{Prop_ExistenceOfLifts}
  Assume the situation of \autoref{Prop_ExistenceOfWeylGroupLifts}.
  Suppose that $\hat Q^\circ$ is a principal $N_{\hat H}(T)$--bundle with an isomorphism
  \begin{equation*}
    \hat Q^\circ\!/T \iso \hat Q^\diamond;
  \end{equation*}
  that is:
  $\hat Q^\circ$ is a lift of the structure group from $W_{\hat H}(T)$ to $N_{\hat H}(T)$.
  Set
  \begin{equation*}
    \hat Q \coloneq \hat Q^\circ\times_{N_{\hat H}(T)}\hat H.
  \end{equation*}
  In this situation, there is a section $\Phi$ of $\bS_{(T)} \coloneq \hat Q \times_{\hat H} S_{(T)}$ satisfying
  \begin{equation*}
    \mu(\Phi) = 0
    \qandq
    p \circ \Phi = s;
  \end{equation*}
  moreover, there is an isomorphism
  \begin{equation*}
    \hat Q^\circ_\Phi \iso \hat Q^\circ.
  \end{equation*}
  Any other section satisfying these conditions is related to $\Phi$ by the action of $\sG(P)$.
\end{prop}

\begin{proof}
  With $\Phi^\diamond$ as in \autoref{Prop_ExistenceOfWeylGroupLifts} define $\Phi \co \hat Q \to \mu^{-1}(0) \subset S$ by 
  \begin{equation*}
    \Phi([q,\hat h])
    \coloneq
      (\theta\times\rho)(\hat h^{-1})\Phi^\diamond(qT).
  \end{equation*}
  This is well-defined because $\Phi^\diamond(qT)$ is $T$--invariant;
  moreover, $\Phi$ is manifestly $\hat H$--equivariant and, hence, defines the desired section.
  The assertion about the uniqueness of $\Phi$ is clear.
\end{proof}

To summarize the preceeding discussion and answer the questions raised at the beginning of this section:
\begin{enumerate}
\item
  $s$ determines the Weyl group bundle $\hat Q^\diamond$ uniquely,
\item
  every $s$ lifts to a section $\Phi^\diamond$ of $\hat Q^\diamond\times_{W_{\hat H}(T)} S_T$, and
\item
  if $\hat Q^\circ$ is a lift of the structure group of $\hat Q^\diamond$ from $W_{\hat H}(T)$ to $N_{\hat H}(T)$ and we set
  $\hat Q \coloneq \hat Q^\circ\times_{N_{\hat H}(T)} \hat H$,
  then $\Phi^\diamond$ induces a section $\Phi$ of $\bS_{(T)} = \hat Q \times_{\hat H} S_{(T)}$ lifting $s$.
\end{enumerate}

\subsection{Projecting the Dirac equation}
\label{Sec_ProjectingDiracEquation}

The second part of the Haydys correspondence is concerned with the question
\begin{quote}
  \centering
  To what extend is the Dirac equation for a section $\Phi \in \Gamma(\bS_{(T)})$ equivalent to a differential equation for $s \coloneq p\circ \Phi \in \Gamma(\bX_{(T)})$?
\end{quote}

\begin{definition}
  The \defined{vertical tangent bundle} of $\bX_{(T)} \xrightarrow{\pi} M$ is
  \begin{equation*}
    V\bX_{(T)} \coloneq \hat R \times_{\hat K} TX_{(T)}.
  \end{equation*}
  The hyperkähler structure on $X_{(T)}$ induces a \defined{Clifford multiplication}
  \begin{equation*}
    \gamma \co \pi^*\Im \H \to \End(V\bX_{(T)}).
  \end{equation*}

  Given $B \in \sA(\hat R)$ we can assign to each $s \in \Gamma(\bS)$ its covariant derivative $\nabla_B s \in \Omega^1(M,s^*V\bX)$.
  A section $s \in \Gamma(\bX)$ is called a \defined{Fueter section} if it satisfies the \defined{Fueter equation}
  \begin{equation}
    \label{Eq_Fueter}
    \fF(s) = \fF_B(s) \coloneq \gamma(\nabla_B s) = 0 \in \Gamma(s^*V\bX_{(T)}).
  \end{equation}
  The map $s \mapsto \fF(s)$ is called the \defined{Fueter operator}.
\end{definition}

\begin{prop}
  \label{Prop_ProjectingDiracEquation}
  Given $\Phi \in \Gamma(\bS_{(T)})$ satisfying $\mu(\Phi) = 0$, set
  \begin{equation*}
    s \coloneq p\circ\Phi \in \Gamma(\bX_{(T)}).
  \end{equation*}
  The following hold:
  \begin{enumerate}
  \item
    \label{Prop_P_Dirac<=>Fueter+Perpendicular}
    $A \in \sA_B(\hat Q)$ satisfies $\slD_A\Phi = 0$ if and only if
    \begin{equation}
      \label{Eq_Fueter+NablaPhiPerpGauge}
      \fF_B(s) = 0 \qandq
      \nabla_A\Phi \perp \rho(\fg_P)\Phi.
    \end{equation}
  \item
    \label{Prop_P_ABPhi}
    Let $\ft_P$ be as in \eqref{Eq_ftP}.
    The space of connections
    \begin{equation}
      \label{Eq_APhiB}
      \sA_B^\Phi(\hat Q)
      \coloneq
      \set*{ 
        A \in \sA_B(\hat Q)
        :
        \nabla_A\Phi \perp \rho(\fg_P)\Phi
      }
    \end{equation}
    is an affine space modeled on $\Omega^1(M,\ft_P)$ with $\ft_P$ as in \eqref{Eq_ftP}.
    In particular, if $\fF_B(s) = 0$, there exists an $A \in \sA_B(\hat Q)$ such that $\slD_A\Phi = 0$;
    $A$ is unique up to $\Omega^1(M,\ft_P)$.
  \item
    \label{Prop_P_AReducesToHatQCirc}
    Any connection $A \in \sA_B^\Phi(\hat Q)$ reduces to a connection on $\hat Q^\circ$.
    Conversely, any connection on $\hat Q^\circ$ induces a connection in $\sA_B^\Phi(\hat Q)$.
  \item
    \label{Prop_P_TPParallel}
    The subbundle $\ft_P \subset \fg_P$ is parallel with respect to any $A \in \sA_B^\Phi(\hat Q)$.
  \end{enumerate}
\end{prop}

\begin{proof}
  We prove \itref{Prop_P_Dirac<=>Fueter+Perpendicular}.
  If $\slD_A\Phi = 0$, then it follows from $p_*(\nabla_A\Phi) = \nabla_Bs$ that $\fF_B(s) = 0$.
  Let $(e^1,e^2,e^3)$ be an orthonormal basis of $T_x^*M$.
  The equations $\slD_A\Phi = 0$ and $\nabla_A\mu(\Phi) = 0$ can be written as
  \begin{equation*}
    \nabla_{A,e_i}\Phi = -
    \epsilon_{ij}^{~~k} \gamma(e^j)\nabla_{A,e_k} \Phi
    \qandq
    \inner{\gamma(e^j)\nabla_{A,e_k}\Phi}{\rho(\xi)\Phi} = 0
  \end{equation*}
  for all $\xi \in \fg_{P,x}$.
  This proves that $\nabla_A\Phi \perp \rho(\fg_P) \Phi$.
  By \autoref{Thm_HyperkahlerQuotientWithStabilizer}\itref{Thm_HyperkahlerQuotientWithStabilizer_X(T)}, \eqref{Eq_Fueter+NablaPhiPerpGauge} implies $\slD_A\Phi = 0$  

  We prove \itref{Prop_P_ABPhi}.
  If $A \in \sA_B^\Phi(\hat Q)$ and $a \in \Omega^1(M,\fg_P)$ are such that $A + a \in \sA_B^\Phi(\hat Q)$, then
  \begin{equation*}
    \rho(a)\Phi \perp \rho(\fg_P)\Phi;
  \end{equation*}
  hence, $\rho(a)\Phi = 0$ and it follows that $a \in \Omega^1(M,\ft_P)$ by \autoref{Prop_HatQCirc}.
  It remains to show that $\sA_B^\Phi(\hat Q)$ is non-empty.
  To see this, note that if $A \in \sA_B^\Phi(\hat Q)$, then one can find $a \in \Omega^1(M,\fg_P)$ such that $\nabla_A\Phi + \rho(a)\Phi$ is perpendicular to $\rho(\fg_P)\Phi$.

  We prove \itref{Prop_P_AReducesToHatQCirc}.
  If $A \in \sA_B^\Phi(\hat Q)$ and $H_A$ denote its horizontal distribution,
  then we need to show that for $q \in \hat Q^\circ$ we have $H_{A,q} \subset T_q\hat Q^\circ$.
  This, however, is an immediate consequence of the definitions of $\sA_B^\Phi(\hat Q)$ and $\hat Q^\circ$.

  We prove \itref{Prop_P_TPParallel}.
  Suppose $\tau \in \Gamma(\ft_P)$, that is, $\rho(\tau)\Phi = 0$.
  Differentiating this identity along $v$ yields
  \begin{equation*}
    \rho(\nabla_{A,v}\tau)\Phi = -\rho(\tau) \nabla_{A,v}\Phi;
  \end{equation*}
  Set $\sigma = \nabla_{A,v}\tau$.
  We need to show that $\rho(\sigma)\Phi = 0$.
  We compute
  \begin{align*}
    \abs{\rho(\sigma)\Phi}^2
    &=
      -\inner{\rho(\sigma)\Phi}{\rho(\tau)\nabla_{A,v}\Phi} \\
    &=
      \inner{\rho(\tau)\rho(\sigma)\Phi}{\nabla_{A,v}\Phi} \\
    &=
      \inner{\rho([\tau,\sigma])\Phi}{\nabla_{A,v}\Phi} = 0
  \end{align*}
  because $\nabla_A \Phi \perp \rho(\fg_P)\Phi$.
\end{proof}

To summarize:
\begin{enumerate}
\item
  The Dirac equation $\slD_A\Phi = 0$ implies the Fueter equation $\fF_B s = 0$.
\item
  Given a solution $s$ of the Fueter equation and $\hat Q^\circ$ as at the end of the last subsection,
  there is a connection $A \in \sA_B(\hat Q)$ such that the lift $\Phi$ satisfies $\slD_A\Phi = 0$.
\item
  $A$ is unique up to $\Omega^1(M,\ft_P)$ with $\ft_P$ as in \eqref{Eq_ftP}.
\end{enumerate}


\section{The ADHM representation}
\label{Sec_NakajimasProof}

We now focus on the case $r=1$ in \autoref{Ex_ADHMAlgebraicData}. 
We will see that in this case the hyperk\"ahler quotient of the representation is the symmetric product $\Sym^k \H$. 
This fact is the basis of the relationship between multiple covers of associatives and ADHM monopoles.

Identifying $\H\otimes_\C\C^r = \Hom_\C(\C^k,\H)$, we can write the quaternionic vector space $S$ from \autoref{Ex_ADHMAlgebraicData} with $r = 1$ as
\begin{equation*}
  S = \Hom_\C(\C^k,\H) \oplus \H\otimes_\R\fu(k).
\end{equation*}
The group $\U(k)$ acts on $S$ via
\begin{equation*}
  \rho(g)(\Psi,\bxi)
  \coloneq
  (\Psi g^{-1},\Ad(g)\bxi)
\end{equation*}
preserving the hyperkähler structure.
We will now determine the hyperkähler quotient $S \hkred \U(k)$ and its decomposition into hyperkähler manifolds described in \autoref{Thm_HyperkahlerQuotientWithStabilizer}.

\begin{definition}
  \label{Def_PartitionGlambdaTlambda}
  A \defined{partition} of $k \in \N$ is a non-increasing sequence of non-negative integers $\lambda = (\lambda_1, \lambda_2, \ldots)$ which sums to $k$.
  The \defined{length} of a partition is
  \begin{equation*}
    \abs{\lambda}
    \coloneq
    \min \set{ n \in \N : \lambda_n = 0 } - 1.
  \end{equation*}
  
  With each partition $\lambda$ we associate the groups
  \begin{equation*}
    G_\lambda
    \coloneq
    \set*{
      \sigma \in S_{\abs{\lambda}}
      :
      \lambda_{\sigma(n)} = \lambda_n \text{ for all } n \in \set {1 \ldots, \abs{\lambda}}
    }
  \end{equation*}
  and
  \begin{equation*}
    T_\lambda
    \coloneq
    \prod_{n = 1}^{\abs{\lambda}} \U(\lambda_n)
    \subset
    \U(k).
  \end{equation*}

  For each partition $\lambda$ of $k$,
  consider the generalized diagonal
  \begin{equation*}
    \Delta_{\abs{\lambda}}
    =
    \set{
      v_1,\ldots,v_{\abs{\lambda}} \in \H^{\abs{\lambda}} : v_i = v_j \text{ for some } i \neq j
    }
  \end{equation*}
  There is an embedding $(\H^{\abs{\lambda}}\setminus \Delta_{\abs{\lambda}})/G_\lambda \into \Sym^k \H$ defined by
  \begin{equation*}
    [v_1,\ldots,v_{\abs{\lambda}}] \mapsto [\underbrace{v_1,\ldots,v_1}_{\lambda_1 \text{ times}},\cdots,\underbrace{v_{\abs{\lambda}},\ldots,v_{\abs{\lambda}}}_{\lambda_{\abs{\lambda}} \text{ times}}].
  \end{equation*}
  The image of this inclusion is denoted by $\Sym^k_\lambda\H$.
\end{definition}

\begin{theorem}[{\citet[Proposition 2.9]{Nakajima1999}}]
  \label{Thm_1kADHM}
  We have
  \begin{equation*}
    S\hkred G
    =
    \bigcup_{\lambda} S_{T_\lambda} \hkred W_{U(k)}(T_\lambda)
    =
    \bigcup_{\lambda} \Sym^k_\lambda \H
    =
    \Sym^k\H.
  \end{equation*}
  Here we take the union over all partitions $\lambda$ of $k$.
\end{theorem}

The proof of \autoref{Thm_1kADHM} occupies the remaining part of this section.
Various algebraic identities derived in the course of proving the theorem are also used in the discussion of the Haydys correspondence for ADHM monopoles.

\begin{prop}
  \label{Prop_1kADHMMomentMap}
  The canonical moment map $\mu \co S \to (\fu(k)\otimes \Im\H)^*$ for the action $\rho \co \U(k) \to \Sp(S)$ is given by
  \begin{equation*}
    \mu(\Psi,\bxi)
    \coloneq
    \mu(\Psi) + \mu(\bxi)
  \end{equation*}
  with%
  \footnote{%
    We identify $(\fu(k)\otimes \Im\H)^* = \fu(k)\otimes \Im\H$.
  }
  \begin{align*}
    \mu(\Psi)
    &\coloneq
      \frac12\big(
      (\Psi^*i\Psi)\otimes i 
      + (\Psi^*j\Psi)\otimes j
      + (\Psi^*k\Psi)\otimes k
      \big) \qand \\
    \mu(\bxi)
    &\coloneq
      ([\bxi_0,\bxi_1] + [\bxi_2,\bxi_3])\otimes i \\
    &\quad
      + ([\bxi_0,\bxi_2] + [\bxi_3,\bxi_1])\otimes j \\
    &\quad
      + ([\bxi_0,\bxi_3] + [\bxi_1,\bxi_2])\otimes k.
  \end{align*}
\end{prop}

\begin{proof}
  We can compute the moment maps for the action of $\U(k)$ on $\Hom(\C^k,\H)$ and $\H\otimes \fu(k)$ separately.
  If $v = v_1i+v_2j+v_3k \in \Im\H$ and $\eta \in \fu(k)$, then
  \begin{align*}
    2\inner{\mu(\Psi)}{v\otimes \eta}
    =
    \inner{\Psi}{\gamma(v)\rho(\eta)\Psi} 
    =
    -\inner{\Psi}{\gamma(v)\Psi\circ \eta} 
    =
    \inner{\Psi^*\gamma(v)\Psi}{\eta}
  \end{align*}
  and
  \begin{align*}
    2\inner{\mu(\bxi)}{v\otimes \eta}
    &=
      \inner{\bxi}{\gamma(v)\rho(\eta)\bxi} \\
    &=
      v_1\(-\inner{\bxi_0}{[\eta,\bxi_1]} + \inner{\bxi_1}{[\eta,\bxi_0]} - \inner{\bxi_2}{[\eta,\bxi_3]} + \inner{\bxi_3}{[\eta,\bxi_2]}\) \\
    &\quad
      +v_2\(-\inner{\bxi_0}{[\eta,\bxi_2]} + \inner{\bxi_1}{[\eta,\bxi_3]} + \inner{\bxi_2}{[\eta,\bxi_0]} - \inner{\bxi_3}{[\eta,\bxi_2]}\) \\
    &\quad
      +v_3\(-\inner{\bxi_0}{[\eta,\bxi_3]} - \inner{\bxi_1}{[\eta,\bxi_2]} - \inner{\bxi_2}{[\eta,\bxi_1]} + \inner{\bxi_3}{[\eta,\bxi_0]}\) \\
    &=
      2v_1\inner{[\bxi_0,\bxi_1] + [\bxi_2,\bxi_3]}{\eta} \\
    &\quad
      + 2v_2\inner{[\bxi_0,\bxi_2] + [\bxi_3,\bxi_1]}{\eta} \\
    &\quad
      + 2v_3\inner{[\bxi_0,\bxi_3] + [\bxi_1,\bxi_2]}{\eta}
  \end{align*}
  using that $\inner{\xi}{[\eta,\zeta]} = -\inner{[\xi,\zeta]}{\eta}$ for  $\xi,\eta,\zeta \in \fu(k)$.
\end{proof}

The key to proving \autoref{Thm_1kADHM} is the following result.

\begin{prop}
  \label{Prop_1kADHMMu=0=>Phi=0}
  If $\mu(\Psi,\bxi) = 0$, then $\Psi = 0$.
\end{prop}

One can derive this result using Geometric Invariant Theory \cite[Section 2.2]{Nakajima1999}.
We provide a proof at the end of this section.
It essentially follows \citeauthor{Nakajima1999}'s reasoning but avoids the use of GIT and comparison results between GIT and Kähler quotients. 

It follows from \autoref{Prop_1kADHMMu=0=>Phi=0} that
\begin{equation*}
  S \hkred \U(k) = \H\otimes \fg \hkred \U(k).
\end{equation*}
The latter can be computed in a straight-forward fashion using the following observation.

\begin{prop}
  \label{Prop_QuaternionifiedAdjointActionMomentMapSquared}
  We have
  \begin{equation*}
    \abs{\mu(\bxi)}^2
    =
    \frac12\sum_{\alpha,\beta=0}^3 \abs{[{\bxi_\alpha},{\bxi_\beta}]}^2.
  \end{equation*}
\end{prop}

\begin{proof}
  A direct computation shows that
  \begin{equation*}
    \abs{\mu(\bxi)}^2 - \frac12\sum_{\alpha,\beta=0}^3 \abs{[\bxi_\alpha,\bxi_\beta]}^2
    =
    -2 \inner{\bxi_0}{[\bxi_1,[\bxi_2,\bxi_3]]+[\bxi_2,[\bxi_3,\bxi_1]]+[\bxi_3,[\bxi_1,\bxi_2]]}.
  \end{equation*}
  This expression vanishes by the Jacobi Identity.
\end{proof}

\begin{proof}[Proof of \autoref{Thm_1kADHM}]
  From  \autoref{Prop_1kADHMMu=0=>Phi=0} and \autoref{Prop_QuaternionifiedAdjointActionMomentMapSquared} it follows that  we have $\mu(\Psi,\bxi) = 0$ if and only if $\Psi = 0$ and $\bxi \in \H\otimes \ft$ for some maximal torus $\ft \subset \fu(k)$.
  Therefore, for a fixed maximal torus $T \subset \U(k)$ and $\ft \coloneq \Lie(T$),
  \begin{equation*}
    S \hkred G = (\H\otimes \ft)/W_{\U(k)}(T) \iso \H^k/S_k = \Sym^k \H,
  \end{equation*}
  using that the Weyl group of $\U(k)$ is the permutation group $S_k$.
  
  The map $S \hkred G \to \Sym^k\H$ can be described more directly as the joint spectrum.
  Since $\mu(\bxi) = 0$ implies $[\bxi,\bxi] = 0 \in \Lambda^2\H\otimes \fg$,
  we can find a basis $e_1,\ldots,e_k$ of $\C^k$ and elements $v_1,\ldots,v_k \in \H$ such that
  \begin{equation*}
    \bxi(e_i) = v_i\otimes e_i.
  \end{equation*}
  Up to ordering, the $v_i$ are independent of the choice of basis $e_i$.
  The isomorphism $S\hkred G \to \Sym^k \H$ is the map
  \begin{equation*}
    \bxi \mapsto \spec(\bxi) \coloneq \set{ v_1, \ldots, v_k }.
  \end{equation*}
  
  From this description the decomposition of $\Sym^k\H$ into its strata $\Sym^k_\lambda \H$ is clear.
\end{proof}

The following result, which can be viewed as the linearization of \autoref{Prop_QuaternionifiedAdjointActionMomentMapSquared}, plays an important role in \autoref{Sec_FormalExpansion}.

\begin{prop}
  \label{Prop_LinearizedMomentMapAtZero}
  Denote by $R_\bxi \co \fu(k) \to \H\otimes\fu(k)$ the linearization of the action of $\U(k)$ on $\H\otimes\fu(k)$ at $\bxi$ and by $R_\bxi^* \co \H\otimes\fu(k) \to \fu(k)$ its adjoint.
  If $\mu(\bxi) = 0$, then
  \begin{equation*}
    \abs{(\rd_\bxi\mu)\bfeta}^2
    + \frac12\abs{R_\bxi^*\bfeta}^2
    =
    \sum_{\alpha\neq\beta=0}^3 \abs{[\bxi_\alpha,\bfeta_\beta]}^2
    + \frac12\sum_{\alpha=0}^3 \abs{[\bxi_\alpha,\bfeta_\alpha]}^2.
  \end{equation*}
\end{prop}

\begin{proof}
  If $\mu(\bxi) = 0$, then on the one hand
  \begin{equation*}
    \abs{\mu(\bxi+t\bfeta)}^2
    = t^2\abs{(\rd_\bxi\mu) \bfeta}^2 + O(t^3);
  \end{equation*}
  while on the other hand
  \begin{align*}
    \abs{\mu(\bxi+t\bfeta)}^2
    &=
      \frac12\sum_{\alpha,\beta=0}^3 \abs{[{\bxi_\alpha+t\bfeta_\alpha},{\bxi_\beta+t\bfeta_\beta}]}^2 \\
    &=
      \frac12\sum_{\alpha,\beta=0}^3
      \abs{[\bxi_\alpha,t\bfeta_\beta] + [t\bfeta_\alpha,\bxi_\beta]}^2
      + O(t^3) \\
    &=
      t^2\sum_{\alpha\neq\beta=0}^3
      \abs{[\bxi_\alpha,\bfeta_\beta]}^2
      + \inner{[\bxi_\alpha,\bfeta_\beta]}{[\bfeta_\alpha,\bxi_\beta]}
      + O(t^3).
  \end{align*}
  We also have
  \begin{align*}
    \abs{R_\bxi^*\bfeta}^2 
    &=
      \abs*{\sum_{\alpha=0}^3 [\bxi_\alpha,\bfeta_\alpha]}^2 \\
    &=
      2\sum_{\alpha\neq \beta = 0}^3 \inner{[\bxi_\alpha,\bfeta_\alpha]}{[\bxi_\beta,\bfeta_\beta]}
      + \sum_{\alpha=0}^3 \abs{[\bxi_\alpha,\bfeta_\alpha]}^2.
  \end{align*}
  By the Jacobi identity
  \begin{align*}
    \inner{[\bxi_\alpha,\bfeta_\beta]}{[\bfeta_\alpha,\bxi_\beta]}
    &=
      - \inner{\bfeta_\beta}{[\bxi_\alpha,[\bfeta_\alpha,\bxi_\beta]]} \\
    &=
      \inner{\bfeta_\beta}{[\bfeta_\alpha,[\bxi_\beta,\bxi_\alpha]]}
      + \inner{\bfeta_\beta}{[\bxi_\beta,[\bxi_\alpha,\bfeta_\alpha]]} \\
    &=
      \inner{\bfeta_\beta}{[\bxi_\beta,[\bxi_\alpha,\bfeta_\alpha]]} \\
    &=
      -\inner{[\bxi_\beta,\bfeta_\beta]}{[\bxi_\alpha,\bfeta_\alpha]}.
  \end{align*}
  Putting everything together, yields the asserted identity.
\end{proof}

\begin{proof}[Proof of \autoref{Prop_1kADHMMu=0=>Phi=0}]
  For the proof it is convenient to write $S$ as $\C^k\oplus j\C^k \oplus \End(\C^k) \oplus j\End(\C^k)$.
  A direct computation shows that with respect to this identification the moment map is given by
  \begin{equation}
    \label{Eq_ComplexMomentMap}
    \mu(v,w,A^*,B)
    = 
    \frac{1}{2}(vv^* - ww^* - [A,A^*] - [B,B^*])
    + j(wv^* - [A,B]).
  \end{equation}
  Therefore, if $(\Psi,\bxi) = (v + jw,A^* + jB) \in \mu^{-1}(0)$, then
  \begin{equation}
    \label{Eq_MuvwAB=0}
    vv^* - ww^* = [A,A^*] + [B,B^*] \qandq
    wv^* = [A,B].
  \end{equation}

  Set $T \coloneq [A,A^*] + [B,B^*]$.
  Taking traces and inner products with $v$ and $w$, \eqref{Eq_MuvwAB=0} implies
  \begin{gather}
    \label{Eq_VWOrthonormal}
    \abs{v} = \abs{w} \eqcolon \lambda, \quad
    \inner{v}{w} = 0, \\
    \label{Eq_TvvTww}
    \inner{Tv}{v} = \lambda^4, \qandq
    \inner{Tw}{w} = -\lambda^4.
  \end{gather}

  \begin{prop}[{\cite[Lemma 2.8]{Nakajima1999}}]
    \label{Prop_VperpV1}
    Denote by $V_1$ the smallest subspace of $\C^k$ which contains $w$ and is preserved by both $A$ and $B$.
    We have $v \perp V_1$.
  \end{prop}

  \begin{proof}
    Let $C$ be a product of $A$s and $B$s.
    We need to show that $\inner{v}{Cw} = 0$.
    The proof is by induction on $k$, the number of factors of $C$.
    If $k = 0$, then $C = \id$ and we have $\inner{v}{w} = 0$ by \eqref{Eq_VWOrthonormal}.
    
    By induction we can assume that $\inner{v}{\tilde Cw} = 0$ for all $\tilde C$ with fewer than $k$ factors.
    If $C = C_lBAC_r$, then
    \begin{align*}
      Cw
      &=
        C_lBAC_rw
        =
        C_lABC_rw - C_l[A,B]C_rw \\
      &=
        C_lABC_rw - C_lwv^*C_rw
        =
        C_lABC_rw
    \end{align*}
    because $v^*C_rw = \inner{v}{C_rw}$ and $C_r$ has fewer than $k$ factors.    
    Henceforth, we can assume that $C = A^{k_1}B^{k_2}$.
    For such $C$, we have
    \begin{align*}
      \inner{v}{A^{k_1}B^{k_2}w}
      &=
        \tr(A^{k_1}B^{k_2}wv^*)
        =
        \tr(A^{k_1}B^{k_2}[A,B]) \\
      &=
        \tr([A^{k_1}B^{k_2},A]B)
        =
        \tr(A^{k_1}[B^{k_2},A]B) \\
      &=
        \sum_{\ell=0}^{k_2-1} \tr(A^{k_1}B^{\ell}[B,A]B^{k_2-\ell})
        =
        \sum_{\ell=0}^{k_2-1} \tr(B^{k_2-\ell}A^{k_1}B^{\ell}[B,A]) \\
      &=
        -\sum_{\ell=0}^{k_2-1} \inner{v}{B^{k_2-\ell}A^{k_1}B^{\ell}w}
        =
        -k_2 \inner{v}{A^{k_1}B^{k_2}w}.
    \end{align*}
    This concludes the proof.
  \end{proof}

  As a warm up consider the case $k = 2$.
  If $\lambda > 0$, then
  \begin{equation*}
    (w/\lambda,v/\lambda)
  \end{equation*}
  is an orthonormal basis for $\C^2$.
  With respect to this basis $A$ and $B$ are given by matrices of the form
  \begin{equation*}
    A =
    \begin{pmatrix}
      a_{11} & a_{12} \\
      0 & a_{22} 
    \end{pmatrix}
    \qandq
    B =
    \begin{pmatrix}
      b_{11} & b_{12} \\
      0 & b_{22} 
    \end{pmatrix}.
  \end{equation*}
  Consequently, the first diagonal entry of $T = [A,A^*] + [B,B^*]$ is
  \begin{equation*}
    T_{11} = \abs{a_{12}}^2 + \abs{b_{12}}^2 > 0.
  \end{equation*}
  However, since $\inner{Tw}{w} = -\lambda^4$ according to \eqref{Eq_TvvTww}, we have
  \begin{equation*}
    T_{11} = -\lambda^2 < 0.
  \end{equation*}
  It follows that $\lambda = 0$;
  that is, $\Psi = v + jw = 0$.

  In general, let $V_1$ be as in \autoref{Prop_VperpV1} and set $V_2 \coloneq V_1^\perp$.
  With respect to the splitting $\C^k = V_1\oplus V_2$, we have
  \begin{equation*}
    A =
    \begin{pmatrix}
      A_{11} & A_{12} \\
      0 & A_{22}
    \end{pmatrix}
    \qandq
    B =
    \begin{pmatrix}
      B_{11} & B_{12} \\
      0 & B_{22}
    \end{pmatrix}.
  \end{equation*}
  It follows from $wv^* = [A,B]$ and $v \in V_2$, that
  \begin{equation*}
    [A_{11},B_{11}]
    =
    [A,B]_{11}
    =
    0;
  \end{equation*}
  Moreover, we have
  \begin{equation*}
    T_{11}
    =
    ([A,A^*] + [B,B^*])_{11}
    =
    [A_{11},A_{11}^*] + [B_{11},B_{11}^*] + A_{12}A_{12}^* + B_{12}B_{12}^*;
  \end{equation*}
  hence,
  \begin{equation*}
    [A_{11},A_{11}^*] + [B_{11},B_{11}^*] + A_{12}A_{12}^* + B_{12}B_{12}^* + ww^* = 0.
  \end{equation*}
  Thus $[A_{11},A_{11}^*] + [B_{11},B_{11}^*] \leq 0$.
  By \autoref{Prop_SimDiag}, it follows that $[A_{11},A_{11}^*] = [B_{11},B_{11}^*] = 0$.
  Since $A_{12}A_{12}^* + B_{12}B_{12}^* + ww^*$ is a sum of non-negative definite matrices,
  we must have $\abs{w} = 0$;
  hence, $\Psi = v + jw = 0$ by \eqref{Eq_VWOrthonormal}.

  \begin{prop}
    \label{Prop_SimDiag}
    If $[A,B] = 0$ and $[A,A^*] + [B,B^*] \leq 0$, then $A$ and $B$ can be simultaneously diagonalized and $[A,A^*] = [B,B^*] = 0$.
  \end{prop}

  \begin{proof}
    Since $A$ and $B$ commute, we can simultaneously upper triagonalize them;
    that is, after conjugating $A$ and $B$  with a unitary matrix we can assume that
    \begin{equation*}
      A = \Lambda + U \qandq
      B = M + V
    \end{equation*}
    where $\Lambda, M$ are diagonal and $U,V$ are strictly upper triangular.
    We have
    \begin{equation*}
      [A,A^*]
      = [\Lambda,\Lambda^*] + [\Lambda,U^*] - [\Lambda^*,U] + [U,U^*].
    \end{equation*}
    The first term vanishes, and the second and third terms have vanishing diagonal entries.
    Writing $U = (u_{mn})$, the $m$--th diagonal of $[A,A^*]$ is
    \begin{equation*}
      \sum_{n=1}^k \abs{u_{mn}}^2 - \abs{u_{nm}}^2;
    \end{equation*}
    and similarly for $B$ with $V = (v_{mn})$.
    
    The first diagonal entry of $[A,A^*]+[B,B^*]$ is
    \begin{equation*}
      \sum_{n=1}^k \abs{u_{1n}}^2 + \abs{v_{1n}}^2.
    \end{equation*}
    Being non-positive, this term vanishes.
    The second diagonal entry is
    \begin{equation*}
      \sum_{n=1}^k \abs{u_{2n}}^2 + \abs{v_{2n}}^2 - \abs{u_{12}}^2 - \abs{v_{12}}^2 = \sum_{n=1}^k \abs{u_{2n}}^2 + \abs{v_{2n}}^2
    \end{equation*}
    Being non-positive, this term vanishes as well.
    Repeating this argument eventually shows that $U = V = 0$.
  \end{proof}
  This completes the proof of \autoref{Prop_1kADHMMu=0=>Phi=0}.
\end{proof}


\printreferences

\end{document}
